\documentclass[12pt]{amsart}
\usepackage{fullpage}
\usepackage{amsfonts,amscd}
\usepackage{amssymb}
\usepackage{url}
\usepackage{graphicx}
\usepackage[english]{babel}
\theoremstyle{plain}
\newtheorem{theorem}                {Theorem}      [section]
\newtheorem{proposition}  [theorem]  {Proposition}
\newtheorem{corollary}    [theorem]  {Corollary}
\newtheorem{lemma}        [theorem]  {Lemma}

\theoremstyle{definition}
\newtheorem{example}      [theorem]  {Example}
\newtheorem{remark}       [theorem]  {Remark}
\newtheorem{definition}   [theorem]  {Definition}

\newcommand{\dv}{\text{ }dV}

\numberwithin{equation}{section}

\def \T{{\mathbb T}}
\def \R{{\mathbb R}}
\def \s{{\mathbb S}}

\def \n{{\mathbb N}}
\def \H{{\mathbb H}}

\usepackage{color}

\def \link {~}
\def \1 {\`}

\DeclareMathOperator{\grad}{grad}
\DeclareMathOperator{\trace}{trace}
\numberwithin{equation}{section}

\title[]{Higher order energy functionals}

\author{V.~Branding}
\address{University of Vienna\\
Department of Mathematics\\
Oskar-Morgenstern-Platz 1, Vienna, 1090, Austria}
\email{volker.branding@univie.ac.at}

\author{S.~Montaldo}
\address{Universit\`a degli Studi di Cagliari\\
Dipartimento di Matematica e Informatica\\
Via Ospedale 72\\
09124 Cagliari, Italia}
\email{montaldo@unica.it}

\author{C.~Oniciuc}
\address{Faculty of Mathematics\\ ``Al.I. Cuza'' University of Iasi\\
Bd. Carol I no. 11 \\
700506 Iasi, ROMANIA}
\email{oniciucc@uaic.ro}

\author{A.~Ratto}
\address{Universit\`a degli Studi di Cagliari\\
Dipartimento di Matematica e Informatica\\
Viale Merello 93\\
09123 Cagliari, Italia}
\email{rattoa@unica.it}
\usepackage{hyperref}
\begin{document}
\begin{abstract}
The study of higher order energy functionals was first proposed by Eells and Sampson in 1965 and, later, by Eells and Lemaire in 1983. These functionals provide a natural generalization of the classical energy functional. More precisely, Eells and Sampson suggested the investigation of the so-called $ES-r$-energy functionals $ E_r^{ES}(\varphi)=(1/2)\int_{M}\,|(d^*+d)^r (\varphi)|^2\,dV$, where $ \varphi:M \to N$ is a map between two Riemannian manifolds. In the initial part of this paper we shall clarify some relevant issues about the definition of an $ES-r$-harmonic map, i.e, a critical point of $ E_r^{ES}(\varphi)$. That seems important to us because in the literature other higher order energy functionals have been studied by several authors and consequently some recent examples need to be discussed and extended: this shall be done in the first two sections of this work, where we obtain the first examples of proper critical points of $E_r^{ES}(\varphi)$ when $N=\s^m$ $(r \geq4,\, m\geq3)$, and we also prove some general facts which should be useful for future developments of this subject. Next, we shall compute the Euler-Lagrange system of equations for $E_r^{ES}(\varphi)$ for $r=4$. We shall apply this result to the study of maps into space forms and to rotationally symmetric maps: in particular, we shall focus on the study of various family of conformal maps. In Section~\ref{section-Models}, we shall also show that, even if $2 r > \dim M$, the functionals $ E_r^{ES}(\varphi)$ may not satisfy the classical Palais-Smale Condition (C). In the final part of the paper we shall study the second variation and compute index and nullity of some significant examples.
\end{abstract}



\thanks{The first author was supported by Austrian Science Fund (FWF) (Projekt P30749-N35); the second and the last authors  were supported by Fondazione di Sardegna (project STAGE) and Regione Autonoma della Sardegna (Project KASBA); the third author was supported by a project funded by the Ministry of
Research and Innovation within Program 1 - Development of the national RD
system, Subprogram 1.2 - Institutional Performance - RDI excellence
funding projects, Contract no. 34PFE/19.10.2018.}

\maketitle

\section{Introduction}\label{intro}

{\it Harmonic maps} are the critical points of the {\em energy functional}
\begin{equation}\label{energia}
E(\varphi)=\frac{1}{2}\int_{M}\,|d\varphi|^2\,dV \, ,
\end{equation}
where $\varphi:M\to N$ is a smooth map between two Riemannian
manifolds $(M,g)$ and $(N,h)$. In particular, $\varphi$ is harmonic if it is a solution of the Euler-Lagrange system of equations associated to \eqref{energia}, i.e.,
\begin{equation}\label{harmonicityequation}
  - d^* d \varphi =   {\trace} \, \nabla d \varphi =0 \, .
\end{equation}
The left member of \eqref{harmonicityequation} is a vector field along the map $\varphi$ or, equivalently, a section of the pull-back bundle $\varphi^{-1} TN$: it is called {\em tension field} and denoted $\tau (\varphi)$. In addition, we recall that, if $\varphi$ is an \textit{isometric immersion}, then $\varphi$ is a harmonic map if and only if the immersion $\varphi$ defines a minimal submanifold of $N$ (see \cite{EL83, EL1} for background). For simplicity, we shall assume that $M$ is compact unless differently specified. However, the Euler-Lagrange equations have validity also when the domain is noncompact, in which case they are referred to compactly supported variations. Now, let us denote $\nabla^M, \nabla^N$ and $\nabla^{\varphi}$ the induced connections on the bundles $TM, TN$ and $\varphi ^{-1}TN$ respectively. The \textit{rough Laplacian} on sections of $\varphi^{-1}  TN$, denoted $\overline{\Delta}$, is defined by
\begin{equation*}
    \overline{\Delta}=d^* d =-\sum_{i=1}^m\left(\nabla^{\varphi}_{e_i}
    \nabla^{\varphi}_{e_i}-\nabla^{\varphi}_
    {\nabla^M_{e_i}e_i}\right)\,,
\end{equation*}
where $\{e_i\}_{i=1}^m$ is a local orthonormal frame field tangent to $M$. In recent years, the following $r$-order versions of the energy functional have attracted an increasing interest from researchers. If $r=2s$, $s \geq 1$:
\begin{eqnarray}\label{2s-energia}
E_{2s}(\varphi)&=& \frac{1}{2} \int_M \, \langle \, \underbrace{(d^* d) \ldots (d^* d)}_{s\, {\rm times}}\varphi, \,\underbrace{(d^* d) \ldots (d^* d)}_{s\, {\rm times}}\varphi \, \rangle_{_N}\, \,dV \nonumber\\ 
&=& \frac{1}{2} \int_M \, \langle \,\overline{\Delta}^{s-1}\tau(\varphi), \,\overline{\Delta}^{s-1}\tau(\varphi)\,\rangle_{_N} \, \,dV\,.
\end{eqnarray}
In the case that $r=2s+1$:
\begin{eqnarray}\label{2s+1-energia}
E_{2s+1}(\varphi)&=& \frac{1}{2} \int_M \, \langle\,d\underbrace{(d^* d) \ldots (d^* d)}_{s\, {\rm times}}\varphi, \,d\underbrace{(d^* d) \ldots (d^* d)}_{s\, {\rm times}}\varphi\,\rangle_{_N}\, \,dV\nonumber \\
&=& \frac{1}{2} \int_M \,\sum_{j=1}^m \langle\,\nabla^\varphi_{e_j}\, \overline{\Delta}^{s-1}\tau(\varphi), \,\nabla^\varphi_{e_j}\,\overline{\Delta}^{s-1}\tau(\varphi)\, \rangle_{_N} \, \,dV \,.
\end{eqnarray}
We say that a map $\varphi$ is \textit{$r$-harmonic} if, for all variations $\varphi_t$,
$$
\left .\frac{d}{dt} \, E_{r}(\varphi_t) \, \right |_{t=0}\,=\,0 \,\,.
$$
In the case that $r=2$, the functional \eqref{2s-energia} is called \textit{bienergy} and its critical points are the so-called \textit{biharmonic maps}. At present, a very ample literature on biharmonic maps is available and we refer to \cite{Chen, Jiang, SMCO, Ou} for an introduction to this topic. More generally, the \textit{$r$-energy functionals} $E_r(\varphi)$ defined in \eqref{2s-energia}, \eqref{2s+1-energia} have been intensively studied (see \cite{Volker, Volker2, Maeta1, Maeta3, Maeta2, Mont-Ratto4, Na-Ura, Wang, Wang2}, for instance). In particular, the Euler-Lagrange equations for $E_r(\varphi)$ were obtained by Wang \cite{Wang} and Maeta \cite{Maeta1}. The expressions for their second variation were derived in \cite{Maeta3}, where it was also shown that a biharmonic map is not always $r$-harmonic ($r \geq 3$) and, more generally, that an $s$-harmonic map is not always $r$-harmonic ($2\leq s < r$). On the other hand, any harmonic map is trivially $r$-harmonic for all $r\geq 2$. Therefore, we say that an $r$-harmonic map is {\it proper} if it is not harmonic (similarly, an $r$-harmonic submanifold, i.e., an $r$-harmonic isometric immersion, is {\it proper} if it is not minimal). As a general fact, when the ambient space has nonpositive sectional curvature there are several results which assert that, under suitable conditions, an $r$-harmonic submanifold is minimal (see \cite{Chen}, \cite{Maeta1}, \cite{Maeta4} and \cite{Na-Ura}, for instance), but the Chen conjecture that an arbitrary biharmonic submanifold of $\R^n$ must be minimal is still open (see \cite{Chen2} for recent results in this direction). More generally, the Maeta conjecture (see \cite{Maeta1}) that any $r$-harmonic submanifold of the Euclidean space is minimal is open. By contrast, let us denote by $\s^m(R)$ the Euclidean sphere of radius $R$ and write $\s^m$ for $\s^m(1)$: for the purposes of the present paper it is important to recall the following examples of proper $r$-harmonic submanifolds into spheres (see \cite{Maeta2} for the case $r=3$ and \cite{Mont-Ratto4} for $r\geq 4$):
\begin{theorem}\label{Corollary-parallel-spheres} Assume that $r \geq 2, m \geq 2$. Then the small hypersphere $ i\,: \, \s^{m-1}(R)\, \hookrightarrow \, \s^{m}$ is a proper $r$-harmonic submanifold of $\s^{m}$ if and only if the radius $R$ is equal to $1 \slash \sqrt{r}$.
\end{theorem}
\begin{theorem}\label{Corollary-parallel-spheres-Clifford} Let $r \geq 2,\,p,q \geq 1 $ and assume that the radii $R_1,R_2$ verify $R_1^2+R_2^2=1$. Then
the generalized Clifford torus $ i\,: \, \s^{p}(R_1) \times \s^{q}(R_2)\, \hookrightarrow \, \s^{p+q+1}$ is:

(a) minimal if and only if

\begin{equation}
\label{condizione-minimalita}
  R_1^2= \frac{p}{p+q} \,\,\quad {\rm and } \,\, \quad R_2^2 =\frac{q}{p+q}\,\, ;
 \end{equation}

 (b) a proper $r$-harmonic submanifold of $\s^{p+q+1}$ if and only if \eqref{condizione-minimalita} does not hold and either
\begin{equation*}
 r=2  \,\, , \quad p \neq q \quad {\rm and } \,\, \quad R_1^2=R^2_2= \frac{1}{2}
 \end{equation*}
or $r \geq 3$ and $t=R_1^2$ is a root of the following polynomial:
\begin{equation}\label{r-harmonicity-Clifford}
P(t)=r(p+q)\,t^3+[q-p-r(q+2p)]\,t^2+(2p+rp)\,t \,-\,p \,\,.
\end{equation}
\end{theorem}
\begin{remark} For a discussion on the existence of positive roots of the polynomial $P(t)$ in \eqref{r-harmonicity-Clifford}, which provide proper $r$-harmonic submanifolds, we refer to \cite{Mont-Ratto4}.
\end{remark}
The setting for $r$-harmonicity which we have outlined so far represents, both from the geometric and the analytic point of view, a convenient approach to the study of higher order versions of the classical energy functional. On the other hand, we now have  to point out that actually the first idea of studying higher order versions of the energy functional was formulated in a different way. More precisely, in 1965 Eells and Sampson (see \cite{ES}) proposed the following functionals which we denote $E^{ES}_r(\varphi)$ to remember these two outstanding mathematicians:
\begin{equation}\label{ES-energia}
    E_r^{ES}(\varphi)=\frac{1}{2}\int_{M}\,|(d^*+d)^r (\varphi)|^2\,dV\,\, .
\end{equation}
Now, to avoid confusion, it is important to fix the terminology: as we said above, a map $\varphi$ is $r$-harmonic if it is a critical point of the functional $ E_{r}(\varphi)$ defined in \eqref{2s-energia}, \eqref{2s+1-energia}. Instead, we say that a map $\varphi$ is $ES-r$-harmonic if it is a critical point of the functional $ E^{ES}_{r}(\varphi)$ defined in \eqref{ES-energia}.
The study of \eqref{ES-energia} was suggested again in \cite{EL83}, but so far very little is known about these functionals. The main aim of this paper is to make some progress in the study of $ E^{ES}_{r}(\varphi)$. In particular, in Section\link\ref{firstresults} we shall prove that Theorems\link\ref{Corollary-parallel-spheres} and \ref{Corollary-parallel-spheres-Clifford} also hold for the Eells-Sampson energy functionals $ E^{ES}_{r}(\varphi)$. To prove this, we shall establish a large setting where the classical principle of symmetric criticality of Palais (see \cite{Palais}) applies: we think that this should prove useful also for future developments on this subject. Next, in Section\link\ref{quadriharmonicity}, we shall obtain the Euler-Lagrange equations in the case that $r=4$: we shall first derive them in the general case. Then we shall illustrate some geometric applications and also the simplifications which occur when the target is a space form. We end this section analysing under which conditions on a conformal change of the metric of the domain the identity map becomes $ES-4$-harmonic. Section\link\ref{section-Models} is devoted to the study of rotationally symmetric $r$-harmonic and $ES-r$-harmonic maps between models. First, we concentrate on the study of the case $r=4$ and analyze both differences and common features for $4$-harmonicity and $ES-4$-harmonicity. Next, we shall focus on the existence of critical points within the class of conformal diffeomorphisms and we shall obtain some nonexistence results, but also a new family of examples for all $r \ge2$. The final part of the paper shall concern the study of the second variation. In particular, we shall introduce this problem and then compute index and nullity of some significant examples.

To end this introduction, we think that it is worth pointing out some difficulties which arise in the study of the functionals $ E^{ES}_{r}(\varphi)$ and the main differences with respect to the $ E_{r}(\varphi)$'s.
The two types of functionals coincide when $r=2$ (the case of biharmonic maps) and $r=3$: this is a consequence of the fact that $d^*$ vanishes on $0$-forms and $d^2\varphi=0$, as computed in \cite{EL83}. The first fundamental difference, as it was already observed in \cite{Maeta4}, arises when $r=4$ because $d^2 \tau(\varphi)$ is not necessarily zero unless $N$ is flat or $\dim M=1$. So, in general, we have
\begin{equation}\label{ES-4-energia}
    E_4^{ES}(\varphi)=\frac{1}{2}\int_{M}\left ( |d^2 \tau(\varphi)|^2+|d^*d \tau(\varphi)|^2\right )\,dV=\frac{1}{2}\int_{M}\,|d^2 \tau(\varphi)|^2\,dV+E_4(\varphi)\,\, .
\end{equation}
This description of $E_4^{ES}(\varphi)$ appeared in \cite{Maeta4}, but the Euler-Lagrange equations associated to the first term on the right-side of \eqref{ES-4-energia} have never been computed and that motivated our work of Section\link\ref{quadriharmonicity}.
When $r \geq5$ things become even more complicated. For instance, we know that the integrand of $E_5(\varphi)$ is the squared norm of a $1$-form, but we cannot write $E_5^{ES}(\varphi)$ as the sum of $E_5(\varphi)$ and a functional which involves only differential forms of degree $p\neq1$. The reason for this is the fact that, in general, the $1$-form $dd^*d\tau(\varphi)$ (whose squared norm is the integrand of $E_5(\varphi)$) may mix up with $d^*d^2\tau(\varphi)$. Difficulties of this type boost as $r$ increases and that motivated our approach of Section\link\ref{firstresults} and Section\link\ref{section-Models}, where we establish a general setting which is suitable to look for symmetric critical points. Another important argument which provides support to this idea is the failure, in general, of the possibility to use the classical Condition (C) of Palais-Smale to deduce the existence of a minimum in a given homotopy class. More precisely, Eells and Sampson, in their paper \cite{ES2}, formulated this hope under the assumption that the order $r$ of the functional is big enough with respect to the dimension of the domain ($2 r > \dim M$). We shall illustrate, at the end of Section\link\ref{section-Models}, that this is not true in general.

\section{The principle of symmetric criticality and existence results}\label{firstresults}
In this section we shall prove a version of Theorems\link\ref{Corollary-parallel-spheres} and \ref{Corollary-parallel-spheres-Clifford} for the Eells-Sampson functional $E_r^{ES}(\varphi)$. More precisely:
\begin{theorem}\label{Corollary-parallel-spheres-ES} Assume that $r \geq 2, m \geq 2$. The small hypersphere $ i\,: \, \s^{m-1}(R)\, \hookrightarrow \, \s^{m}$ is a proper $ES-r$-harmonic submanifold of $\s^{m}$ if and only if the radius $R$ is equal to $1 \slash \sqrt{r}$.
\end{theorem}
\begin{theorem}\label{Corollary-parallel-spheres-Clifford-ES} Let $r \geq 2,\,p,q \geq 1 $ and assume that the radii $R_1,R_2$ verify $R_1^2+R_2^2=1$. Then
the generalized Clifford torus $ i\,: \, \s^{p}(R_1) \times \s^{q}(R_2)\, \hookrightarrow \, \s^{p+q+1}$ is:

(a) minimal if and only if

\begin{equation}\label{condizione-minimalita-ES}
  R_1^2= \frac{p}{p+q} \,\,\quad {\rm and } \,\, \quad R_2^2 =\frac{q}{p+q}\,\, ;
 \end{equation}

 (b) a proper $ES-r$-harmonic submanifold of $\s^{p+q+1}$ if and only if \eqref{condizione-minimalita-ES} does not hold and either
\begin{equation}\label{caso:r=2-ES}
 r=2  \,\, , \quad p \neq q \quad {\rm and } \,\, \quad R_1^2=R^2_2= \frac{1}{2}
 \end{equation}
or $r \geq 3$ and $t=R_1^2$ is a root of the following polynomial:
\begin{equation}\label{r-harmonicity-Clifford-ES}
P(t)=r(p+q)\,t^3+[q-p-r(q+2p)]\,t^2+(2p+rp)\,t \,-\,p \,\,.
\end{equation}
\end{theorem}
The proof of these results requires essentially two ingredients. One is the explicit computation of the terms involving $d^2$: this will be carried out below. The other key tool will be Proposition\link\ref{Principle-symm-crit-BMOR} below, where we show that we can apply a rather general theorem of Palais which ensures the validity of the so-called principle of symmetric criticality. This result of Palais can be found in \cite{Palais}, p.22. However, since the paper \cite{Palais} is written using a rather obsolete notation, we rewrite it here in a form which is suitable for our purposes. In order to do this, let us assume that $G$ is a Lie group which acts on both $M$ and $N$. Then $G$ acts on $C^{\infty}(M,N)$ by $(g \varphi)(x)=g \varphi (g^{-1}x)$, $x\in M$. We say that a map $\varphi$ is $G$-\textit{equivariant} (shortly, equivariant) if $g \varphi= \varphi$ for all $g \in G$. Now, let $E:C^{\infty}(M,N)\to \R$ be a smooth function. Then we say that $E$ is \textit{$G$-invariant} if, for all $\varphi \in C^{\infty}(M,N)$, $E(g\varphi)=E(\varphi)$ for all $g \in G$. Now we can state the main result in this context:
\begin{theorem}\label{Principle-symm-crit}\cite{Palais} Let $M,N$ be two Riemannian manifolds and assume that $G$ is a compact Lie group which acts on both $M$ and $N$. Let $E:C^{\infty}(M,N)\to \R$ be a smooth, $G$-invariant function. If $\varphi$ is $G$-equivariant, then $\varphi$ is a critical point of $E$ if and only if it is stationary with respect to $G$-equivariant variations, i.e., variations $\varphi_t$ through $G$-equivariant maps.
\end{theorem}
\begin{remark} A map $\varphi:M \to N$ can be viewed as a section of the trivial bundle $\pi : M \times N \to M$. Actually, the original version of Theorem\link\ref{Principle-symm-crit} proved in \cite{Palais} includes the case of sections of a general $G$-bundle $\pi : Y \to M$ such that $\pi$ is equivariant.
\end{remark}
Palais observed in \cite{Palais} that, if $G$ is a group of \textit{isometries} of both $M$ and $N$, then the volume functional and the energy functional are both $G$-invariant and so the principle of symmetric criticality stated in Theorem\link\ref{Principle-symm-crit} applies in both cases: the first, beautiful instances of this type can be found in the paper \cite{HL} for minimal submanifolds and in \cite{Smith} for harmonic maps. It is also easy to show that the same is true for the bienergy functional: this is a special case in a more general setting for a reduction theory for biharmonic maps developed in \cite{MOR4, Mont-Ratto3}. Here we shall extend this to the Eells-Sampson functionals $E_r^{ES}(\varphi)$, $r \geq3$. In particular, we shall prove:
\begin{proposition}\label{Principle-symm-crit-BMOR} Let $M,N$ be two Riemannian manifolds and assume that $G$ is a compact Lie group which acts by isometries on both $M$ and $N$. If $\varphi$ is a $G$-equivariant map, then $\varphi$ is a critical point of $E_r^{ES}(\varphi)$ if and only if it is stationary with respect to $G$-equivariant variations.
\end{proposition}
\begin{proof} According to Theorem\link\ref{Principle-symm-crit}, it suffices to show that the Eells-Sampson functionals $E_r^{ES}(\varphi)$ are invariant by isometries. This is a direct consequence of the following two lemmata. We point out that the ideas underlying the proof of Theorem\link\ref{Principle-symm-crit} imply that, if $\varphi$ is a $G$-equivariant map,
$$
dg(\tau^{ES}_r(\varphi))=\tau^{ES}_r(g\circ\varphi)=\tau^{ES}_r(\varphi\circ g)=\tau^{ES}_r(\varphi)\circ g,
$$
for any $g\in G$, i.e., $\tau^{ES}_r(\varphi)$ is a $G$-equivariant section. Moreover, using the exponential map of $N$ and $\tau^{ES}_r(\varphi)$, we can construct as usually a $G$-equivariant variation of the map $\varphi$ such that its variation vector field is $\tau^{ES}_r(\varphi)$, and then we can conclude. By way of summary, to end the proof of Proposition\link\ref{Principle-symm-crit-BMOR}, we just need to establish the following two lemmata.
\end{proof}

\begin{lemma}\label{inv-isometria-dom} Let $\varphi:M \to N$ be a smooth map. If $\psi:M \to M$ is an isometry, then
\begin{equation}\label{invar-isom-M}
E_r^{ES}(\varphi)=E_r^{ES}(\varphi \circ \psi).
\end{equation}
\end{lemma}
\begin{proof} Let $\omega \in C(\Lambda^kT^\ast M\otimes(\varphi \circ \psi)^{-1}TN)=A^k((\varphi \circ \psi)^{-1}TN)$
and $\overline{\omega} \in A^k(\varphi^{-1}TN)$. We say that $\omega$ and $\overline{\omega}$ are $\psi$-related if $\omega$ is the pull-back of $\overline{\omega}$, i.e.,
\begin{equation*}
\omega\left (X_1,\ldots,X_k \right )_x=\overline{\omega}\left (d\psi(X_1),\ldots,d\psi(X_k )\right )_{\psi(x)}
\end{equation*}
for all vectors $X_1,\ldots,X_k\in T_xM$ and for all $x\in M$. Performing a change of variables $y=\psi(x)$ and using the fact that $\psi$ is an isometry it is easy to verify that, if $\omega$ and $\overline{\omega}$ are $\psi$-related, then $|\omega|^2$ and $|\overline{\omega}|^2$ are $\psi$-related and
\begin{equation*}
\int_M \left | \omega \right |^2\, dV = \int_M \left |\overline{ \omega} \right |^2 \, dV\,.
\end{equation*}
Therefore, in order to prove \eqref{invar-isom-M}, it suffices to show that
\begin{equation}\label{psi-relation-key}
\left(d^{\varphi\circ\psi}+(d^{\varphi\circ\psi})^*\right)^{r-2}\tau(\varphi \circ \psi)\,\, {\rm and }\,\,\left(d^{\varphi}+(d^{\varphi})^*\right)^{r-2}\tau(\varphi ) \quad  \quad {\rm are}\,\,\psi-{\rm related}\,,
\end{equation}
where the notation $d^{\varphi\circ\psi}$, $d^{\varphi}$, $(d^{\varphi\circ\psi})^*$, $(d^{\varphi})^*$ highlights in an obvious way the dependence on the connection under consideration.
First, we observe that the $0$-forms $\tau(\varphi \circ \psi)$ and $\tau(\varphi)$ are $\psi$-related. Indeed,
\[
\tau(\varphi \circ \psi)(x)=d\varphi(\tau(\psi))+{\rm Trace}\nabla d\varphi\left ( d\psi(e_i),d\psi(e_i))\right)=\tau(\varphi)(\psi(x))
\]
because $\psi$ is an isometry. Then, by a routine induction argument, we can say that claim \eqref{psi-relation-key} is proved if we show that the following two facts are true:
\begin{equation}\label{psi-relation-key-2}
{\small 
\begin{array}{ll}
{\rm(i)} & {\rm If \,\,two}\, \,k{\rm -forms}\,\, \omega\,\,{\rm and}\,\,\overline{\omega}\,\, {\rm are}\,\, \psi{\rm -related},\,{\rm then} \,\,d^{\varphi\circ\psi }\omega\,\,{\rm and}\,\, d^{\varphi}\overline{\omega} \,\,{\rm are}\,\, \psi{\rm -related}; \\
{\rm(ii)} & {\rm If \,\, two}\,\, k{\rm -forms}\,\, \omega\,{\rm and}\,\,\overline{\omega}\,\, {\rm are}\, \,\psi{\rm -related},\,{\rm then} \,\,(d^{\varphi\circ\psi })^*\omega\,\,{\rm and}\,\, (d^{\varphi})^*\overline{\omega} \,\,{\rm are}\,\, \psi{\rm -related}\,.
\end{array}
}
\end{equation}
We prove \eqref{psi-relation-key-2}(i): at a point $x\in M$ we have
\begin{eqnarray}\nonumber
d^{\varphi\circ\psi }\omega\left(X_1,\ldots,X_{k+1} \right)&=&\sum_{i=1}^{k+1}(-1)^{i+1}\left ( \nabla_{X_i}^{\varphi\circ\psi}\omega\right )\left(X_1,\ldots,\widehat{X_i},\ldots,X_{k+1} \right)\\ \nonumber
&=&\sum_{i=1}^{k+1}(-1)^{i+1}\Big [\nabla_{X_i}^{\varphi\circ\psi}\left ( \omega\left(X_1,\ldots,\widehat{X_i},\ldots,X_{k+1} \right)\right)\\ \nonumber
&&- \sum_{j\neq i,j=1}^{k+1}\omega\left(X_1,\ldots,\nabla^M_{X_i}X_j,\dots,\widehat{X_i},\ldots,X_{k+1} \right) \Big ] \\ \nonumber
&=&\sum_{i=1}^{k+1}(-1)^{i+1}\Big [\nabla_{d\psi(X_i)}^{\varphi}\left ( \overline{\omega}\left(d\psi(X_1),\ldots,\widehat{d\psi(X_i)},\ldots,d\psi(X_{k+1}) \right)\right)\\\nonumber
&&- \sum_{j\neq i,j=1}^{k+1} \overline{\omega}\left(d\psi(X_1),\ldots,\nabla^M_{d\psi(X_i)}d\psi(X_j),\dots,\widehat{d\psi(X_i)},\ldots,d\psi(X_{k+1}) \right) \Big ] \\ \nonumber
&=&\sum_{i=1}^{k+1}(-1)^{i+1}\left ( \nabla_{d\psi(X_i)}^{\varphi} \overline{\omega}\right )\left(d\psi(X_1),\ldots,\widehat{d\psi(X_i)},\ldots,d\psi(X_{k+1}) \right)\\ \nonumber
&=&d^{\varphi} \overline{\omega}\left(d\psi(X_1),\ldots,d\psi(X_{k+1}) \right)\,, \nonumber
\end{eqnarray}
where for the third equality we have used the hypothesis that $\omega$ and $\overline{\omega}$ are $\psi$-related and $d\psi\left ( \nabla^M_{X_i}X_j\right )=\nabla^M_{d\psi(X_i)}d\psi(X_j)$.
This completes the proof of \eqref{psi-relation-key-2}(i). The argument for \eqref{psi-relation-key-2}(ii) is similar and we omit it, so the proof of Lemma\link\ref{inv-isometria-dom} is ended.
\end{proof}
\begin{lemma}\label{inv-isometria-codom} Let $\varphi:M \to N$ be a smooth map. If $\Psi:N \to N$ is an isometry, then
\[
E_r^{ES}(\varphi)=E_r^{ES}( \Psi\circ \varphi).
\]
\end{lemma}
\begin{proof}
%
Consider $\omega\in A^k(\varphi^{-1}TN)$ and define $d\Psi(\omega)\in A^k((\Psi\circ\varphi)^{-1}TN)$ by
$$
(d\Psi(\omega))(X_1,\ldots,X_k)_x=d\Psi_{\varphi(x)}\big(\omega(X_1,\ldots,X_k)_x\big)
$$
for any vectors $X_1,\ldots,X_k$ tangent to $M$ at $x$, and for any $x\in M$. As $\Psi$ is an isometry, we have
$$
|d\Psi(\omega)|^2=|\omega|^2
$$
and therefore
$$
\int_M |d\Psi(\omega)|^2 dV= \int_M |\omega|^2 dV\,.
$$
We note that for the tension fields $\tau(\varphi)$ and $\tau(\Psi\circ\varphi)$, which are $0$-forms, we have $d\Psi(\tau(\varphi))=\tau(\Psi\circ\varphi)$. As in Lemma~\ref{inv-isometria-dom}, it is sufficient to prove the following formulas:
\begin{equation}\label{eq1:lemma2.6}
d\Psi\big(d^\varphi\omega)=d^{\Psi\circ\varphi}\big(d\Psi(\omega)\big)
\end{equation}
\begin{equation}\label{eq2:lemma2.6}
d\Psi\big((d^\varphi)^\ast\omega\big) =(d^{\Psi\circ\varphi})^\ast\big( d\Psi(\omega)\big).
\end{equation}
We just prove \eqref{eq1:lemma2.6} since \eqref{eq2:lemma2.6} is similar. At a point $x\in M$ we have
\begin{eqnarray*}
(d\Psi(d^{\varphi}\omega))(X_1,\ldots,X_{k+1})&=&d\Psi((d^{\varphi}\omega)(X_1,\ldots,X_{k+1}))\\
&=&d\Psi\Big(\sum_{i=1}^{k+1}(-1)^{i+1}\Big [\nabla_{X_i}^{\varphi} \omega\left(X_1,\ldots,\widehat{X_i},\ldots,X_{k+1}\right)\\
&&- \sum_{j\neq i,j=1}^{k+1}\omega\left(X_1,\ldots,\nabla^M_{X_i}X_j,\dots,\widehat{X_i},\ldots,X_{k+1} \right) \Big ] \Big)\\
&=& \sum_{i=1}^{k+1}(-1)^{i+1}\Big[\nabla_{X_i}^{\Psi\circ\varphi}\Big( (d\Psi(\omega))\left(X_1,\ldots,\widehat{X_i},\ldots,X_{k+1}\right)\Big)\\
&& - \sum_{j\neq i,j=1}^{k+1}(d\Psi(\omega))\left(X_1,\ldots,\nabla^M_{X_i}X_j,\dots,\widehat{X_i},\ldots,X_{k+1} \right) \Big ]\\
&=& \left(d^{\Psi\circ\varphi}(d\Psi(\omega))\right)(X_1,\ldots,X_{k+1})\,.
\end{eqnarray*}
Now, for $r=1,2$, it is clear that $E^{ES}_r(\varphi)=E^{ES}_r(\Psi\circ\varphi)$. If $r\geq 3$ we have
\begin{eqnarray*}
E_r^{ES}(\varphi)&=&\int_M \big| \left({(d^{\varphi})}^*+  d^{\varphi}\right)^r (\varphi) \big|^2 dV\\
&=& \int_M \big| \left({(d^{\varphi})}^*+  d^{\varphi}\right)^{r-2} \tau(\varphi) \big|^2 dV\\
&=& \int_M \big| d\Psi\big(\left({(d^{\varphi})}^*+  d^{\varphi}\right)^{r-2} \tau(\varphi)\big) \big|^2 dV\\
&=& \int_M \big| \left({(d^{\Psi\circ\varphi})}^*+  d^{\Psi\circ\varphi}\right)^{r-2} d\Psi(\tau(\varphi)) \big|^2 dV\\
&=& \int_M \big| \left({(d^{\Psi\circ\varphi})}^*+  d^{\Psi\circ\varphi}\right)^{r-2} \tau(\Psi\circ\varphi) \big|^2 dV\\
&=& E_r^{ES}(\Psi\circ\varphi)\,.
\end{eqnarray*}

\end{proof}
\begin{remark} The conclusion of Proposition\link\ref{Principle-symm-crit-BMOR} is true also for the $r$-energy functional $E_r(\varphi)$: the proof is essentially the same and so we omit the details.
\end{remark}
\begin{remark} All the objects which contribute to the definitions of $E_r^{ES}(\varphi)$ and $E_r(\varphi)$ depend on the Riemannian metrics of $M$ and $N$ and their covariant derivatives. Therefore, one may suspect that automatically these two families of functionals are invariant by isometries. However, there is no proof of this claim in the literature and Proposition\link\ref{Principle-symm-crit-BMOR} plays a central role in this paper and, hopefully, in future works on this subject. For these reasons we thought that it could be useful for the reader to include the details of the proofs of Lemmata\link\ref{inv-isometria-dom} and \ref{inv-isometria-codom}. Moreover, we point out that, when $m=2r$, the functionals $E_r^{ES}(\varphi)$ and $E_r(\varphi)$ are invariant under homothetic changes of the metric on the domain. If we perform such homothetic changes, either in the domain or in the codomain, in the case that $m\neq 2r$, then  the corresponding $r$-energies are multiplied by a constant and the associated Euler-Lagrange equations are invariant. By way of conclusion, we think that these results confirm that these two families of functionals provide geometrically interesting higher order versions of the classical energy functional.
\end{remark}
Now we are in the right position to prove our existence results.
\begin{proof}[Proof of Theorem\link\ref{Corollary-parallel-spheres-ES}]
In order to prove Theorem~\ref{Corollary-parallel-spheres-ES} it is sufficient to determine the condition of $ES-r$-harmonicity for a map defined as follows:
\begin{equation}\label{rotationallysymmetricmaps}\begin{array}{lcll}
                                           \varphi_{\alpha^*} \,:\, &\s^{m-1} &\to &  \s^{m} \subset \R^{m} \times \R \\
                                           &&& \\
                                           &w &\mapsto & \quad \quad ( \sin \alpha^*\,w, \,\cos \alpha^*) \,\, ,
                                         \end{array}
\end{equation}
where $m \geq2$ and $\alpha^*$ is a fixed constant value in the interval $(0,\pi/2)$. Indeed, if $\varphi_{\alpha^*}$ is a map as in \eqref{rotationallysymmetricmaps}, then the induced metric on $\s^{m-1}$ is given by $[\varphi_{\alpha^*}]^* (g_{\s^{m}})= (\sin^2 \alpha^*)\,g_{\s^{m-1}} $. Therefore, since $ES-r$-harmonicity is preserved by multiplication of the Riemannian metric of the domain manifold by a positive constant, we conclude that if $\varphi_{\alpha^*}$ is a proper $ES-r$-harmonic map, then its image $\varphi_{\alpha^*}(\s^{m-1})= \s^{m-1}(\sin \alpha^*)$ is a proper $ES-r$-harmonic small hypersphere of radius $R=\sin \alpha^*$. By way of summary, we only have to prove that the map $\varphi_{\alpha^*}$ in \eqref{rotationallysymmetricmaps} is a proper $ES-r$-harmonic map if and only if $\sin \alpha^* = 1 \slash \sqrt r$. Now, a key step in the proof of Theorem\link\ref{Corollary-parallel-spheres-ES} is the following fact, which we state in the form of a proposition:
\begin{proposition}\label{main-theor-parallels} Let $\varphi_{\alpha^*}\,: \,\,\s^{m-1} \to \,\s^{m}$ be a map of the type \eqref{rotationallysymmetricmaps}. Then
\begin{equation}\label{r-energia-esplicita-parallel-spheres-ES}
E_r^{ES} \left (\varphi_{\alpha^*} \right )=E_r \left (\varphi_{\alpha^*} \right )= c\, \varepsilon_r (\alpha^*)\,,
\end{equation}
where
\[
c=\frac{1}{2}\,{\rm Vol}(\s^{m-1})\,(m-1)^r
\]
and the function $\varepsilon_r \, : (0,\pi/2) \to \R $ is defined by
\begin{equation*}
 \varepsilon_r (\alpha)= \sin^2\alpha \, \cos^{2(r-1)}\alpha \,\, .
\end{equation*}
\end{proposition}
\begin{proof}The second equality in \eqref{r-energia-esplicita-parallel-spheres-ES} was obtained in Lemma 3.8 of \cite{Mont-Ratto4}. Therefore, we just have to prove that the first equality of \eqref{r-energia-esplicita-parallel-spheres-ES} holds.
In \cite{Mont-Ratto4} we computed $(d^*d)^{k}\tau \left (\varphi_{\alpha^*}\right )$ and showed that, for all $k \geq0$, this is of the form $c\,\, \partial / \partial \alpha$, where $c$ is a real constant which depends on $\alpha^*$ and $\partial / \partial \alpha$ is a unit section in the normal bundle of $\s^{m-1}(\sin\alpha^*)$ in $\s^m$. Therefore, it suffices to show that
\begin{equation}\label{d^2alpha}
d^2\, \left (\frac{\partial}{\partial \alpha}\right )\equiv 0
\end{equation}
for any map of the type \eqref{rotationallysymmetricmaps} (note that, with a slight abuse of terminology, $\partial / \partial \alpha$ in \eqref{d^2alpha} represents a section of $\varphi_{\alpha^*}^{-1}T\s^{m}$). Now, in order to compute the left side of \eqref{d^2alpha}, we need to recall some general basic facts (see \cite{EL83}, where a different sign convention for the curvature is used). Let $\varphi:M \to N$ be a smooth map between Riemannian manifolds and $ \sigma \in \Lambda^k \left ( \varphi^{-1}TN \right )$. Then
\[
 d^2 \sigma= R^{\varphi} \wedge \sigma \,,
\]
where $R^{\varphi}$ denotes the curvature tensor field of $\varphi^{-1}TN $. In the special case that $\sigma$ is a $0$-form we have, for all $X,Y \in {C}(TM)$,
\begin{equation}\label{d2-formula}
 d^2 \sigma(X,Y)= R^{\varphi}(X,Y)\sigma=R^N(d\varphi(X),d\varphi(Y)) \sigma \,,
\end{equation}
where $R^N$ denotes the curvature tensor field of $N$. We also recall that, in the special case that $N(\epsilon)$ is a space form of constant sectional curvature $\epsilon$, then
\begin{equation}\label{tensor-curvature-s^n}
R^{N(\epsilon)}(X,Y)Z=\epsilon \big( -\langle X,Z \rangle Y+\langle Y,Z \rangle X \big ) \quad  \forall \,X,Y,Z \in C(TN(\epsilon)) \,.
\end{equation}
Let $\{ e_i \}$, $i=1, \dots, m-1$, be a local orthonormal frame field on $\s^{m-1}$. By using \eqref{d2-formula} and \eqref{tensor-curvature-s^n} we compute:
\begin{equation*}
 d^2\, \left (\frac{\partial}{\partial \alpha}\right )\left ( e_i,e_j \right )=
 R^{\s^{m}}\left(d\varphi_{\alpha^*}(e_i),d\varphi_{\alpha^*}(e_j)\right)\left (\frac{\partial}{\partial \alpha}\right )\,=0 \,, \quad \quad  \,1 \leq i,j \leq m-1 ,
\end{equation*}
and so the proof of Proposition\link\ref{main-theor-parallels} is ended.
\end{proof}
We can now end the proof of Theorem\link\ref{Corollary-parallel-spheres-ES}. We observe that $G={\rm SO}(m)$ acts naturally by isometries on both the domain and the codomain of $\varphi_{\alpha^*}$, and $\varphi_{\alpha^*}$ is $G$-equivariant (the action of $G$ on the codomain is on the first $m$ coordinates of $\R^{m+1}$). Therefore, we can apply the principle of symmetric criticality as stated in Proposition\link\ref{Principle-symm-crit-BMOR} and conclude that $\varphi_{\alpha^*}$ is a critical point of the $ES-r$-energy functional if and only if it is stationary with respect to equivariant variations. Now, since $m \geq2$, this means that we just have to consider variations of the type
\[
\begin{array}{lcll}
                                          \left( \varphi_{\alpha^*}\right)_t \,:\, &\s^{m-1} &\to &  \s^{m} \subset \R^{m} \times \R \\
                                           &&& \\
                                           &w &\mapsto & \quad \quad ( \sin (\alpha^*+h(t))\,w, \,\cos( \alpha^*+h(t)) \,\, ,
                                         \end{array}
\]
where $h(t)$ is an arbitrary smooth function with $h(0)=0$.
Therefore, we conclude that $\varphi_{\alpha^*}$ is $ES-r$-harmonic if and only if $\alpha^*$ is a critical point of $\varepsilon_r({\alpha})$. As shown in \cite{Mont-Ratto4}, this happens if and only if $\sin \alpha^*=1 / \sqrt r$ and so the proof of Theorem\link\ref{Corollary-parallel-spheres-ES} is completed (formally, if $m=2$, within the class of equivariant variations one should also include variations through isometries in the direction tangent to the submanifold. But, since the functional is invariant by isometries, the conclusion is the same).
\end{proof}
\begin{proof}[Proof of Theorem\link\ref{Corollary-parallel-spheres-Clifford-ES}] The proof of this theorem follows the same lines of the proof of Theorem\link\ref{Corollary-parallel-spheres-ES} (here $G={\rm SO}(p+1) \times {\rm SO}(q+1)$) and so we just point out the relevant modifications. In this case we have to study maps of the following type:
\begin{equation}\label{def-Clifford-tori}\begin{array}{lcll}
                                           \varphi_{\alpha^*} \,:\, &\s^p(R_1) \times \s^q(R_2) &\to &  \s^{p+q+1} \subset \R^{p+1}\times \R^{q+1}\\
                                           &&& \\
                                           &(R_1\,w,R_2\,z) &\mapsto & ( \sin \alpha^*\,w,  \cos \alpha^*\,z) \,\, ,
                                         \end{array}
\end{equation}
where $w$ and $z$ denote the generic point of $\s^p$ and $\s^q$ respectively, $\alpha^* $ is an arbitrarily fixed value in the interval $(0,\pi \slash 2)$ and $R_1$, $R_2$ are arbitrary positive constants. First, by the same methods of Proposition \ref{main-theor-parallels}, we conclude that $(d^*d)^{k}\tau \left (\varphi_{\alpha^*}\right )=c \,\eta$, where $c$ is a constant and $\eta$ is a unit section in the normal bundle of $\s^p(\sin\alpha^*)\times\s^q(\cos\alpha^*)$ in $\s^{p+q+1}$. Next, we find that $d^2\eta=0$ and so we obtain:
\begin{proposition}\label{main-theor-Clifford}Let $\varphi_{\alpha^*}\,: \,\,\s^p(R_1) \times \s^q(R_2) \to \,\s^{p+q+1}$ be as in \eqref{def-Clifford-tori}.
\begin{equation*}
E_r^{ES} \left (\varphi_{\alpha^*} \right )=E_r \left (\varphi_{\alpha^*} \right )= c\, \varepsilon^C_r (\alpha^*)\,,
\end{equation*}
where
\[
c=\frac{1}{2}\, {\rm Vol}\left (\s^p(R_1) \times \s^q(R_2)\right )\, \left [\frac{p}{R_1^2} \,-\,\frac{q}{R_2^2}\right ]^2
\]
and the function $\varepsilon^C_r \, : (0,\pi/2) \to \R $ is defined by
\begin{equation*}
 \varepsilon^C_r (\alpha)= \sin^2 \alpha \, \cos^2 \alpha \,\,\left [\frac{p}{R_1^2}\,\cos^2 \alpha\,+\,\frac{q}{R_2^2}\,\sin^2 \alpha\right ]^{r-2} \,\, .
\end{equation*}
\end{proposition}
We observe that, in particular,  a map of the type \eqref{def-Clifford-tori} is harmonic if and only if
\begin{equation}\label{not-harmonic-Clifford}
  \frac{p}{R_1^2} \,-\,\frac{q}{R_2^2} =  0\,.
\end{equation}
Now we apply Proposition\link\ref{Principle-symm-crit-BMOR} and deduce that we can consider equivariant variations only. In a fashion similar to the proof of Theorem\link\ref{Corollary-parallel-spheres-ES} we conclude that $\varphi_{\alpha^*}$ is a proper $ES-r$-harmonic map ($r \geq 2$) if and only if
\eqref{not-harmonic-Clifford} does not hold and $\alpha^*$ is a critical point of the function $\varepsilon^C_r$.

If $r \geq3$, then the explicit form of the condition $(\varepsilon^C_r)' (\alpha^*)=0$ is equivalent to:
\begin{equation}\label{condizione-r-harmonicity-clifford-esplicita}
     \frac{p}{R_1^2}+ \left [ (r-1)\,\left ( \frac{q}{R_2^2}-\frac{p}{R_1^2}\right )-2\,\frac{p}{R_1^2} \right ] \sin^2 \alpha^* + r\,\left [ \frac{p}{R_1^2}\,-\,\frac{q}{R_2^2}\right ]\sin^4 \alpha^* =0 \,\, .
\end{equation}
In the case of maps as in \eqref{def-Clifford-tori} the induced pull-back metric identifies the domain with $\s^p(\sin \alpha^* )\times \s^q(\cos \alpha^*)$. Therefore, in order to ensure that an $ES-r$-harmonic map of the type \eqref{def-Clifford-tori} is an \textit{isometric immersion}, it is enough to determine the solutions of \eqref{condizione-r-harmonicity-clifford-esplicita} with $R_1^2=\sin^2 \alpha^* $ and $R_2^2=\cos^2 \alpha^*$. By setting $R_1^2=\sin^2 \alpha^*=t$, \eqref{condizione-r-harmonicity-clifford-esplicita} becomes equivalent to the fact that $t$ is a root of the polynomial $P(t)$ in \eqref{r-harmonicity-Clifford-ES} ($0<t<1$).

If $r=2$, the condition $(\varepsilon^C_r)' (\alpha^*)=0$ is equivalent to $\alpha^*= \pi /4$. Then, imposing the condition of isometric immersion and since we are looking for proper solutions, we find \eqref{caso:r=2-ES} and the proof is ended.
\end{proof}
\begin{remark}We point out that Theorems\link\ref{Corollary-parallel-spheres} and \ref{Corollary-parallel-spheres-Clifford} were not formulated clearly in \cite{Mont-Ratto4}. More precisely, it was not made clear there that they referred to the $E_r(\varphi)$ functional rather than to $E_r^{ES}(\varphi)$. Similarly, in our recent work \cite{Mont-Ratto5} we proved the existence of several other $G$-equivariant examples of proper $r$-harmonic immersions and maps into rotation hypersurfaces and ellipsoids: again, it was not clearly stated that these examples were obtained by studying $E_r(\varphi)$ and not $E_r^{ES}(\varphi)$. However, with the methods of the present paper, it is not difficult to verify that all the examples of \cite{Mont-Ratto5} are not only $r$-harmonic, but also $ES-r$-harmonic. Therefore, we consider the present section of this work as a natural completion of \cite{Mont-Ratto4} and \cite{Mont-Ratto5}.
\end{remark}
\begin{remark}\label{remark-clifford1} Theorem\link\ref{Corollary-parallel-spheres-Clifford-ES} suggests that the geometric features of proper $ES-r$-harmonic submanifolds differ significantly from the biharmonic case and may depend on $r$. By way of example, assume that $p = q$. Then the polynomial $P(t)$ in \eqref{r-harmonicity-Clifford-ES} takes the following form:
\begin{equation*}
P(t)=p \,(2t-1) \,(rt^2-rt+1) \, .
\end{equation*}
Now, if $2 \leq r \leq 4$, the only root is $t= (1\slash 2)$ and produces a minimal submanifold. But it is important to point out that, if $r \geq 5$, then $P(t)$ has two admissible solutions
$$
t= \frac{1}{2} \pm \frac{1}{2 } \,\sqrt {\frac{r-4}{r}}
$$
which give rise to proper, $r$-harmonic ($ES-r$-harmonic) generalized Clifford tori in $\s^{2p+1}$.
Moreover, let now $p \neq q$ be fixed. It was proved in \cite{Mont-Ratto4}, by studying the discriminant of $P(t)$, that there exist three distinct proper $r$-harmonic ($ES-r$-harmonic) generalized Clifford tori in $\s^{p+q+1}$ provided that $r$ is sufficiently large (see \cite{Mont-Ratto4} for details).

\end{remark}
\begin{remark}It is possible to study biharmonicity and $ES-r$-harmonicity  also when we drop the assumption of isometric immersion. More precisely, let us consider a map $\varphi_{\alpha^*}$ as in \eqref{def-Clifford-tori}: the analysis in the proof of Theorem\link\ref{Corollary-parallel-spheres-Clifford-ES} shows that the map $\varphi_{\alpha^*}$ is proper biharmonic if and only if $\alpha^*=\pi /4$ and $p/R_1^2 \neq q/R_2^2$. If $r \geq3$, the map $\varphi_{\alpha^*}$ is $ES-r$-harmonic if and only if equation \eqref{condizione-r-harmonicity-clifford-esplicita} is satisfied. In particular, a routine analysis shows that, for any $p$, $q$, $R_1$ and $R_2$ such that $p/R_1^2 \neq q/R_2^2$, there always exists a proper  $ES-r$-harmonic map $\varphi_{\alpha^*}$ of type  \eqref{def-Clifford-tori}, but the explicit value of $\alpha^*$ can be obtained only by using numerical methods.
\end{remark}
\begin{remark}
Theorems~\ref{Corollary-parallel-spheres-ES} and \ref{Corollary-parallel-spheres-Clifford-ES} were known when $r=2$ (see \cite{CMO02} and \cite{Jiang}) and $r=3$ (see \cite{Maeta2}). The proofs given in \cite{CMO02, Jiang, Maeta2} do not use a variational approach as we did, but they are based on geometric constraints which the second fundamental form of a biharmonic, or triharmonic, immersion into $\s^m$ must satisfy. In the spirit of the discussion in this remark, we think that it would be interesting to determine the geometric requirements that the second fundamental form of a proper $ES-r$-harmonic immersion must verify.
\end{remark}
\begin{remark}\label{remark-donotknowELequations}
We point out that the use of the principle of symmetric criticality of Proposition\link\ref{Principle-symm-crit-BMOR} enables us to prove the existence of $G$-equivariant critical points even if we do not know the explicit general expression of the $ES-r$-tension field. For this reason, this seems to be a very convenient approach to the study of the Eells-Sampson functionals $E_r^{ES}(\varphi)$. We shall encounter other instances of this type in Section\link\ref{section-Models}.
\end{remark}
\subsection{Curves}\label{Sub-sec-curves} The notions of $r$-harmonicity and $ES-r$-harmonicity can also be defined when
the domain is noncompact, considering compactly supported variations (for
more details see Subsection\link\ref{Sub-Sect-noncompact}). In the special case of curves, it is
easy to check that $\varphi:(M^1,dt^2)\to(N^n,h)$ is $r$-harmonic if and only if
it is $ES-r$-harmonic. Now, we assume
that $\varphi(t)$ is a geodesic and we reparametrise it using a diffeomorphism $\mu(s)$, where $s$ denotes the arc length; i.e., we consider
$\psi(s)=(\varphi\circ\mu )(s)$. Then $\psi$ is proper ES-r-harmonic if and only if $\mu(s)$ is a polynomial of order $r'$, where $2 \leq r' \leq 2r-1$. Indeed, from \cite{Maeta1}, we have:
\[
\tau_r(\psi)=\overline{\Delta}^{r-1}\tau(\psi)=(-1)^{r-1}\,\mu^{(2r)}(s)\varphi'(\mu(s))\,.
\]
The previous observation produces, in the case of noncompact $1$-dimensional domains, proper $ES-r$-harmonic ($r$-harmonic) curves.
\section{The Euler-Lagrange equations for $E_4^{ES}(\varphi)$: the general case and maps into space forms}\label{quadriharmonicity}
The main aim of this section is to compute the Euler-Lagrange equations for the Eells-Sampson functional $E^{ES}_4(\varphi)$. First, we shall obtain the equations in the general case. Next, we shall illustrate some relevant simplifications which occur when the target is a space form. Finally, we will provide some geometric applications in two different contexts: isometric immersions and conformal deformations of the domain metric.

We consider smooth maps between two Riemannian manifolds \(\varphi\colon (M^m,g)\to (N^n,h)\), where \(M\) is compact and both metrics \(g,h\) are fixed. For simplicity, when the context is clear, we shall write $\nabla$ for $\nabla^M$.
For a given arbitrary point \(p\in M\), it is
sometimes easier to consider a geodesic frame field \(\{X_i\}\) around \(p\)
and perform the calculations at the point \(p\).
We recall that, when $r=4$, the Eells-Sampson functional is
\begin{align*}
E^{ES}_4(\varphi)=\frac{1}{2}\int_M|(d^\ast + d)(d\tau(\varphi))|^2\dv
=\frac{1}{2}\int_M|d^\ast d\tau(\varphi)+d^2\tau(\varphi)|^2\dv.
\end{align*}
Note that \(d^\ast d\tau(\varphi)\in C(\varphi^{-1}TN)=A^0(\varphi^{-1}TN)\) and
\(d^2\tau(\varphi)\in C(\Lambda^2T^\ast M\otimes\varphi^{-1}TN)=A^2(\varphi^{-1}TN)\).
In order to simplify the formal sum in \(E_4^{ES}(\varphi)\) we observe that
\begin{align*}
|d^\ast d\tau(\varphi)+d^2\tau(\varphi)|^2=|d^\ast d\tau(\varphi)|^2+|d^2\tau(\varphi)|^2=|\overline{\Delta}\tau(\varphi)|^2+|d^2\tau(\varphi)|^2.
\end{align*}
The curvature term here acquires the form
\begin{align*}
|d^2\tau(\varphi)|^2=|R^\varphi\wedge\tau(\varphi)|^2=\frac{1}{2}\sum_{i,j}|R^N(d\varphi(X_i),d\varphi(X_j))\tau(\varphi)|^2.
\end{align*}
In the sequel, we shall omit to write the symbol $\sum$ when it is clear from the context.
Therefore, we have
\begin{align*}
E^{ES}_4(\varphi)
&=\frac{1}{2}\int_M|\overline{\Delta}\tau(\varphi)|^2\dv
+\frac{1}{4}\int_M|R^N(d\varphi(X_i),d\varphi(X_j))\tau(\varphi)|^2\dv\\
&=E_4(\varphi)+\frac{1}{4}\int_M|R^N(d\varphi(X_i),d\varphi(X_j))\tau(\varphi)|^2\dv.
\end{align*}
\begin{remark}
It was already noted in \cite{Maeta4}, equation (2.8), that the four energy of Eells and Sampson contains a curvature contribution.
\end{remark}
In the following we will determine the Euler-Lagrange equation for \(E^{ES}_4(\varphi)\).
To this end we set
\begin{align*}
\widehat E_4(\varphi)&=\frac{1}{2}\int_M|d^2\tau(\varphi)|^2\dv=\frac{1}{4}\int_M|R^N(d\varphi(X_i),d\varphi(X_j))\tau(\varphi)|^2\dv,
\end{align*}
so that
\[
E^{ES}_4(\varphi)=E_4(\varphi)+\widehat E_4(\varphi).
\]
Let us consider a smooth variation of \(\varphi\), that is we consider a smooth map
\begin{align*}
\Phi\colon\R\times M\to N,\qquad (t,p)\mapsto\Phi(t,p)=\varphi_t(p)
\end{align*}
such that \(\varphi_0(p)=\varphi(p)\) for any \(p\in M\), and denote by $V$ its variation vector field, i.e., \(\frac{d}{dt}\big|_{t=0}\varphi_t=V\). The first variation of \(E_4(\varphi)\) is already known in the literature (see \cite{Maeta1}) and it is given by
\begin{align}
\left .\frac{d}{dt} \, E_{4}(\varphi_t) \, \right |_{t=0}=&-\int_M\langle \overline{\Delta}^3\tau(\varphi)+R^N(\overline{\Delta}\tau(\varphi),\nabla^\varphi_{X_i}\tau(\varphi))d\varphi(X_i)\nonumber \\
\nonumber&\qquad + R^N(d\varphi(X_i),\overline{\Delta}^2\tau(\varphi))d\varphi(X_i)
-R^N(\nabla^\varphi_{X_i}\overline{\Delta}\tau(\varphi),\tau(\varphi))d\varphi(X_i),V\rangle\dv\\
=&-\int_M\langle \tau_4(\varphi), V \rangle\dv\,.\label{first-variation-e4}
\end{align}
We then compute the first variational formula for $\widehat E_4(\varphi)$. For this, we first note that, for \((t,p)\) arbitrary but fixed, we have
\begin{align*}
R^{\varphi_t}_p(X_i,X_j)\tau(\varphi_t)_p=&R^N(d\varphi_{t,p}(X_i),d\varphi_{t,p}(X_j))\tau(\varphi_t)_p \\
=&R^N(d\Phi_{(t,p)}(X_i),d\Phi_{(t,p)}(X_j))\tau(\varphi_t)_p\\
=&R^\Phi_{(t,p)}(X_i,X_j)\tau(\varphi_t)_p,
\end{align*}
where, in the last term, \(\tau(\varphi_t)_p\) is to be understood as a section \(\tilde\tau\) in \(\Phi^{-1}TN\).
Of course, \(\tilde\tau(t,p)=\tau(\varphi_t)_p\) is not equal to \(\tau(\Phi)_{(t,p)}\).
With this setting, the first variation of $\widehat E_4(\varphi_t)$ becomes
\begin{align}
\left .\frac{d}{dt}\widehat E_4(\varphi_t)\right|_{t=0}&=\frac{1}{4}\int_M\frac{\partial}{\partial t}(0,p)(|R^\Phi(X_i,X_j)\tilde\tau|^2)\dv \nonumber\\
&=\frac{1}{2}\int_M\langle\nabla^\Phi_{\frac{\partial}{\partial t}(0,p)}R^\Phi(X_i,X_j)\tilde\tau,R^\varphi(X_i,X_j)\tau(\varphi)\rangle\dv \label{eq-es-4-harmonic1}.
\end{align}

Now, a direct calculation in local coordinates, using
\begin{align*}
\nabla^\Phi_{\frac{\partial}{\partial t}(0,p)}\tilde\tau=-\overline{\Delta} V-\trace R^N(d\varphi(\cdot),V)d\varphi(\cdot),
\end{align*}
gives
\begin{align*}
\nabla^\Phi_{\frac{\partial}{\partial t}(0,p)}R^\Phi(X_i,X_j)\tilde\tau=&(\nabla^N_{V(p)}R^N)(d\varphi_p(X_i),d\varphi_p(X_j),\tau(\varphi)_p)\nonumber \\
&+R^N_{\varphi(p)}(\nabla^\varphi_{X_i}V,d\varphi_p(X_j))\tau(\varphi)_p+R^N_{\varphi(p)}(d\varphi_p(X_i),\nabla^\varphi_{X_j}V)\tau(\varphi)_p \\
&+R^N_{\varphi(p)}(d\varphi_p(X_i),d\varphi_p(X_j))(-\overline{\Delta} V-\trace R^N(d\varphi(\cdot),V)d\varphi(\cdot))\nonumber.
\end{align*}
This fact, together with 
\begin{align*}
\langle R^N(d\varphi(X_i),\nabla^\varphi_{X_j}V)&\tau(\varphi),R^N(d\varphi(X_i),d\varphi(X_j))\tau(\varphi)\rangle \\
&=-\langle R^N(\nabla^\varphi_{X_j}V,d\varphi(X_i))\tau(\varphi),R^N(d\varphi(X_i),d\varphi(X_j))\tau(\varphi)\rangle \\
&=-\langle R^N(\nabla^\varphi_{X_i}V,d\varphi(X_j))\tau(\varphi),R^N(d\varphi(X_j),d\varphi(X_i))\tau(\varphi)\rangle \\
&=\langle R^N(\nabla^\varphi_{X_i}V,d\varphi(X_j))\tau(\varphi),R^N(d\varphi(X_i),d\varphi(X_j))\tau(\varphi)\rangle,
\end{align*}
and taking into account \eqref{eq-es-4-harmonic1}, gives the following formula
\begin{align}
\label{variation-e4es-a}
\left. \frac{d}{dt}\widehat E_4(\varphi_t)\right|_{t=0}=
\frac{1}{2}\int_M \nonumber&
\langle (\nabla^N_{V}R^N)(d\varphi(X_i),d\varphi(X_j),\tau(\varphi))+2R^N(\nabla^\varphi_{X_i}V,d\varphi(X_j))\tau(\varphi) \\
&+R^N(d\varphi(X_i),d\varphi(X_j))(-\overline{\Delta} V-\trace R^N(d\varphi(\cdot),V)d\varphi(\cdot)), \\ \nonumber
& R^N(d\varphi(X_i),d\varphi(X_j))\tau(\varphi)\big\rangle\dv.
\end{align}

We are now in the right position to state the main result of this section:
\begin{theorem}\label{Main-EL-equations-tau4ES}
Let \((M^m,g)\) be a compact Riemannian manifold and \((N^n,h)\) a Riemannian manifold.
Consider a smooth map \(\varphi\colon M\to N\). Then the following formula holds
\begin{align*}
\left . \frac{d}{dt}E_4^{ES}(\varphi_t)\right |_{t=0}=-\int_M\langle\tau_4^{ES}(\varphi),V\rangle\dv,
\end{align*}
where \(\tau_4^{ES}(\varphi)\) is given by the following expression
\begin{align}
\label{tau-4-ES-general-target}
\tau_4^{ES}(\varphi)=
\tau_4(\varphi)+\hat{\tau}_4(\varphi)\,,
\end{align}
where 
\begin{align}
\nonumber\tau_4(\varphi)=&\overline{\Delta}^3\tau(\varphi)+ R^N(d\varphi(X_i),\overline{\Delta}^2\tau(\varphi))d\varphi(X_i)
-R^N(\nabla^\varphi_{X_i}\overline{\Delta}\tau(\varphi),\tau(\varphi))d\varphi(X_i) \\
\nonumber&+R^N(\overline{\Delta}\tau(\varphi),\nabla^\varphi_{X_i}\tau(\varphi))d\varphi(X_i) \,,\\
\label{tau-4-hat}
\nonumber\hat{\tau}_4(\varphi)=& -\frac{1}{2}\big(2\xi_1+2d^\ast\Omega_1+\overline{\Delta}\Omega_0+\trace R^N(d\varphi(\cdot),\Omega_0)d\varphi(\cdot)\big)\,,
\end{align}
and we have used the following abbreviations
\[
\begin{array}{lcll}
\Omega_0&=&R^N(d\varphi(X_i),d\varphi(X_j))(R^N(d\varphi(X_i),d\varphi(X_j))\tau(\varphi)),& \Omega_0\in C(\varphi^{-1}TN), \\
\Omega_1(X)&=&R^N(R^N(d\varphi(X),d\varphi(X_j))\tau(\varphi),\tau(\varphi))d\varphi(X_j),& \Omega_1\in A^1(\varphi^{-1}TN),\\
\xi_1&=&-(\nabla^N R^N)(d\varphi(X_j),R^N(d\varphi(X_i),d\varphi(X_j))\tau(\varphi),\tau(\varphi),d\varphi(X_i)),& \xi_1\in C(\varphi^{-1}TN),
\end{array}
\]
\end{theorem}
\begin{proof}
The proof consists of a manipulation of the terms on the right hand side of \eqref{variation-e4es-a}. First of all we rewrite the second addend as
\begin{align*}
\langle R^N(\nabla^\varphi_{X_i}V&,d\varphi(X_j))\tau(\varphi),R^N(d\varphi(X_i),d\varphi(X_j))\tau(\varphi)\rangle\\
=&\langle R^N(R^N(d\varphi(X_i),d\varphi(X_j))\tau(\varphi),\tau(\varphi))d\varphi(X_j),\nabla^\varphi_{X_i}V\rangle\\
=&\langle\Omega_1(X_i),\nabla^\varphi_{X_i}V\rangle.
\end{align*}

Using that \(\{X_i\}\) is a geodesic frame field around a point \(p\) we obtain, at \(p\),
\begin{align}
\langle \Omega_1(X_i),\nabla^\varphi_{X_i}V\rangle&=X_i\langle\Omega_1(X_i),V\rangle-\langle\nabla^\varphi_{X_i}\Omega_1(X_i),V\rangle \nonumber \\
&=\operatorname{div}Y+\langle d^\ast\Omega_1,V\rangle, \label{proof-main-4-ES-1}
\end{align}
where \(Y=\langle\Omega_1(X_k),V\rangle X_k\) is a well-defined, global tangent vector field on \(M\).

Next, for the third addend on the right hand side of \eqref{variation-e4es-a}, we find
\begin{align*}
-\langle R^N(d\varphi(X_i),d\varphi(X_j))\overline{\Delta} V,R^N(d\varphi(X_i),d\varphi(X_j))\tau(\varphi)\rangle
=\langle\Omega_0,\overline{\Delta} V\rangle.
\end{align*}
It follows that
\begin{align}
-\int_M\langle R^N(d\varphi(X_i),d\varphi(X_j))&\overline{\Delta} V,R^N(d\varphi(X_i),d\varphi(X_j))\tau(\varphi)\rangle\dv\nonumber \\
&=\int_M\langle \overline{\Delta}\Omega_0,V \rangle\dv.\label{proof-main-4-ES-2}
\end{align}

As for the last term on the right hand side of \eqref{variation-e4es-a}, we obtain
\begin{align}
-\langle R^N(d\varphi(X_i)&,d\varphi(X_j))(\trace R^N(d\varphi(\cdot),V)d\varphi(\cdot)),R^N(d\varphi(X_i),d\varphi(X_j))\tau(\varphi)\rangle\nonumber \\
&=\langle R^N(d\varphi(X_i),d\varphi(X_j))(R^N(d\varphi(X_i),d\varphi(X_j))\tau(\varphi)),\trace R^N(d\varphi(\cdot),V)d\varphi(\cdot)\rangle\nonumber \\
&=\langle \Omega_0,R^N(d\varphi(X_k),V)d\varphi(X_k)\rangle \nonumber\\
&=\langle\trace R^N(d\varphi(\cdot),\Omega_0)d\varphi(\cdot),V\rangle.\label{proof-main-4-ES-3}
\end{align}

The term in \eqref{variation-e4es-a} that involves the derivative of the curvature on the target is the most
complicated one. In order to manipulate it we need the following symmetries of the derivative of the curvature tensor field:
 \begin{equation}
   \label{curvature-identity-a}
  \langle (\nabla R)(X,Y,Z,W),T\rangle =\langle (\nabla R)(X,T,W,Z),Y\rangle
 \end{equation}
 \begin{equation*}
  \langle (\nabla R)(X,Y,Z,W),T\rangle =\langle (\nabla R)(X,W,T,Y),Z\rangle
 \end{equation*}
 \begin{equation}
    \label{curvature-identity-c}
  (\nabla R)(X,Y,Z,W) =-(\nabla R)(X,Z,Y,W),
 \end{equation}
where we use the notation
\begin{align*}
\langle (\nabla R)(X,Y,Z,W),T\rangle&=\langle (\nabla_XR)(Y,Z,W),T\rangle \\
&=\langle\nabla_X R(Y,Z)W-R(\nabla_XY,Z)W-R(Y,\nabla_XZ)W-R(Y,Z)\nabla_XW,T\rangle.
\end{align*}
Using the second Bianchi-identity, we have
\begin{align*}
(\nabla^N_{V(p)}R^N)(d\varphi_p(X_i),d\varphi_p(X_j),\tau(\varphi)_p)
=&-(\nabla^N_{d\varphi_p(X_j)}R^N)(V(p),d\varphi_p(X_i),\tau(\varphi)_p) \\
&-(\nabla^N_{d\varphi_p(X_i)}R^N)(d\varphi_p(X_j),V(p),\tau(\varphi)_p).
\end{align*}
This leads us to
\begin{align*}
\langle(\nabla^N_{V(p)}R^N)(d\varphi_p(X_i),&d\varphi_p(X_j),\tau(\varphi)_p),R^N(d\varphi(X_i),d\varphi(X_j))\tau(\varphi)\rangle \\
=&-\langle(\nabla^N_{d\varphi_p(X_j)}R^N)(V(p),d\varphi_p(X_i),\tau(\varphi)_p),R^N(d\varphi(X_i),d\varphi(X_j))\tau(\varphi)\rangle \\
&-\langle(\nabla^N_{d\varphi_p(X_i)}R^N)(d\varphi_p(X_j),V(p),\tau(\varphi)_p),R^N(d\varphi(X_i),d\varphi(X_j))\tau(\varphi)\rangle \\
=&-2\langle(\nabla^N_{d\varphi_p(X_j)}R^N)(V(p),d\varphi_p(X_i),\tau(\varphi)_p),R^N(d\varphi(X_i),d\varphi(X_j))\tau(\varphi)\rangle,
\end{align*}
where we have applied \eqref{curvature-identity-c} in the second step.
Then, applying \eqref{curvature-identity-a}, we obtain
\begin{align}
\langle(\nabla^N_{V(p)}R^N)&(d\varphi_p(X_i),d\varphi_p(X_j),\tau(\varphi)_p),R^N(d\varphi(X_i),d\varphi(X_j))\tau(\varphi)\rangle\nonumber \\
&=-2\langle(\nabla^N R^N)(d\varphi(X_j),R^N(d\varphi(X_i),d\varphi(X_j))\tau(\varphi),\tau(\varphi),d\varphi(X_i)),V\rangle\nonumber\\
&=2\langle \xi_1,V \rangle. \label{proof-main-4-ES-4}
\end{align}
Finally, replacing \eqref{proof-main-4-ES-1}, \eqref{proof-main-4-ES-2}, \eqref{proof-main-4-ES-3}, \eqref{proof-main-4-ES-4} into \eqref{variation-e4es-a}, we obtain
\begin{eqnarray}
\left .\frac{d}{dt}\right|_{t=0}\widehat E_4(\varphi_t)&=&\frac{1}{2}\int_M\langle 2\xi_1+2d^\ast\Omega_1+\overline{\Delta}\Omega_0+\trace R^N(d\varphi(\cdot),\Omega_0)d\varphi(\cdot),V\rangle\dv\nonumber\\
&=&-\,\int_M\langle \hat{\tau}_4(\varphi),V\rangle\dv\,,\nonumber
\end{eqnarray}
from which the proof follows immediately taking into account \eqref{first-variation-e4}.

\end{proof}
\begin{remark}\label{remark-leading-terms}
We point out that the Euler-Lagrange equation $\tau_4^{ES}(\varphi)=0$ is a semi-linear elliptic system of order $8$. The leading terms are given by $\tau_4(\varphi)$, while $\hat{\tau}_4(\varphi)$ provides a differential operator of order $4$.
\end{remark}
\subsection{The case of space form target}
In the case that the target manifold \((N^n,h)\) is a real space form \(N^n(\epsilon)\) with constant curvature \(\epsilon\)
we can expect that the first variational formula of \(\widehat E_4(\varphi)\) simplifies. Indeed, since the curvature is constant, \eqref{variation-e4es-a} becomes:
\begin{align}
\label{spaceform-a}
\left .\frac{d}{dt}\widehat E_4(\varphi_t)\right |_{t=0}=\frac{1}{2}\int_M &
\big\langle 2R^N(\nabla^\varphi_{X_i}V,d\varphi(X_j))\tau(\varphi) \nonumber\\
\nonumber&+R^N(d\varphi(X_i),d\varphi(X_j))(-\overline{\Delta} V-\trace R^N(d\varphi(\cdot),V)d\varphi(\cdot)), \nonumber\\
&R^N(d\varphi(X_i),d\varphi(X_j))\tau(\varphi)\big\rangle\dv.
\end{align}
In the following we will compute all the terms on the right hand side of \eqref{spaceform-a}.
Recall that
\begin{align*}
\langle R^N(\nabla^\varphi_{X_i}V&,d\varphi(X_j))\tau(\varphi),R^N(d\varphi(X_i),d\varphi(X_j))\tau(\varphi)\rangle
=\operatorname{div}Y+\langle d^\ast\Omega_1,V\rangle,
\end{align*}
where \(\Omega_1\in A^1(\varphi^{-1}TN)\) is defined as
\begin{align*}
\Omega_1(X)=R^N\big(R^N(d\varphi(X),d\varphi(X_j))\tau(\varphi),\tau(\varphi)\big)d\varphi(X_j)
\end{align*}
and \(Y=\langle\Omega_1(X_k),V\rangle X_k\) is a well-defined, global vector field on \(M\).
Next, for our purposes, it turns out to be useful to define the following vector field:
\begin{align*}
Z=\langle\tau(\varphi),d\varphi(X_k)\rangle X_k=-(\operatorname{div} S)^\sharp,
\end{align*}
where \(S\) is the stress-energy tensor field associated to \(\varphi\).
Clearly, we have
\begin{align}\label{divZ}
\operatorname{div}Z=|\tau(\varphi)|^2+\langle d\varphi,\nabla^\varphi\tau(\varphi)\rangle.
\end{align}
We can now state our main result in the context of maps into a space form:
\begin{theorem}\label{Prop-space-forms}
In the case that \((N^n,h)=N^n(\epsilon)\)
the terms in the expression of \(\tau_4^{ES}(\varphi)\) given by \eqref{tau-4-ES-general-target}
simplify as follows:
\begin{align*}
\xi_1=&0, \\
\Omega_0=&2\epsilon^2(\trace\langle d\varphi(\cdot),d\varphi(Z)\rangle d\varphi(\cdot)-|d\varphi|^2d\varphi(Z)),\\
\Omega_1=&\epsilon^2\big(|Z|^2d\varphi(\cdot)-Z^\flat\otimes d\varphi(Z)
-\langle d\varphi(Z),d\varphi(\cdot)\rangle\tau(\varphi)+|d\varphi|^2Z^\flat\otimes\tau(\varphi)\big).
\end{align*}
\end{theorem}
\begin{proof}
By assumption \(N\) has constant sectional curvature and that implies \(\xi_1=0\).
By a direct calculation we find:
\begin{align*}
\Omega_0=&\epsilon \langle R^N(d\varphi(X_i),d\varphi(X_j))\tau(\varphi),d\varphi(X_j)\rangle d\varphi(X_i)  \\
&-\epsilon \langle R^N(d\varphi(X_i),d\varphi(X_j))\tau(\varphi),d\varphi(X_i)\rangle d\varphi(X_j) \\
=&\epsilon^2\big(
\langle\tau(\varphi),d\varphi(X_j)\rangle\langle d\varphi(X_i),d\varphi(X_j)\rangle d\varphi(X_i) \\
&-\langle\tau(\varphi),d\varphi(X_i)\rangle\langle d\varphi(X_j),d\varphi(X_j)\rangle d\varphi(X_i) \\
&-\langle\tau(\varphi),d\varphi(X_j)\rangle\langle d\varphi(X_i),d\varphi(X_i)\rangle d\varphi(X_j) \\
&+\langle\tau(\varphi),d\varphi(X_i)\rangle\langle d\varphi(X_j),d\varphi(X_i)\rangle d\varphi(X_j)
\big)\\
=&2\epsilon^2\big(
Z^j\langle d\varphi(X_i),d\varphi(X_j)\rangle d\varphi(X_i)-Z^i\langle d\varphi(X_j),d\varphi(X_j)\rangle d\varphi(X_i)
\big)\\
=&2\epsilon^2\big(
\trace\langle d\varphi(\cdot),d\varphi(Z)\rangle d\varphi(\cdot)-|d\varphi|^2d\varphi(Z)\big).
\end{align*}

In addition, we obtain
\begin{align*}
\Omega_1(X_i)=&\epsilon\big(\langle d\varphi(X_j),\tau(\varphi)\rangle R^N(d\varphi(X_i),d\varphi(X_j))\tau(\varphi)
-\langle d\varphi(X_j),R^N(d\varphi(X_i),d\varphi(X_j))\tau(\varphi)\rangle\tau(\varphi)
\big) \\
=&\epsilon^2\big(
\langle d\varphi(X_j),\tau(\varphi)\rangle\langle\tau(\varphi),d\varphi(X_j)\rangle d\varphi(X_i)
-\langle d\varphi(X_j),\tau(\varphi)\rangle\langle\tau(\varphi),d\varphi(X_i)\rangle d\varphi(X_j) \\
&
-\langle\tau(\varphi),d\varphi(X_j)\rangle\langle d\varphi(X_j),d\varphi(X_i)\rangle\tau(\varphi)
+\langle\tau(\varphi),d\varphi(X_i)\rangle\langle d\varphi(X_j),d\varphi(X_j)\rangle\tau(\varphi)
\big) \\
=&\epsilon^2(|Z|^2d\varphi(X_i)-\langle Z,X_i\rangle d\varphi(Z)-\langle d\varphi(Z),d\varphi(X_i)\rangle\tau(\varphi)
+\langle Z,X_i\rangle|d\varphi|^2\tau(\varphi)),
\end{align*}
where we used the expression for \(Z\) in the last step.

\end{proof}

\subsection{Some geometric applications}\label{Sub-Sect-noncompact}
So far we have studied the Eells-Sampson \(4\)-energy and its critical points
in the case of a compact domain \(M\). However, we can extend our results to the case of
a noncompact domain \(M\). To this purpose, we consider all compact subsets \(D\subset M\)
with smooth boundary and define
\begin{align*}
E_4^{ES}(\varphi;D)=\frac{1}{2}\int_D|(d^\ast+d)^4(\varphi)|^2\dv.
\end{align*}
For each such subset \(D\) we consider all smooth variations \(\Phi=\{\varphi_t\}_t\) of \(\varphi\)
such that \(\varphi_t=\varphi\) on \(M\setminus D\) for any \(t\), that is we consider
all variations which have their support in \(D\).
All terms of the form \(\operatorname{div} Y\) which appear in the derivation
of the Euler-Lagrange equation for \(E^{ES}_4\) have the property that \(Y\)
contains the variation vector field \(V\) or its covariant derivatives of first or second order. For this reason \(Y\) vanishes on \(M\setminus D\)
and, by continuity, on its closure \(\overline{M\setminus D}\).
Consequently, \(Y\) vanishes on the boundary of \(D\).
Finally, using the divergence theorem, we conclude that all the results which we have proved in the compact case also hold in the case of a noncompact
domain.
Now, we recall that
\begin{align}\label{***}
\left . \frac{d}{dt}\widehat E_4(\varphi_t)\right |_{t=0}=\frac{1}{2}\int_M
\langle\nabla^\Phi_{\frac{\partial}{\partial t}(0,p)}R^\Phi(X_i,X_j)\tilde\tau,R^\varphi(X_i,X_j)\tau(\varphi)\rangle\dv.
\end{align}
In particular, we observe that if \(R^\varphi(X,Y)\tau(\varphi)=0\) for any \(X,Y\in C(TM)\), then $\varphi$ is an absolute minimum for $\widehat{E}_4(\varphi)$ and so, from \eqref{***}, we recover that it is a critical point for $\widehat{E}_4(\varphi)$.
By way of summary, we have proved that the following proposition is true for arbitrary \(M\) (compact or noncompact):
\begin{proposition}\label{pro:d2t0}
Let \(\varphi\colon (M^m,g)\to (N^n,h)\) be a smooth map.
Assume that \(R^\varphi(X,Y)\tau(\varphi)=0\) for any \(X,Y\in C(TM)\).
Then \(\varphi\) is a critical point of \(E^{ES}_4\) if and only if it is a critical point of \(E_4\).
\end{proposition}

\begin{corollary}
Let \(\varphi\colon (M^m,g)\to N^n(\epsilon)\) be a smooth map.
Assume that \(\tau(\varphi)\) is orthogonal to the image of the map.
Then \(\varphi\) is ES-4-harmonic if and only if it is \(4\)-harmonic.
In particular, if \(\varphi\colon M^m\to N^n(\epsilon)\) is an isometric immersion,
then it is ES-4-harmonic if and only if it is \(4\)-harmonic.
\end{corollary}
\begin{proof}
Essentially, this corollary is a direct consequence of \eqref{tensor-curvature-s^n}. Alternatively, we can prove it by using our vector field $Z$ as follows:
\begin{align*}
\frac{1}{2}|R^N(d\varphi(X_i),d\varphi(X_j))\tau(\varphi)|^2=&\epsilon^2(|Z|^2|d\varphi|^2-Z^iZ^j\langle d\varphi(X_i),d\varphi(X_j)\rangle) \\
=&\epsilon^2(|Z|^2|d\varphi|^2-|d\varphi(Z)|^2\rangle).
\end{align*}
Now, if \(\tau(\varphi)\) is orthogonal to the image of the map, then \(Z=\langle d\phi(X_k),\tau(\varphi)\rangle X_k=0\) and consequently \(\widehat E_4(\varphi)=0\).
\end{proof}
\subsection{Conformal deformations and $ES-4$-harmonic metrics}\label{Sub-conf-metrics}
In their paper \cite{Baird-K} the authors introduced the notion of a biharmonic metric as follows. Let us consider the identity map ${\rm Id}: (M, g) \to (M, g)$: they say that a conformally equivalent metric $\tilde{g} = e^{2\gamma} g$, where $\gamma$ denotes a smooth function on $M$, is
a\textit{ biharmonic metric} (with respect to $g$) if the identity map
\begin{equation}\label{Id-tilda}
\tilde{{\rm Id}}: (M, \tilde{g}) \to (M, g)
\end{equation}
is biharmonic. There turns out to be an interesting connection between the construction of biharmonic metrics and isoparametric functions. In particular, Baird and Kamissoko proved that, \textit{if $ (M^m, g)$ ($m\neq2$) is an Einstein manifold and
$\tilde{g} = e^{2\gamma} g$ is biharmonic, then $\gamma$ is an isoparametric function. Conversely, given an
isoparametric function $f$ on $M$, there exists a local reparametrization $\gamma=\gamma(f)$ which defines a biharmonic metric.}
In a similar spirit, the aim of this subsection is to introduce the notion of an $ES-4$-harmonic metric and compute the relevant Euler-Lagrange equation using our general results for $ES-4$-harmonic maps. More precisely, let us assume that the manifold $(M,g)$ in \eqref{Id-tilda} is a space form $N(\epsilon)$ of constant sectional curvature $\epsilon$. We shall consider several different differential operators: the symbol $\,\tilde{}\,$ over an operator indicates that it must be computed with respect to the metrics $\tilde{g}$ in the domain and $g$ in the codomain. If the $\,\tilde{}\,$ is omitted, it means that we are considering an operator defined by means of $g$ both in the domain and the target. Now, in order to describe our program, it is convenient to recall a few basic general facts (see \cite{Baird-K}). Let $\varphi:(M^m,g) \to (N^n,h)$ be a smooth map and $\tilde{g}=e^{2\gamma}g$ a metric conformally equivalent to $g$. Then, setting $\tilde{\varphi}:(M^m,\tilde{g}) \to (N^n,h)$, we have:
\begin{equation*}
\label{tau-tilde-general-case}
\tau(\tilde{\varphi})=e^{-2\gamma}\, \left [\tau(\varphi)+(m-2) d\varphi(\grad \gamma )\right ]\,.
\end{equation*}
Moreover, for any $V\in \Gamma(\varphi^{-1}TN)$, we have:
\begin{equation}\label{rough-laplacian-tilde-general-case}
\tilde{\Delta}=\tilde{\overline{\Delta}}=e^{-2\gamma}\, \left [\overline{\Delta}V-(m-2) \nabla^{\varphi}_{\grad \gamma} V\right ]\,.
\end{equation}
We observe that, in the special case of maps as in \eqref{Id-tilda}, the tension field, which we shall denote by $\tilde{\tau}$, assumes the following simple expression:
\begin{equation}\label{tau-tilde-our-case}
\tilde{\tau}=\tau(\tilde{{\rm Id}})=(m-2)\,e^{-2\gamma}\, \grad \gamma \,.
\end{equation}
Now, it is natural to give the following
\begin{definition} Let $(M^m,g)$, $m \neq2$, be a Riemannian manifold. We say that a conformally equivalent metric $\tilde{g} = e^{2\gamma} g$ is
an \textit{$ES-4$-harmonic metric} (with respect to $g$) if the identity map \eqref{Id-tilda} is $ES-4$-harmonic. 
\end{definition}
As an application of Theorems\link\ref{Main-EL-equations-tau4ES} and \ref{Prop-space-forms}, we obtain the following description of $ES-4$-harmonic metrics on space forms:
\begin{proposition}\label{theorem-ES-4-harmonic-metrics} Let $(M^m,g)$, $m \neq2$, be a space form $N(\epsilon)$. Then a conformally equivalent metric $\tilde{g} = e^{2\gamma} g$ is
an $ES-4$-harmonic metric if and only if  
\begin{equation}\label{tau-4-ES-explicit-conformal-def}
\tau_4^{ES}(\tilde{{\rm Id}})=\tau_4(\tilde{{\rm Id}})+\hat{\tau}_4(\tilde{{\rm Id}})=0 \,,
\end{equation}
where, in terms of the tension field $\tilde{\tau}$ in \eqref{tau-tilde-our-case}, we have:
\begin{eqnarray}\label{tau-4-explicit-conformal-def}\nonumber
\tau_4(\tilde{{\rm Id}})&=&\tilde{\Delta}^3 \tilde{\tau}+ \epsilon e^{-2\gamma}\Big [(1-m)\tilde{\Delta}^2 \tilde{\tau}-\nabla_{\tilde{\tau}}\tilde{\Delta} \tilde{\tau}-\nabla_{\tilde{\Delta} \tilde{\tau}}\tilde{\tau}\\
&&+ \left ({\rm div}\tilde{\Delta} \tilde{\tau} \right )  \tilde{\tau}+\left ({\rm div}\tilde{\tau} \right)\tilde{\Delta} \tilde{\tau}\Big ]
\end{eqnarray}
and
\begin{eqnarray}\label{tau-hat-explicit-conformal-def}\nonumber
\hat{\tau}_4(\tilde{{\rm Id}})&=&(m-1)\epsilon^2\,\tilde{\Delta}\left( e^{-4\gamma}\,\tilde{\tau}\right )\\
&&+\epsilon^2 e^{-4\gamma}\Big [ (m-2)({\rm div}\,\tilde{\tau})\tilde{\tau}+(m-2)\nabla_{\tilde{\tau}}\tilde{\tau}+(m-4)|\tilde{\tau}|^2 \grad \gamma \\ \nonumber
&&+(m-2)(m-4) \langle \tilde{\tau}, \grad \gamma \rangle\tilde{\tau} + \grad \big ( |\tilde{\tau}|^2\big)-(m-1)^2\epsilon e^{-2\gamma}\,\tilde{\tau} \Big ],
\end{eqnarray}
where $| (\cdot)|$ and $\langle  (\cdot), (\cdot)\rangle$ are computed with respect to $g$.
\end{proposition}
\begin{proof} The proof amounts to the explicit computation of $\tau_4^{ES}(\tilde{{\rm Id}})$ according to the general formulas obtained in Theorems\link\ref{Main-EL-equations-tau4ES} and \ref{Prop-space-forms}. Since no new ideas are involved in this type of calculations, we limit ourselves to summarize the intermediate steps which can then be added up together to yield \eqref{tau-4-explicit-conformal-def} and \eqref{tau-hat-explicit-conformal-def}. In particular, using the explicit form \eqref{tensor-curvature-s^n} of the sectional curvature tensor field we obtain:
\begin{eqnarray*}
R^{N(\epsilon)}\big(X_i,\tilde{\Delta}^2\tilde{\tau}\big )X_i&=&\epsilon e^{-2\gamma}\,(1-m)\tilde{\Delta}^2 \tilde{\tau} \\
-R^{N(\epsilon)}\big(\nabla_{X_i}\tilde{\Delta}\tilde{\tau},\tilde{\tau}\big)X_i&=&\epsilon e^{-2\gamma}\,\big [-\nabla_{\tilde{\tau}}\tilde{\Delta} \tilde{\tau}-\nabla_{\tilde{\Delta} \tilde{\tau}}\tilde{\tau} \big ] \\
R^{N(\epsilon)}\big(\tilde{\Delta}\tilde{\tau},\nabla_{X_i}\tilde{\tau}\big )X_i&=&\epsilon e^{-2\gamma}\,\big [\left ({\rm div}\tilde{\Delta} \tilde{\tau} \right )  \tilde{\tau}+\left ({\rm div}\tilde{\tau} \right)\tilde{\Delta} \tilde{\tau} \big ]
\end{eqnarray*}
As for the part which concerns $\hat{\tau}_4(\tilde{{\rm Id}})$, we have:
\begin{eqnarray*}
\Omega_0&=&2 \epsilon^2 \, e^{-4\gamma}\,(1-m)\, \tilde{\tau}\\
\Omega_1(X)&=&\epsilon^2 \, e^{-2\gamma}\,\big [ | \tilde{\tau}|^2 X+(m-2)\langle  \tilde{\tau},X\rangle  \tilde{\tau}\big ] \\
\tilde{\Delta}\Omega_0&=&2(1-m)\epsilon^2\,\tilde{\Delta}\left( e^{-4\gamma}\,\tilde{\tau}\right )\\
R^{N(\epsilon)}\big(X_i,\Omega_0\big)X_i&=&2(m-1)^2 \epsilon^3 \,e^{-6\gamma}\,\tilde{\tau}
\end{eqnarray*}
and
\begin{eqnarray*}
d^*\Omega_1&=&- \epsilon^2 \, e^{-4\gamma}\,\Big [ (m-2) ({\rm div}\tilde{\tau})\tilde{\tau}+(m-2)\nabla_{\tilde{\tau}}\tilde{\tau}+(m-4)|\tilde{\tau}|^2 \grad \gamma \\
&& +(m-2)(m-4) \langle \tilde{\tau},\grad \gamma \rangle \tilde{\tau}+ \grad \big ( |\tilde{\tau}|^2\big )\big ]\,.
\end{eqnarray*}
\end{proof}
\begin{remark}Using \eqref{rough-laplacian-tilde-general-case} it is possible to express $\hat{\tau}_4(\tilde{{\rm Id}})$ in terms of $\gamma$. In particular, a straightforward computation shows that
\begin{eqnarray}\label{tau-hat-gamma-conformal-def}\nonumber
\hat{\tau}_4(\tilde{{\rm Id}})&=&\epsilon^2 e^{-8\gamma}(m-2)\Big \{(m-1)\overline{\Delta}\grad \gamma \\
&&+\frac{1}{2}(13m-14) \grad \left(|\grad \gamma|^2 \right )-\Big[(m^2+2m-2)\Delta \gamma \\ \nonumber
&&-(m-1)(m^2-4m-32) |\grad \gamma|^2 +\epsilon (m-1)^2\Big ] \grad \gamma \Big\}\,. 
\end{eqnarray}
Inspection of \eqref{tau-hat-gamma-conformal-def} suggests that the equation $\hat{\tau}_4(\tilde{{\rm Id}})=0$ displays some common features with the condition for biharmonic metrics which was obtained and studied in \cite{Baird-K}. In order to illustrate this claim in more detail, let us assume that $\gamma=\gamma(\rho)$, where $\rho$ denotes the distance from a fixed point $p$. Then a routine computation shows that, in the two significant cases, the equation 
$\hat{\tau}_4(\tilde{{\rm Id}})=0$ takes the following form:

Case $\epsilon=1$, $m \geq3$:
\begin{eqnarray}\label{tau-hat-eq-sfera}\nonumber
\gamma ^{(3)}(\rho )&=&\left(m^2-4 m-32\right)  \dot{\gamma }(\rho
   )^3+\left(m^2+2 m-2\right) \cot (\rho )  \dot{\gamma }(\rho )^2\\
   &&-(m-1) \cot
   (\rho )  \ddot{\gamma }(\rho )+ \dot{\gamma }(\rho ) \left[(m+16)  \ddot{\gamma }(\rho
   )+(m-1) \left(\cot ^2(\rho )-1\right)\right]\,,
\end{eqnarray}
where $0<\rho < \pi$ and the metric $\tilde{g}$ admits a smooth extension through the poles if and only if the function $\gamma$ is smooth on $[0,\pi]$ and
\begin{equation}\label{boundary-gamma-sfera}
\gamma^{(2k+1)}(0)=\gamma^{(2k+1)}(\pi)=0 \quad {\rm for \,\, all\,\, }k \geq 0\,.
\end{equation}
Case $\epsilon=-1$, $m \geq3$:
\begin{eqnarray}\label{tau-hat-eq-iperbolic}\nonumber
\gamma ^{(3)}(\rho )&=&\left(m^2-4 m-32\right) \dot{\gamma }(\rho
   )^3+\left(m^2+2 m-2\right) \coth (\rho )  \dot{\gamma }(\rho )^2\\
   &&-(m-1) \coth
   (\rho )  \ddot{\gamma }(\rho )+ \dot{\gamma }(\rho ) \left[(m+16)  \ddot{\gamma }(\rho
   )+(m-1) \left(\coth ^2(\rho )+1\right)\right]\,,
\end{eqnarray}
where $\rho >0$ and the metric $\tilde{g}$ admits a smooth extension through the pole if and only if the function $\gamma$ is smooth on $[0,+\infty)$ and
\[
\gamma^{(2k+1)}(0)=0 \quad {\rm for \,\, all\,\,} k \geq 0\,.
\]
Now, as in \cite{Baird-K}, let us assume that, in the case $\epsilon=1$, $\gamma$ is a function of the isoparametric function $\cos \rho$, i.e., set
\begin{equation}\label{change-variables}
t=\cos \rho \,, \quad \gamma(\rho)=\xi(t)\,,\quad \beta (t)=\xi'(t) \,,
\end{equation}
where $\xi(t)$ is a smooth function on the closed interval $[-1,1]$ (note that this implies that the boundary conditions \eqref{boundary-gamma-sfera} hold). Then equation \eqref{tau-hat-eq-sfera} can be rewritten in terms of the function $\beta(t)$ as follows:
\begin{eqnarray}\label{tau-hat-beta}\nonumber
 \beta''(t)&=& \left(m^2-4 m-32\right)  \beta (t)^3+\left(m^2+3
   m+14\right) \,\frac{t}{t^2-1} \,\beta (t)^2\\
   &&+\beta (t) \left[(m+16) 
   \beta '(t)+\frac{m-2}{t^2-1} \right]-(m+2)\, \frac{t}{t^2-1} \, \beta '(t)\,,
\end{eqnarray}
where $-1 < t < 1$. We observe that \eqref{tau-hat-beta} has the same analytical structure as equation $(10)$ of \cite{Baird-K}. In particular, away from the singular locus $t=\pm 1$, which corresponds to the two focal varieties of the isoparametric function $\cos \rho$, the standard existence theorem for ordinary differential equations guarantees the existence of local solutions of \eqref{tau-hat-beta}. In general, these solutions may not be globally defined: by way of example, a numerical analysis carried out with the software Mathematica suggests that the solution of \eqref{tau-hat-beta} with $m=8$ and initial conditions $\beta(0)=0,\,\beta'(0)=1$ blows up at $\pm t^*$, where $t^* \approx 0.44$. Similar arguments apply to the case $\epsilon=-1$: here \eqref{change-variables} must be replaced by 
\[
t=\cosh \rho \,, \quad \gamma(\rho)=\xi(t)\,,\quad \beta (t)=\xi'(t) \,,
\]
where now $\xi$ is a smooth function on $[1,+\infty)$ and, in terms of $\beta$, \eqref{tau-hat-eq-iperbolic} becomes again \eqref{tau-hat-beta}, but with $t>1$.

To end this subsection, we point out that the derivation of an expression of the type \eqref{tau-hat-gamma-conformal-def} for $\tau_4(\tilde{{\rm Id}})$ requires very long computations and so we omit details in this direction. We just remark that, again in the special case that we assume $\gamma=\gamma(\rho)$, where $\rho$ denotes the distance from a fixed point $p$, we find that the condition $\tau_4^{ES}(\tilde{{\rm Id}})=0$ in \eqref{tau-4-ES-explicit-conformal-def} becomes an ordinary differential equation of order $7$ for the function $\gamma(\rho)$. In particular, this ordinary differential equation turns out to be of the form
\[
\gamma^{(7)}(\rho)=F(\rho,\dot{\gamma},\ldots,\gamma^{(6)})
\]
for a suitable function $F$, not depending on $\gamma$, which is smooth away from $\rho=0$ and the cut locus. Therefore, the standard local existence and uniqueness theorem for ordinary differential equations guarantees the local existence of $ES-4$-harmonic metrics. We refer to Section\link\ref{section-Models} and, in particular, to Remark\link\ref{Remark-ODE}, for a more detailed discussion of problems of this type.
\end{remark}
\subsection{Second variation}
Let us consider a smooth map \(\varphi\colon (M^m,g)\to (N^n,h)\) and, for simplicity, assume that \(M\) is compact.
We consider a two-parameters smooth variation of \(\varphi\),
that is a smooth map
\begin{align*}
\Phi\colon\R\times\R\times M\to N,\qquad (t,s,p)\mapsto \Phi(t,s,p)=\varphi_{t,s}(p)
\end{align*}
such that \(\varphi_{0,0}(p)=\varphi(p)\) for any \(p\in M\).
To a given two-parameters variation of \(\varphi\) we associate the corresponding variation vector fields, i.e., the sections \(V,W\in C(\varphi^{-1}TN)\)
which are defined by
\begin{align*}
 V(p)=&\left .\frac{d}{dt}\varphi_{t,0}(p)\right |_{t=0}\,\in T_{\varphi(p)}N, \\
W(p)=&\left .\frac{d}{ds}\varphi_{0,s}(p)\right |_{s=0}\,\in T_{\varphi(p)}N.
\end{align*}
We will now compute
\begin{align*}
\left.\frac{\partial^2}{\partial t\partial s}\widehat E_4(\varphi_{t,s})\right|_{(t,s)=(0,0)}\,
\end{align*}
starting with
\begin{align*}
\left.\frac{\partial}{\partial s}\widehat E_4(\varphi_{t,s})\right|_{(t,s)=(t,0)}\,
=&\frac{1}{2}\int_M\langle\nabla^\Phi_{\frac{\partial}{\partial s}(t,0,p)}R^\Phi(X_i,X_j)\tilde\tau,R^{\varphi_{t,0}}(X_i,X_j)\tau(\varphi_{t,0})\rangle\dv,
\end{align*}
where \(\tilde\tau\in C(\Phi^{-1}TN)\) is defined by
\begin{align*}
\tilde\tau(t,s,p)=\tau(\varphi_{t,s})_p\in T_{\varphi_{t,s}(p)}N.
\end{align*}
Then we find
\begin{align*}
\left.\frac{\partial^2}{\partial t\partial s}\widehat E_4(\varphi_{t,s})
\right|_{(t,s)=(0,0)}\,=&\frac{1}{2}\int_M
\big(\langle\nabla^\Phi_{\frac{\partial}{\partial t}(0,0,p)}\nabla^\Phi_{\frac{\partial}{\partial s}}R^\Phi(X_i,X_j)\tilde\tau,R^{\varphi}(X_i,X_j)\tau(\varphi)\rangle \\
&+
\langle\nabla^\Phi_{\frac{\partial}{\partial s}(0,0,p)}R^\Phi(X_i,X_j)\tilde\tau,\nabla^\Phi_{\frac{\partial}{\partial t}(0,0,p)} R^\Phi(X_i,X_j)\tilde\tau\rangle\big)\dv.
\end{align*}
Even if  \(R^\varphi(X,Y)\tau(\varphi)=0\) for any \(X,Y\in C(TM)\),
so that \(\varphi\) is a critical point of \(\widehat E_4\), the Hessian of \(\widehat E_4\)
can be different from zero.
Indeed, in this case we have
\begin{align*}
\left.\frac{\partial^2}{\partial t\partial s} \widehat E_4(\varphi_{t,s})
\right|_{(t,s)=(0,0)}\,=&\frac{1}{2}\int_M
\langle\nabla^\Phi_{\frac{\partial}{\partial s}(0,0,p)}R^\Phi(X_i,X_j)\tilde\tau,\nabla^\Phi_{\frac{\partial}{\partial t}(0,0,p)} R^\Phi(X_i,X_j)\tilde\tau\rangle\dv
\end{align*}
and this term will not vanish in general.
We can conclude that, if \(R^\varphi(X,Y)\tau(\varphi)=0\) and \(\varphi\) is a critical
point for both \(E_4^{ES}\) and \(E_4\), then the stability of \(\varphi\) may depend on which of the two functionals we are actually considering. Since, in this case, $\varphi$ is an absolute minimum point for $\widehat{E}_4$, its index computed with respect to $E_4^{ES}$ could be smaller than the one computed using $E_4$.
However, in the case of a one-dimensional domain, there is no difference.
For this reason, in the final part of
this article, we shall focus on the study of the second variation for curves.
\section{Rotationally symmetric maps and conformal diffeomorphisms}\label{section-Models}
In this section we study the functionals $E_r(\varphi)$ and $E_r^{ES}(\varphi)$ in the context of rotationally symmetric maps. The basic difference with respect to Section\link\ref{firstresults} is the fact that for this family of maps $d^2 \tau(\varphi)$ does not necessarily vanish, as we shall see in Proposition\link\ref{prop-d2taunotzero}, where we shall obtain the relevant ordinary differential equation for $ES-4$-harmonicity. We shall also compute the $r$-harmonicity equation for all $r \geq 2$. Then we shall apply these results to the study of conformal diffeomorphisms.

First, let us introduce a family of warped product manifolds which will be suitable for our purposes. We set
\begin{equation}\label{rotationallysymmetricmanifolds}
\left( M,g_M \right )= \left (\s^{m-1} \times I, f^2(\rho)g_{\s^{m-1}}+d\rho^2 \right ),
\end{equation}
where $I\subset \R$ is an open interval and $f(\rho)$ is a smooth function which is positive on $I$.
\begin{remark}\label{remark-interior-I}
In some instances, it may be of interest to extend the analysis through the closure $\overline{I}$ of $I$. By way of example, if $\overline{I}=[0,+\infty)$ and
\begin{equation*}
\label{condizioni-su-f}
\left \{
  \begin{array}{l}
    f(0)=0 \,, \quad f'(0)=1  ; \\
    \,\\
    f^{(2\ell)}(0)=0 \quad {\rm for} \,\, {\rm all } \,\, \ell \geq 1 \,\, ,\\
  \end{array}
\right .
\end{equation*}
then the manifold \eqref{rotationallysymmetricmanifolds} becomes a \textit{model} in the sense of Greene and Wu (see \cite{GW}). In particular, if $f(\rho)=\rho$ (respectively, $f(\rho)=\sinh \rho$) it is isometric to the Euclidean space $\R^{m}$ (respectively, the hyperbolic space $\H^{m}$). In a similar spirit, if $\overline{I}=[0,\pi]$ and $f(\rho)=\sin \rho$, then we have the Euclidean unit sphere $\s^{m}$.

We point out that all the calculations and results of this section are valid on $I$. In particular, the study of regularity across the loci associated to $\partial I$ (poles or boundary of $M_f$) needs a case by case analysis.
\end{remark}
By way of summary, we shall refer to a manifold as in \eqref{rotationallysymmetricmanifolds} as to a \textit{rotationally symmetric manifold} and, to shorten notation, we shall write $M_{f}$ to denote it.
We work with coordinates $w^j, \rho$ on $M_{f}$, where $w^1,\, \ldots ,\,w^{m-1}$ is a set of local coordinates  on $\s^{m-1}$.
A straightforward computation, based on the well-known formula
\begin{equation*}
\Gamma^k_{ij}= \frac{1}{2}\, g^{k \ell} \left(\frac{\partial g_{j \ell}}{\partial x^i}+ \frac{\partial g_{ \ell i}}{\partial x^j}-\frac{\partial g_{ij}}{\partial x^{\ell}} \right)\,\, ,
\end{equation*}
leads us to establish the following lemma:

\begin{lemma}\label{lemma-christoffels} Let $w^1, \ldots,w^{m-1},\rho$ be local coordinates as above on $M_{f}^m$. Then their associated Christoffel symbols $\Gamma^k_{ij}$ are described by the following table:
\begin{equation*}
\begin{array}{lll}
{\rm (i)}&{\rm If}\, 1 \leq i,j,k \leq m-1: & \Gamma^k_{ij}={}^{\s}\Gamma^k_{ij} \\
{\rm (ii)}&{\rm If}\, 1 \leq i, j \leq m-1: & \Gamma^m_{ij}=\, -\,f(\rho)\,f'(\rho)\,\,(g_{\s})_{ij} \\
{\rm (iii)}&{\rm If}\, 1 \leq i,j \leq m-1: & \Gamma^j_{i m}=\frac{f'(\rho)}{f(\rho)} \,\,\delta_i^j \\
{\rm (iv)}&{\rm If}\, 1 \leq j \leq m: & \Gamma^j_{mm}=0=\Gamma^m_{jm} \,\,,\\
\end{array}
 \end{equation*}
where ${}^{\s}\Gamma^k_{ij}$ and $g_{\s}$ denote respectively the Christoffel symbols and the metric tensor of $\s^{m-1}$ with respect to the coordinates $w^1, \ldots,w^{m-1}$.
\end{lemma}

Now we are in the right position to start our process of computing the quantities and equations which are relevant to the study of our high order energy functionals in the context of maps between two rotationally symmetric manifolds as in \eqref{rotationallysymmetricmanifolds}. More specifically, our first goal is to derive the condition of $r$-harmonicity and $ES-r$-harmonicity for rotationally symmetric maps of the following type:
\begin{equation}\label{rotationallysymmetricmaps-models}\begin{array}{llll}
                                           \varphi_{\alpha} \,:\, &\left (\s^{m-1} \times I, f^2(\rho)g_{\s}+d\rho^2 \right ) &\to &  \left (\s^{m-1} \times I', h^2(\alpha)g_{\s}+d\alpha^2 \right ) \\
                                           &&& \\
                                           &(w,\rho) &\mapsto & (w,\alpha(\rho)) \,\, ,
                                         \end{array}
\end{equation}
where $\alpha(\rho)$ is a smooth function on $I$ with values in $I'$. To denote a rotationally symmetric map as in \eqref{rotationallysymmetricmaps-models} we shall write $\varphi_{\alpha}:M_f \to M_h$ or, if the context is clear, simply $\varphi_\alpha$. Now we begin our work to determine the conditions under which $\varphi_{\alpha}$ is $r$-harmonic ($r \geq 2$). The biharmonicity of rotationally symmetric maps $\varphi_{\alpha}:M_f \to M_h$ was extensively studied in \cite{MOR2}. In particular (see \cite{MOR2}), the tension field of $\varphi_{\alpha}$ is given by:
\begin{equation}\label{tension-field-fi-alfa}
\tau \left (\varphi_{\alpha}\right )= \tau_\alpha (\rho)\, \frac{\partial}{\partial \alpha},
\end{equation}
where
\begin{equation}\label{tension-field-fi-alfa-funzione}
\tau_\alpha =\ddot{\alpha} +(m-1)\, \frac{\dot{f}}{f}\,\dot{\alpha}-\frac{(m-1)}{f^2}\, h(\alpha)\,h'(\alpha)
\end{equation}
and $\cdot$ denotes the derivative with respect to $\rho$. This is the starting point to proceed to the explicit computation of the $r$-energy functional for rotationally symmetric maps $\varphi_\alpha$. We begin with some lemmata whose proofs are based on the calculation of several covariant derivatives by means of Lemma\link\ref{lemma-christoffels}.
Let
$\left\{\partial \slash \partial w^1,\dots,\partial \slash \partial w^{m-1} \right\}
$ be a local coordinate frame field on $\s^{m-1}$ and denote by $dw^1, \dots,dw^{m-1}$ the set of dual 1-forms. We observe that $d\varphi_\alpha( \partial \slash \partial w^j)= \partial \slash \partial w^j$, so that with a slight abuse of notation we shall use $\partial \slash \partial w^j$ both for the domain and the codomain of $\varphi_\alpha$.
We shall write $\tau$ to denote $\tau(\varphi_\alpha)$. Now, $d\tau$, precisely as $d \varphi_\alpha$, is a $1$-form with values in the vector bundle $\varphi_\alpha ^{-1}TM_h$ or, equivalently, a section of $T^*M_f\otimes \varphi_\alpha ^{-1}TM_h$. Our first relevant lemma is:
\begin{lemma}\label{lemma-dtau-explicit}
\begin{equation}\label{d-tau-inlocalcharts}
d \tau= \sum_{j=1}^{m-1}\,\left( \tau_{\alpha}\,\frac{h'( \alpha)}{h( \alpha)}\right)\,dw^j\otimes \frac{\partial}{\partial w^j}\,+\, \dot{\tau}_{\alpha}\, d\rho \otimes \frac{\partial}{\partial \alpha},
\end{equation}
where $\tau_{\alpha}(\rho)$ is the function given in \eqref{tension-field-fi-alfa-funzione} and again $\cdot{}$ denotes the derivative with respect to\link$\rho$.
\end{lemma}
\begin{proof}
The expression \eqref{d-tau-inlocalcharts} for $d\tau$ is an immediate consequence of the following calculations of covariant derivatives for which we use the expression of the Christoffel symbols of $M_h$ as it is given in Lemma\link\ref{lemma-christoffels}:
\begin{align}
\nonumber
\nabla_{\partial \slash \partial w^i}^{\varphi_\alpha} \tau &= \nabla_{\partial \slash \partial w^i}^{\varphi_\alpha} \tau_\alpha \, \frac{\partial}{\partial \alpha} = \tau_\alpha\,\nabla_{\partial \slash \partial w^i}^{M_h} \, \frac{\partial}{\partial \alpha}\\ \nonumber
&= \tau_\alpha\,\left [\sum_{j=1}^{m-1}\,\Gamma_{im}^j\,
\frac{\partial}{\partial w^j}+ \Gamma_{im}^m\,
\frac{\partial}{\partial \alpha}\right ]  \\ \nonumber
&= \tau_\alpha\,\left [\sum_{j=1}^{m-1}\,\frac{h'(\alpha)}{h(\alpha)}\delta_{i}^j\,
\frac{\partial}{\partial w^j}+ 0\,
\frac{\partial}{\partial \alpha}\right ]  \\ 
&= \tau_\alpha\,\frac{h'(\alpha)}{h(\alpha)}\,
\frac{\partial}{\partial w^i}  
\nonumber
\end{align}
and
\begin{align}
\nonumber
\nabla_{\partial \slash \partial \rho}^{\varphi_\alpha} \tau &= \nabla_{\partial \slash \partial \rho}^{\varphi_\alpha} \tau_\alpha \, \frac{\partial}{\partial \alpha} = \dot{\tau}_\alpha\,\frac{\partial}{\partial \alpha}+\dot{\alpha}\,\nabla_{\partial \slash \partial \alpha}^{M_h} \, \frac{\partial}{\partial \alpha}\\ 
&= \dot{\tau}_\alpha\,\frac{\partial}{\partial \alpha}\,.
\nonumber
\end{align}
\end{proof}
It follows easily from Lemma~\ref{lemma-dtau-explicit} that
\begin{align}
\nonumber
|d\tau|^2&=\sum_{i,j=1}^{m-1}\frac{(g_{\s})^{ij}}{f^2}\langle d\tau \left( \frac{\partial}{\partial w^i} \right),d\tau \left(\frac{\partial}{\partial w^j} \right) \rangle_{M_h}+\langle d\tau \left( \frac{\partial}{\partial \rho} \right),d\tau \left( \frac{\partial}{\partial \rho} \right)\rangle_{M_h} \\ 
&=(m-1) \,\frac{h'^2( \alpha)}{f^2}\, \tau_{\alpha}^2\,+\,\dot{\tau}_{\alpha}^2  
\nonumber
\end{align}
and so we can conclude by writing:
\begin{equation*}
E_3(\varphi_\alpha)= \frac{1}{2}\,{\rm Vol}(\s^{m-1})\,\int_I\, \left [\dot{\tau}_{\alpha}^2+(m-1) \,\frac{h'^2( \alpha)}{f^2}\, \tau_{\alpha}^2\right ]\, f^{m-1}\,d\rho \,.
\end{equation*}
Next, we obtain the relevant information concerning the $4$-energy. More precisely, we compute (see \cite{EL83} for calculations of this type):
\begin{lemma}\label{lemma-d*dtau-explicit}
\begin{equation}\label{d*d-tau-inlocalcharts}
-\,d^*d \tau= \left [\ddot{\tau}_{\alpha}+\,(m-1) \,\frac{\dot{f}}{f}\,\dot{\tau}_\alpha\, -\,(m-1)\,\frac{h'^2( \alpha)}{f^2}\,\tau_\alpha \right] \,\frac{\partial}{\partial \alpha}.
\end{equation}
\end{lemma}
\begin{proof}
We compute in local coordinates:
\begin{equation}\label{seconda-forma-coordinate-per-dtau}
\left ( \nabla d \tau \right )^\gamma_{ij} =\left ((d\tau)_j^\gamma \right ) _{i} - {}^{M_f} \Gamma^k_{ij} (d\tau)_k^\gamma + {}^{M_h} \Gamma^\gamma_{\beta \delta}  (d\tau)_i^\beta \, \varphi^\delta_{j}
\end{equation}
and we have to calculate
$$
\left [-\,d^*d \tau \right ]^\gamma=\left ( g_{M_f}\right )^{ij}\, \left ( \nabla d \tau \right )^\gamma_{ij} \,\,.
$$
From Lemma\link\ref{lemma-dtau-explicit} we know that the only nonzero entries of $d\tau$ are
\begin{equation}\label{nonzero-entries-dtau}
(d\tau)_i^i = \tau_\alpha \, \frac{h'(\alpha)}{h(\alpha)} \quad (1 \leq i \leq m-1) \quad {\rm and} \quad (d\tau)_m^m = \dot{\tau}_\alpha \,\,.
\end{equation}
Using \eqref{nonzero-entries-dtau} and the expression of the Christoffel symbols given in Lemma\link\ref{lemma-christoffels} into \eqref{seconda-forma-coordinate-per-dtau} we compute and find
\begin{align}\label{ultimo-step-d*dtau}\nonumber
\left [-\,d^*d \tau \right ]^\gamma&= 0 \quad {\rm if} \quad 1 \leq \gamma \leq m-1 \,\, ; \\ \nonumber
\left [-\,d^*d \tau \right ]^m& =\left ( g_{M_f}\right )^{ij}\,\left [\dot{ (d\tau)_m^m}-{}^{M_f} \Gamma^m_{ij} (d\tau)_m^m  +{}^{M_h} \Gamma^m_{\beta j}  (d\tau)_i^\beta \right ] \\ \nonumber
& =\ddot{ \tau}_\alpha+ (m-1)\,\frac{\dot{f}}{f}\,\dot{ \tau}_\alpha -(m-1)\,\frac{h'^2( \alpha)}{f^2}\,\tau_\alpha  \nonumber
\end{align}
from which \eqref{d*d-tau-inlocalcharts} follows immediately.
\end{proof}
It follows readily from Lemma~\ref{lemma-d*dtau-explicit} that
\begin{equation}\label{4-energy-rot-symm-maps}
E_4(\varphi_\alpha)= \frac{1}{2}\,{\rm Vol}(\s^{m-1})\,\int_I\,\left [\ddot{\tau}_{\alpha}+\,(m-1) \,\frac{\dot{f}}{f}\,\dot{\tau}_\alpha\, -\,(m-1)\,\frac{h'^2( \alpha)}{f^2}\,\tau_\alpha \right]^2 \, f^{m-1}\,d\rho \,.
\end{equation}
Now we are in the right position to prove a result which summarises the present discussion:
\begin{theorem}\label{theorem-r-energy}
Set $V=f^{m-1}$ and denote
\begin{align}\label{assumptions-recursivity}
\mathcal{T}_2&= \tau_{\alpha}  \nonumber\\
\mathcal{T}_{2k}&= \ddot{\mathcal{T}}_{2(k-1)}+\,(m-1) \,\frac{\dot{f}}{f}\,\dot{\mathcal{T}}_{2(k-1)}\, -\,(m-1)\,\frac{h'^2( \alpha)}{f^2}\,\mathcal{T}_{2(k-1)}  \quad ( k \geq 2)\\ \nonumber
\mathcal{T}_{2k+1}&=\left [\dot{\mathcal{T}}_{2k}^2+(m-1) \,\frac{h'^2( \alpha)}{f^2}\, \mathcal{T}_{2k}^2\right ]^{1/2} \quad (k \geq 1)\,\,,
\end{align}
where $\tau_{\alpha}(\rho)$ is the function introduced in \eqref{tension-field-fi-alfa-funzione} and again $\cdot$ indicates the derivative with respect to $\rho$. Then the $r$-energy of a rotationally symmetric map $\varphi_{\alpha}:M_f \to M_h$ as in \eqref{rotationallysymmetricmaps-models} is
\begin{equation}\label{r-energy-rot-symm-maps-Lagrangian}
E_r(\varphi_\alpha)={\rm Vol}(\s^{m-1})\,\int_I\,L_r\left(\rho, \alpha(\rho), \dot{\alpha}(\rho),\ldots, \alpha^{(r)}(\rho)\right) \, d\rho \quad (r \geq 2) \,,
\end{equation}
where the explicit expression for the Lagrangians $L_r$ is:
\begin{align}\label{r-energy-explicit-rot-maps}
L_r &=  \frac{1}{2}\,\,\mathcal{T}_{r}^2\, \,V\,  \,\, \quad \quad (r \geq 2) \,.
\end{align}
Moreover, $\varphi_\alpha$ is an $r$-harmonic map if and only if the function $\alpha$ satisfies the Euler-Lagrange equation
\begin{equation}\label{Euler-lagrange-generale}
\sum_{i=1}^r \, (-1)^i \, \frac{d^i}{d\rho^i}\,\left( \frac{\partial L_r}{\partial \alpha^{(i)}}\right)+\frac{\partial L_r}{\partial \alpha}=0 \,\,.
\end{equation}
\end{theorem}
\begin{proof}
First, we observe that $-\,d^*d \tau$ is of the form \eqref{tension-field-fi-alfa}, i.e., a smooth function depending on $\rho$ times $\partial / \partial \alpha$. Therefore, the computations performed in Lemmata~\ref{lemma-dtau-explicit} and \ref{lemma-d*dtau-explicit} can be iterated and we can obtain a recursive expression for the $r$-energy of $\varphi_\alpha$ for all $r \geq2$. That leads us to the definitions \eqref{assumptions-recursivity} and to the conclusion in \eqref{r-energy-rot-symm-maps-Lagrangian}--\eqref{r-energy-explicit-rot-maps}.
Now, since $G=SO(m)$ acts naturally by isometries on both $M_f$ and $M_h$, and the $G$-equivariant maps between $M_f$ and $M_h$ are of the type \eqref{rotationallysymmetricmaps-models} (except in the case $m=2$, where the family of $SO(2)$-equivariant maps also includes irrelevant isometries of $\s^1$),
we can apply the principle of symmetric criticality as in Proposition\link\ref{Principle-symm-crit-BMOR}.
It follows that  $\varphi_\alpha$ is $r$-harmonic if and only if it is a critical point with respect to equivariant variations, i.e., $\varphi_\alpha$ is $r$-harmonic if and only if $\alpha$ is a critical point of the reduced $r$-energy functional
\[
E_{r,{\rm red}}:C^{\infty}(I)\to \R,\quad E_{r,{\rm red}}(\alpha)=E_r(\varphi_\alpha),
\]
where $E_r(\varphi_\alpha)$ is defined in \eqref{r-energy-rot-symm-maps-Lagrangian}. Now, by general principles in the theory of $1$-dimensional calculus of variations, the function $\alpha$ must satisfy the Euler-Lagrange equation associated to $L_r$, i.e., \eqref{Euler-lagrange-generale}, which is an ordinary differential equation of order $2r$ (see also \cite{MOR4}).
\end{proof}
\vspace{2mm}
Now we begin the study of the Eells-Sampson functional $E^{ES}_r(\varphi)$ in this context of rotationally symmetric maps. Our first result is:
\begin{proposition}\label{prop-d2taunotzero} Let $\varphi_\alpha:M_f \to M_h$ be a rotationally symmetric map as in \eqref{rotationallysymmetricmaps-models}. Then
\begin{equation*}
E^{ES}_4(\varphi_\alpha)= \frac{1}{2}\,{\rm Vol}(\s^{m-1})\,\int_I\,\left [(m-1)\,\dot{\alpha}^2 \,\tau^2_\alpha\, \,\frac{h''^2( \alpha)}{f^2} \right] \, f^{m-1}\,d\rho +E_4(\varphi_\alpha)\,,
\end{equation*}
where $\tau_\alpha$ is defined in \eqref{tension-field-fi-alfa-funzione} and the explicit expression of $E_4(\varphi_\alpha)$ is given in \eqref{4-energy-rot-symm-maps}.
\end{proposition}
\begin{proof} According to \eqref{ES-4-energia}, we just have to show that
\begin{equation*}
\frac{1}{2} \int_{M_f}\,\left | d^2 \tau(\varphi_\alpha)\right |^2 \,dV= \frac{1}{2}\,{\rm Vol}(\s^{m-1})\,\int_I\,\left [(m-1)\,\dot{\alpha}^2 \,\tau^2_\alpha\, \,\frac{h''^2( \alpha)}{f^2} \right] \, f^{m-1}\,d\rho \,.
\end{equation*}
To this purpose, it is enough to verify that
\begin{equation}\label{d2tau-rot-maps-explicit-without-integral}
\left | d^2  \tau(\varphi_\alpha)\right |^2= (m-1)\,\dot{\alpha}^2 \,\tau^2_\alpha\, \,\frac{h''^2( \alpha)}{f^2}  \,.
\end{equation}
We use local coordinate frames and compute by means of \eqref{d2-formula}. Writing $d^2 \tau$ instead of $d^2 \tau(\varphi_\alpha)$, we use the expression for the sectional curvature tensor of a warped product (see \cite[Chapter~7, Proposition~42]{Oneill}) and find:
\begin{eqnarray}\label{formula1-d2tau-prop}
d^2 \tau\left (\frac{\partial}{\partial w^i},\frac{\partial}{\partial w^k}\right )&=&0 \,, \qquad 1 \leq i,k \leq m-1 \,; \nonumber \\
d^2 \tau\left (\frac{\partial}{\partial \rho},\frac{\partial}{\partial \rho}\right )&=&0 \,;  \\
d^2 \tau\left (\frac{\partial}{\partial \rho},\frac{\partial}{\partial w^i}\right )&=&\dot{\alpha}\,\frac{h''(\alpha)}{h(\alpha)}\,\tau_\alpha\,\frac{\partial}{\partial w^i} \,, \qquad 1 \leq i \leq m-1 \,.  \nonumber
\end{eqnarray}
Now \eqref{d2tau-rot-maps-explicit-without-integral} follows easily from
\eqref{formula1-d2tau-prop}.
\end{proof}
As an application of Proposition\link\ref{Principle-symm-crit-BMOR}, putting together the results of Proposition\link\ref{prop-d2taunotzero} and \eqref{4-energy-rot-symm-maps} we have:
\begin{proposition}\label{Propos-ES-4-energy-models}
Let
\[
L_4^{ES}= \frac{1}{2}\,\left \{ (m-1)\,\dot{\alpha}^2 \,\tau^2_\alpha\, \,\frac{h''^2( \alpha)}{f^2} +\left [\ddot{\tau}_{\alpha}+\,(m-1) \,\frac{\dot{f}}{f}\,\dot{\tau}_\alpha\, -\,(m-1)\,\frac{h'^2( \alpha)}{f^2}\,\tau_\alpha \right]^2 \right \} \, f^{m-1}\,,
\]
where $\tau_{\alpha}(\rho)$ is the function introduced in \eqref{tension-field-fi-alfa-funzione} and again $\cdot$ denotes the derivative with respect to $\rho$. Then a rotationally symmetric map $\varphi_{\alpha}:M_f \to M_h$ as in \eqref{rotationallysymmetricmaps-models} is $ES-4$-harmonic if and only if it satisfies the Euler-Lagrange equation:
\begin{equation*}
\sum_{i=1}^4 \, (-1)^i \, \frac{d^i}{d\rho^i}\,\left( \frac{\partial L^{ES}_4}{\partial \alpha^{(i)}}\right)+\frac{\partial L^{ES}_4}{\partial \alpha}=0 \,\,.
\end{equation*}
\end{proposition}

\begin{remark}\label{rm-ABC} For equivariant maps as in \eqref{rotationallysymmetricmaps-models}, we describe $\tau^{ES}_4(\varphi_{\alpha})$ as a pair $(\tau^1,\tau^2)$, where $\tau^1(w,\rho)$ is tangent to $\mathbb{S}^{m-1}$ at $w$. Since $\tau^{ES}_4(\varphi_{\alpha})$ is an equivariant section, it follows that $\tau^1=0$ and $\tau^2(w,\rho)$ does not depend on $w$, i.e., the $ES-4$-tension field can be written as a smooth function depending on $\rho$ times $\partial / \partial \alpha$. Then, arguing as in \cite{MOR4}, we get
\begin{equation}\label{A+B+C}
\tau^{ES}_4\left(\varphi_\alpha \right )= -\frac{1}{f^{m-1}} \,\left [\sum_{i=1}^4 \, (-1)^i \, \frac{d^i}{d\rho^i}\,\left( \frac{\partial L_4^{ES}}{\partial \alpha^{(i)}}\right)+\frac{\partial L_4^{ES}}{\partial \alpha} \right ]\, \frac{\partial}{\partial \alpha}\,\,.
\end{equation}
In the case that the target is a space form equation \eqref{A+B+C} can be verified using the results of Section\link\ref{quadriharmonicity}. More precisely, since the interesting situation corresponds to the case that the target is not flat, we assume now that either $h(\alpha)=\sin \alpha$ or $h(\alpha)=\sinh \alpha$. Then we set $\varepsilon_h=1$ if $h(\alpha)=\sin \alpha$ and $\varepsilon_h=-1$ if $h(\alpha)=\sinh \alpha$. Our aim here is to compute  $\tau_4^{ES}(\varphi_\alpha)$ as in \eqref{tau-4-ES-general-target}, taking into account the simplifications obtained in Theorem\link\ref{Prop-space-forms} which apply to the case that the target is a space form. To this purpose, let $W\in C\left ( \varphi_\alpha^{-1}TM_h\right )$ be of the type
\[
W=F(\rho) \frac{\partial}{\partial \alpha}\,,
\]
where $F(\rho)$ is any smooth function on $I$. Then a routine computation shows that
\begin{equation*}
\overline{\Delta} W= \mathcal{L}_\Delta (F)\, \frac{\partial}{\partial \alpha}\,,
\end{equation*}
where the differential operator $\mathcal{L}_\Delta$ is defined by:
\begin{equation*}
\mathcal{L}_\Delta (F)=-\,\left [ \ddot{F}+(m-1)\frac{\dot{f}}{f}\dot{F}-(m-1)\frac{Fh'^2}{f^2}\right ]\,.
\end{equation*}
Now, a long but straightforward computation shows that
\begin{equation}\label{A}
\tau_4(\varphi_\alpha)=\left [ \mathcal{L}_\Delta^3 (\tau_\alpha)-\varepsilon_h(m-1)\frac{\mathcal{L}_\Delta^2 (\tau_\alpha)\,h^2}{f^2}+2 \varepsilon_h(m-1)\frac{\mathcal{L}_\Delta (\tau_\alpha)\,\tau_\alpha \, h\,h'}{f^2}\right ]\, \frac{\partial}{\partial \alpha}\,,
\end{equation}
where $\tau_\alpha$ is given in \eqref{tension-field-fi-alfa-funzione}. Next, we compute the terms coming from the contribution of $\widehat{E}_(\varphi_\alpha)$. We find:
\begin{equation*}
\Omega_0= G\,\frac{\partial}{\partial \alpha}\,,
\end{equation*}
where
\[
G=-2 (m-1)\,\, \frac{\tau_\alpha \,h^2\,\dot{\alpha}^2}{f^2} \,.
\]
Then we compute:
\begin{equation}\label{C}
\frac{1}{2}\left [\overline{\Delta}\Omega_0+\trace R^{M_h}\left (d\varphi_\alpha(\cdot), \Omega_0\right )d\varphi_\alpha(\cdot) \right ]=\frac{1}{2}\left [\mathcal{L}_\Delta(G)-\varepsilon_h (m-1)\frac{h^2\,G}{f^2} \right ]\frac{\partial}{\partial \alpha}\,.
\end{equation}
The next step is to derive the expression for $d^*\Omega_1$. After a long computation we find:
\begin{align}\label{B}
d^*\Omega_1&=- (m-1)\left[ -\,
\frac{\tau_\alpha^2\,\dot{\alpha}^2\,h\,h'}{f^2}+(m-1)\frac{\dot{f}\,h^2\,\dot{\alpha}\,\tau_\alpha^2}{f^3}+ \frac{d}{d\rho}\left(\frac{h^2\,\dot{\alpha}\,\tau_\alpha^2}{f^2}\right )\right ]\,\frac{\partial}{\partial \alpha}\nonumber\\
&=\left[ \frac{f^2 \,h'\,G^2}{4(m-1)h^3\dot{\alpha}^2}
+(m-1)\frac{\dot{f}\,G\,\tau_\alpha}{2f\dot{\alpha}}+ \frac{1}{2}\frac{d}{d\rho}\left(\frac{G\,\tau_\alpha}{\dot{\alpha}}\right )\right ]\,\frac{\partial}{\partial \alpha} \,.
\end{align}
Finally, adding up, as prescribed in \eqref{tau-4-ES-general-target}, the terms provided in \eqref{A}, \eqref{C} and \eqref{B}, we can verify, after a long computation, that the expression for $\tau_4^{ES}(\varphi_\alpha)$ so obtained coincides with the one given in \eqref{A+B+C}.
\end{remark}
\begin{remark}\label{Remark-ODE} The Euler-Lagrange equation for $4$-harmonicity is of the type
\begin{equation}\label{ODE-4-har}
\alpha^{(8)}(\rho)=F(\rho, \alpha,\dot{\alpha},\ldots,\alpha^{(7)})\,.
\end{equation}
The Euler-Lagrange equation for $ES-4$-harmonicity is of the type
\begin{equation}\label{ODE-ES-4-har}
\alpha^{(8)}(\rho)=F(\rho, \alpha,\dot{\alpha},\ldots,\alpha^{(7)})+H(\rho, \alpha,\dot{\alpha},\ldots,\alpha^{(4)})\,,
\end{equation}
where the function $F$ which provides the leading terms is the same in \eqref{ODE-4-har} and \eqref{ODE-ES-4-har}. The function $H$ in \eqref{ODE-ES-4-har} is of the form
\[
H(\rho, \alpha,\dot{\alpha},\ldots,\alpha^{(4)})=-\,\frac{(m-1)}{f^2}\,
\dot{\alpha}^2\,h''^2(\alpha)\,\alpha^{(4)}+ \overline{H}(\rho, \alpha,\dot{\alpha},\ddot{\alpha},\alpha^{(3)})\,.
\]
We note, for future application in the study of the existence of constant solutions, that the function $H$ vanishes when $\alpha(\rho)$ is a
constant function. Also, $H$ vanishes when $h''=0$. The standard existence and uniqueness theorem for ordinary differential equations applies to both \eqref{ODE-4-har}, \eqref{ODE-ES-4-har} and it guarantees the existence of local solutions. Moreover, in the case that $h'' \neq0$, choosing appropriate initial conditions we deduce that locally there are solutions which provide $ES-4$-harmonic maps which are not $4$-harmonic, and conversely as well. To illustrate this point in more detail we assume that $h'' \neq0$, so that the function $H$ in \eqref{ODE-ES-4-har} is not identically zero. Then we choose initial conditions
\begin{equation}\label{cond-iniziali-bis}
\alpha(\rho_0)=\alpha_0 , \,\dot{\alpha}(\rho_0)=\alpha_1 , \dots,\,\alpha^{(7)}(\rho_0)=\alpha_7
\end{equation}
such that $H(\rho_0, \alpha_0,\alpha_1,\ldots,\alpha_4)\neq 0$. Now, let us denote by $\overline{\alpha}$ the unique solution of \eqref{ODE-4-har} which satisfies the initial conditions \eqref{cond-iniziali-bis}, and by $\overline{\alpha}^{ES}$ the unique solution of \eqref{ODE-ES-4-har} which satisfies the initial conditions \eqref{cond-iniziali-bis}. Then $\overline{\alpha}$ gives rise to a map $\varphi_{\overline{\alpha}}$ which is $4$-harmonic, but not $ES-4$-harmonic. Similarly, $\overline{\alpha}^{ES}$ produces a map $\varphi_{\overline{\alpha}^{ES}}$ which is $ES-4$-harmonic, but not $4$-harmonic.
\end{remark}
\subsection{Solutions with $\alpha(\rho)$ equal to a constant}\label{subsect-constant-solutions}
A rather natural question is to investigate the existence of constant solutions $\alpha(\rho)=\alpha^*$. The most interesting case occurs when we study maps from the punctured Euclidean unit ball to the Euclidean sphere. More precisely, let us consider:
\begin{eqnarray}\label{rotationallysymmetricmaps-ball-to-sphere}
\varphi_{\alpha^*}:B^m\backslash  \{ O\} &\to& \s^m \subset \R^m \times \R\\\nonumber
w &\mapsto&\left ( \sin \alpha^*\,\,\frac{w}{|w|}, \cos \alpha^*\right )\,,
\end{eqnarray}
where $\alpha^* \in (0, \pi /2)$ is a constant. We observe that the maps \eqref{rotationallysymmetricmaps-ball-to-sphere} are of the type \eqref{rotationallysymmetricmaps-models} with $f(\rho)=\rho, \,h(\alpha)= \sin \alpha$ and $\alpha(\rho)=\alpha^*$. It is easy to show, using Proposition~\ref{pro:d2t0} and \eqref{formula1-d2tau-prop}, that any such map is  $ES-4$-harmonic if and only if it is $4$-harmonic.  We prove the following result:
\begin{theorem}\label{th-constant-soltz} There exists a map $\varphi_{\alpha^*}:B^m \backslash  \{ O\} \to \s^m$ of the type \eqref{rotationallysymmetricmaps-ball-to-sphere} which is both $ES-4$-harmonic and $4$-harmonic if and only if $m =8,9$.
\end{theorem}
\begin{proof}
Using Proposition\link\ref{Propos-ES-4-energy-models}, or Remark~\ref{rm-ABC}, and computing we find that a map $\varphi_{\alpha^*}$ as in \eqref{rotationallysymmetricmaps-ball-to-sphere} is $ES-4$-harmonic if and only if it is $4$-harmonic if and only if
\begin{eqnarray*}\label{4-harmonicity-constant-soltz}
&&(m-1) \Big[2 \left(258 m^2-2932 m+8002\right) \cos (2
   \alpha^*)+(m-1) \Big(4 (m-1) \cos (2 \alpha^*) \cos (4 \alpha^*)\\\nonumber
   &&+(25 m-131) (2
   \cos (4 \alpha^*)+1)\Big)\Big]+m \big(5 m (217 m-3653)+96539\big)-159999=0\,.
\end{eqnarray*}
Now, setting $x=\cos (2\alpha^*)$, this equation becomes
\begin{eqnarray*}\label{Polinomio-in-x}
&&P_m(x)=(8 m^3-24 m^2+24 m-8)x^3+(100 m^3-724 m^2+1148 m-524)x^2\\
&&+(512 m^3-6368 m^2+21856 m-16000)x+1060 m^3-18084 m^2+96252 m-159868=0\,.\nonumber
\end{eqnarray*}
The roots of the polynomial $P_m(x)$ are
$$
\frac{17-5 m}{m-1}\,;\quad \quad
\frac{1}{4}\left(\frac{-15 m^2+112 m-97}{(m-1)^2}\pm\frac{\sqrt{-199 m^2+2882 m-9399}}{m-1}\right)\,.
$$
Now, a straightforward analysis shows that there exists a (unique) root $x\in(-1,1)$ of $P_m(x)$ if and only if $m=8,9$. The corresponding solutions are
\[
\begin{array}{ll}
m=8 :\quad& \alpha^*=\dfrac{1}{2} \arccos\left(\dfrac{1}{28}\left(\sqrt{921}-23\right)\right)\\
&\\
m=9 :\quad&  \alpha^*= \dfrac{1}{2} \arccos\left(\dfrac{1}{16} \left(\sqrt{105}-19\right)\right)\,.
\end{array}
\]
\end{proof}
\begin{remark} In the case that $m=9$, the solution obtained in Theorem\link\ref{th-constant-soltz} belongs to the Sobolev space $W^{4,2}\left ( B^m,\s^m\right )$ and it provides an example of a \textit{weak critical point} $\varphi_{\alpha^*}:B^m \to \s^m$ for both the $ES-4$-energy and the $4$-energy. Since in this paper we focus on smooth critical points, we do not provide further details in this direction.
\end{remark}
\subsection{Conformal diffeomorphisms}\label{conformal-diffeo} Proper biharmonic conformal diffeomorphisms of $4$-dimensional Riemannian manifolds play an interesting role in the study of the bienergy functional. A basic example (see \cite{Baird}) is the inverse stereographic projection $\varphi:\R^4 \to \s^4\backslash\{{\rm South \,Pole}\}$. We proved in \cite{MOR2} that its restriction to the open unit ball $B^4$ is strictly stable with respect to compactly supported equivariant variations.
Here we investigate the existence of conformal solutions $\varphi_\alpha: M_f \to M_h$. Note that a rotationally symmetric map $\varphi_\alpha$ is conformal if and only if
$$
\dot{\alpha}(\rho)= \frac{h(\alpha(\rho))}{f(\rho)}\,.
$$
The Euler-Lagrange equations given in Theorem\link\ref{theorem-r-energy} and Proposition\link\ref{prop-d2taunotzero} are very long and difficult to deal with, but a computer aided, case by case verification allowed us to check the validity of the following non-existence result:
\begin{proposition}\label{non-existence-conformal-r=3-4} Assume that the models $M_f$ and $M_h$ are chosen among $\R^m, \H^m$ and $\s^m\backslash\{{\rm South \,Pole}\}$ ($m \geq 2$). Let $\varphi_\alpha: M_f \to M_h$ be a rotationally symmetric conformal diffeomorphism. Then $\varphi_\alpha$ is neither proper $r$-harmonic nor proper $ES-r$-harmonic, $r=3,4$.
\end{proposition}
\begin{remark}The dimension $m=2$ is special for harmonic maps since conformal diffeomorphisms of surfaces are always harmonic. In the case that $m=4$, we know that the inverse of the stereographic projection is a proper biharmonic conformal diffeomorphism. Therefore, it was reasonable to expect that some conformal diffemorphism between space forms could provide an example of a $3$-harmonic map when $m=6$, or of a $4$-harmonic ($ES-4$-harmonic) map when $m=8$, but Proposition\link\ref{non-existence-conformal-r=3-4} shows that this is not the case.
\end{remark}
Next, we show that things drastically change and existence may occur if we consider maps into a cylinder, i.e., if $h(\alpha) \equiv 1$. More precisely, using polar coordinates on $\R^m \setminus\{ O \}$, $m \geq 4$, let
\begin{align}\label{mappe-Baird-Fardoun}
\varphi_\alpha \,:\R^m\setminus \{O \}=& \s^{m-1} \times (0,+\infty)  \to \s^{m-1} \times \R \\ \nonumber
(&\underline{\gamma},\qquad \rho)\quad \quad  \mapsto \, \left (\underline{\gamma}, \,\alpha( \rho) \right )  \,,
\end{align}
where $\alpha(\rho)$ is a smooth function on $(0,+\infty)$. In the case that $\alpha(\rho)=\log \rho $ the map $\varphi_\alpha$ in \eqref{mappe-Baird-Fardoun} is a conformal diffeomorphism. Moreover, when $m=4$ it was observed in \cite{Baird} that $\varphi_\alpha$ is proper biharmonic and the study of its equivariant stability was carried out in \cite{MOR-last}. Now, our aim is to investigate maps as in \eqref{mappe-Baird-Fardoun} in the context of our higher order energy functionals. Our main result is the following:
\begin{theorem}\label{prop-Baird-ex} Let $\varphi_\alpha \,:\R^m\setminus \{O \}  \to \s^{m-1} \times \R $ be the conformal diffeomorphism defined as in \eqref{mappe-Baird-Fardoun} with $\alpha(\rho)=\log \rho $.
Then $\varphi_\alpha$ is both proper $ES-r$-harmonic and $r$-harmonic provided that
\begin{equation*}
m=2k \,\,{\rm and}\,\, r \geq k\geq2\,.
\end{equation*}
By contrast, if $m \geq3$ is odd and $r \geq2$, then $\varphi_{\alpha}$ is neither $ES-r$-harmonic nor $r$-harmonic.
\end{theorem}
\begin{proof} First, we observe that the family \eqref{mappe-Baird-Fardoun} is made of rotationally symmetric maps of the type \eqref{rotationallysymmetricmaps-models}, with $f(\rho)=\rho$ and $h(\alpha)\equiv 1$. Moreover, if $\alpha(\rho)=\log \rho $, then $\varphi_{\alpha}$ is a conformal diffeomorphism with conformality factor equal to $1/\rho$. Next, it is not difficult to check that, for any $\alpha$, $(d^*d)^k \tau(\varphi_\alpha)$ is of the form $F(\rho)\partial / \partial \alpha$, where $F(\rho)$ denotes a smooth function of $\rho$. Now, computing as in \eqref{d2tau-rot-maps-explicit-without-integral}, it is easy to deduce that
\[
d^2 \left ( F(\rho) \frac{\partial}{\partial \alpha}\right ) =0\,.
\]
Then it follows easily that, for rotationally symmetric maps $\varphi_\alpha \,:\R^m\setminus \{O \}  \to \s^{m-1} \times \R $, we have $E^{ES}_r(\varphi_\alpha)=E_r(\varphi_\alpha)$ for all $m,r \geq2$. Moreover, the principle of symmetric criticality stated in Proposition\link\ref{Principle-symm-crit-BMOR} applies and so, within this family, $r$-harmonicity and $ES-r$-harmonicity are equivalent.
Now, in this case
\begin{equation*}
\tau \left ( \varphi_\alpha\right )= \mathcal{T}_2\,\frac{\partial}{\partial \alpha}\,,
\end{equation*}
where
\[
\mathcal{T}_2(\rho)= \ddot{\alpha}(\rho) + \frac{m-1}{\rho}\, \dot{\alpha}(\rho)  \,.
\]

For convenience, to end the proof we carry out two separate steps:

\textbf{Step $1$:} The explicit form of the Euler-Lagrange equations for $r$-harmonicity ($r \geq 1$) is:
\begin{equation}\label{EL-equations-conformal-example}
\Delta^{r-1} \mathcal{T}_2=0 \,,
\end{equation}
where $\Delta $ is the Laplace operator which acts on radial functions as follows:
\begin{equation}\label{Laplacian-radial-coordinates}
\Delta g(\rho)=- \left [ \ddot{g}(\rho) + \frac{m-1}{\rho}\, \dot{g}(\rho) \right ] \,.
\end{equation}
\textit{Proof of Step $1$:} We apply the explicit expression of the $r$-tension field given by Maeta in \cite[Theorem~2.5 and Theorem~2.6]{Maeta1}. More precisely, since all the involved curvature terms of $\s^{m-1}\times\R$  vanish, it is easy to obtain
\[
\tau_r(\varphi_\alpha)= \overline{\Delta}^{r-1} \left ( \mathcal{T}_2\,\frac{\partial}{\partial \alpha} \right ) =  \left (\Delta^{r-1} \mathcal{T}_2\right )\,\frac{\partial}{\partial \alpha}  \,,
\]
where the second equality is true because here $h(\alpha)$ is a constant function.

\textbf{Step $2$:} We show that, if $\alpha(\rho)=\log \rho$, then
\begin{equation}\label{key-conformal-example}
\Delta^{r-1} \mathcal{T}_2= 2^{r-1}\,(r-1)!\,\,\frac{1}{\rho^{2r}}\,\, \prod_{k=1}^{r}\big(m-2k\big)\, \quad ( r \geq 1)\,.
\end{equation}
\textit{Proof of Step $2$:} First, we find that, if $\alpha(\rho)= \log \rho$, then
\begin{equation*}
\mathcal{T}_2(\rho)=\frac{m-2}{\rho^2}
\end{equation*}
and so \eqref{key-conformal-example} is true for $r=1$. Then the proof of this step can be completed by induction using \eqref{Laplacian-radial-coordinates}.

Finally, we observe that the conclusion of the proof of Theorem\link\ref{prop-Baird-ex} is an immediate consequence of \eqref{key-conformal-example} together with \eqref{EL-equations-conformal-example}.
\end{proof}
\begin{remark}\label{biharmonic-morphisms} We say that a map is $(p)$-harmonic if it is a critical point of
\[
E_{(p)}(\varphi)= \frac{1}{p} \int_M |d\varphi|^p \,dV \,.
\]
We refer to \cite{DF, HLin} for existence and regularity results for $(p)$-harmonic maps.
The notion of biharmonic morphism was introduced in \cite{LOu}. These maps, which are defined as those which preserve germs of biharmonic functions, were characterized as smooth maps $\varphi:(M^m,g) \to (N^n,h)$ which are horizontally weakly conformal, biharmonic, $(4)$-harmonic and satisfy the following equation:
\begin{eqnarray}\label{biharm-morph-equation}
&&|\tau(\varphi)|^4+2\left (\Delta\lambda^2\right )\,|\tau(\varphi)|^2-4 \left (\Delta \lambda^2 \right ){\rm div}\langle d\varphi,\tau(\varphi)\rangle \\\nonumber
&&+n\left (\Delta \lambda^2 \right )^2+
2\langle d\varphi,\tau(\varphi)\rangle \left (\nabla |\tau(\varphi)|^2\right) +|S|^2=0\,,
\end{eqnarray}
where $\lambda$ is the dilation, $S \in C\left (\odot^2\varphi^{-1}TN\right )$ is the symmetrization of the $g$-trace of $d\varphi \otimes \nabla^{\varphi}\tau(\varphi)$ and 	$\langle d\varphi,\tau(\varphi)\rangle(X)=\langle d\varphi(X),\tau(\varphi)\rangle$ (note that our sign convention for $\Delta$ on functions is different from the one in \cite{LOu}). Now, when $m=4$, the map $\varphi_\alpha$ of Theorem\link\ref{prop-Baird-ex} is horizontally weakly conformal, biharmonic and $(4)$-harmonic, but it does not verify \eqref{biharm-morph-equation} and so it is not a biharmonic morphism.

For rotationally symmetric maps as in \eqref{mappe-Baird-Fardoun} the equation for $r$-harmonicity is \eqref{EL-equations-conformal-example}. Therefore, if we drop the requirement that $\varphi_\alpha$ be a conformal map, by a routine analysis of this linear ODE we can determine other explicit solutions. In particular, we find that $\varphi_\alpha$ is proper $r$-harmonic (and $ES-r$-harmonic) in the following cases (note that, since the $r$-harmonicity equation is linear, linear combinations of solutions provide further solutions; also, adding a harmonic function to a proper $r$-harmonic one yields another proper $r$-harmonic function, $r \geq2$):
\begin{equation}\label{biharm-not-conformal}
\begin{array}{lll}
m\geq 2 & r \geq2 & \alpha(\rho)=\rho^2  \\
m=2k & r \geq k+1\geq 2 & \alpha(\rho)=\rho^2\log \rho \\
m\geq2 & r\geq r'\geq2 &  \alpha(\rho)=\rho^{2r'-m} \,.
\end{array}
\end{equation}
We can observe that, since the operator $\Delta$ in \eqref{EL-equations-conformal-example} is the standard Laplacian of $\R^m$ in polar coordinates, the Almansi property applies (see \cite{Almansi}, and \cite {MR1} for recent developments). In particular, a proper $r$-harmonic solution multiplied by $\rho^2$ becomes a proper $(r+1)$-harmonic example. This observation, together with the result of Theorem\link\ref{prop-Baird-ex}, leads us to conclude, as in \eqref{biharm-not-conformal}, that the function $\alpha(\rho)=\rho^2 \log \rho$ gives rise to proper $r$-harmonic maps if $m=2k$ and $r \geq k+1\geq 2$.
\end{remark}
\subsection{Condition (C)}\label{subsection-condition-C}
To our knowledge, no previous work in the literature clarifies and proves in which contexts the Condition (C) of Palais-Smale holds for the $ES-r$-energy ($r$-energy) functionals.
A general belief (see \cite{ES2, Eliasson, Lemaire}) is that, if $2r>\dim M$ and the curvature of the target is non-positive, then the $ES-r$-energy ($r$-energy) functionals may
satisfy Condition (C). But, for each of these functionals, a further difficulty in the search of proper critical points is the fact that the minimum point in a given homotopy class can very well be reached by a harmonic map. By contrast, when the target has positive curvature, there is little hope that these
higher order energy functionals satisfy Condition (C). We
illustrate this by means of the following result which displays a homotopy class where the $ES-4$-energy functional does not reach the infimum.
\begin{theorem}\label{theor-cond-C}
Let $\T^2$ denote the flat $2$-torus. Then
\begin{itemize}
\item[{\rm (i)}]
\begin{equation}\label{Inf=0}
{\rm Inf} \left \{E_4^{ES}(\varphi) \colon \varphi\in C^{\infty}\left (\T^2, \s^2\right),\,\, \varphi\,  {\rm\,has \,degree \,one} \right \}=0 \,.
\end{equation}
\item[{\rm (ii)}] The functional $E_4^{ES}(\varphi)$ does not admit a minimum in the homotopy class of maps $\varphi:\T^2 \to \s^2$ of degree one.
\end{itemize}
\end{theorem}

\begin{proof} (i) Let $\xi:\R\to \R$ be a smooth function such that
\begin{enumerate}
\item $\xi(\rho)=0, \quad \forall \rho\in(-\infty, 1]$; \item $\xi(\rho)=1,
\quad \forall \rho\in[2,\infty)$; \item $\xi(\rho)\in(0,1), \quad \forall
\rho\in(1,2$);\item $\dot{\xi}(\rho)>0$ on $(1,2)$, so $\xi$ is
strictly increasing on $[1,2]$.
\end{enumerate}

Let $a>1$ and define the following function $\alpha_a:\R\to \R$:
\begin{equation}\label{def-alpha-a}
\alpha_a(\rho)=2\arctan(a\rho)+\xi(\rho)(\pi-2\arctan(a\rho))\,.
\end{equation}
We observe that $\alpha_a(\rho)$ is a smooth function and its derivative is
\begin{equation}\label{Cond-C-first-der}
\dot{\alpha_a}(\rho)=(1-\xi)\frac{2a}{1+a^2\rho^2}
+\dot{\xi}(\pi-2\arctan(a\rho))\,.
\end{equation}
Then $\dot{\alpha}_a>0$ on $(-\infty,2)$, so $\alpha_a$ is strictly
increasing on $(-\infty,2]$ and $\alpha_a([0,2])=[0,\pi]$.

We consider the $2$-dimensional flat torus $\T^2$ modelled, with the usual identifications, by
$$
\mathcal{Q}^2(3)=\{(x,y)\in\R^2 \ : \ \vert x\vert\leq 3, \quad \vert
y\vert \leq 3\},
$$
and define the map $\varphi_a:\T^2\to \s^2$ as follows:
\begin{enumerate}
\item $\varphi_a(0)=N$, where $N=(0,0,1)$ is the North pole; \item
${\varphi_a}_{\big{\vert}\T^2\backslash B^2(2)}=S$, the South pole, where
$$
B^2(R)=\{(x,y)\in\R^2 \ : \ \vert(x,y)\vert<R\}\,;
$$ \item in
the polar coordinates $(\vartheta,\rho)$ on $\R^2\backslash\{0\}$ and the
spherical coordinates $(\vartheta,\alpha)$ on $\s^2\backslash\{N,S\}$, the
map $\varphi_a$ is given by
$$
\varphi_a(\vartheta, \rho)=(\vartheta, \alpha_a(\rho)), \quad \rho\in(0,2),\vartheta\in\s^1\,.
$$
\end{enumerate}
The map $\varphi_a$ is well defined and smooth since the general regularity conditions
\[
\begin{aligned}
&\alpha_a(0)=0, \,\,\alpha_a^{(2k)}(0)=0 \,\,(k\geq1) \quad {\rm and} \,\,\,\,\alpha_a^{(2k+1)}(0)\in \R \,\,(k \geq0); \\
&\alpha_a(2)=\pi \quad {\rm and}\quad \alpha_a^{(k)}(2)=0 \,\,(k \geq1)
\end{aligned}
\]
are satisfied. We also note that all the maps $\varphi_a$ have degree $1$. Therefore, it is enough to show that
\begin{equation}\label{lim-a-infty}
\lim_{a \to +\infty} E^{ES}_4 (\varphi_a)=0 \,.
\end{equation}
Now, in order to compute $E^{ES}_4 (\varphi_a)$, we use Proposition~\ref{Propos-ES-4-energy-models} with $m=2$, $f(\rho)=\rho$ and $h(\alpha)=\sin \alpha$. We find
\begin{eqnarray}\label{E-4-cond-C}
E^{ES}_4 (\varphi_a)&=& \int_{B^2(2) \backslash B^2(1)} L_4^{ES}\left(\rho,\alpha_a(\rho),\dot{\alpha_a}(\rho),\ldots,\alpha_a^{(4)}(\rho)\right ) \,dV \nonumber
\\& =&2\pi \int_1^2  L_4^{ES}\left(\rho,\alpha_a(\rho),\dot{\alpha_a}(\rho),\ldots,\alpha_a^{(4)}(\rho)\right ) \,d\rho\,,
\end{eqnarray}
where the integral is just over $B^2(2) \backslash B^2(1)$ because $\varphi_a$ is harmonic on $B^2(1)$ and outside $B^2(2)$. The explicit expression for the Lagrangian  $L_4^{ES}\left(\rho,\alpha(\rho),\dot{\alpha}(\rho),\ldots,\alpha^{(4)}(\rho)\right )=L$ in \eqref{E-4-cond-C} is the following (we write it in an expanded form because this simplifies the remaining part of the analysis):
\begin{tiny}
\begin{eqnarray*}\label{Explicit-Lagrangian-cond-C}\nonumber
L&=&\frac{\sin ^2(\alpha ) \cos ^6(\alpha )}{2 \rho^7}-\frac{4 \sin
   ^2(\alpha ) \cos ^4(\alpha )}{\rho^7}+\frac{8 \sin ^2(\alpha )
   \cos ^2(\alpha )}{\rho^7}+\frac{2 \dot{\alpha} \sin (\alpha ) \cos
   ^5(\alpha )}{\rho^6}\\\nonumber &&-\frac{3 \dot{\alpha} \sin ^3(\alpha ) \cos
   ^3(\alpha )}{\rho^6}-\frac{7\dot{\alpha} \sin (\alpha ) \cos
   ^3(\alpha )}{\rho^6}+\frac{12 \dot{\alpha} \sin ^3(\alpha ) \cos
   (\alpha )}{\rho^6}-\frac{4 \dot{\alpha}\sin (\alpha ) \cos
   (\alpha )}{\rho^6} \\ \nonumber
   &&+\frac{\left(\dot{\alpha
   }\right)^2}{2 \rho^5}+\frac{9 \left(\dot{\alpha
   }\right)^2 \sin ^4(\alpha )}{2 \rho^5}-\frac{3 \left(\dot{\alpha
   }\right)^2 \sin ^2(\alpha )}{\rho^5}+\frac{2 \left(\dot{\alpha
   }\right)^2 \cos ^4(\alpha )}{\rho^5}+\frac{2 \left(\dot{\alpha
   }\right)^2 \cos ^2(\alpha )}{\rho^5}\\ \nonumber
   &&+\frac{4 \left(\dot{\alpha
   }\right)^2 \sin ^2(\alpha ) \cos ^4(\alpha
   )}{\rho^5}+\frac{\left(\dot{\alpha
   }\right)^2 \sin ^4(\alpha ) \cos
   ^2(\alpha )}{2 \rho^5}-\frac{22 \left(\dot{\alpha
   }\right)^2 \sin
   ^2(\alpha ) \cos ^2(\alpha )}{\rho^5}+\frac{4 \left(\dot{\alpha
   }\right)^3 \sin (\alpha ) \cos (\alpha )}{\rho^4}\\ \nonumber
   &&-\frac{2 \ddot{\alpha} \sin (\alpha ) \cos ^5(\alpha
   )}{\rho^5}+\frac{7 \ddot{\alpha} \sin (\alpha ) \cos ^3(\alpha
   )}{\rho^5}-\frac{4 \ddot{\alpha} \sin ^3(\alpha ) \cos (\alpha
   )}{\rho^5}
   +\frac{4 \ddot{\alpha} \sin (\alpha ) \cos (\alpha
   )}{\rho^5}\\&&+\frac{8 \left(\dot{\alpha}\right)^3 \sin (\alpha ) \cos
   ^3(\alpha )}{\rho^4}-\frac{13 \left(\dot{\alpha}\right)^3 \sin
   ^3(\alpha ) \cos (\alpha )}{\rho^4}+\frac{\left(\dot{\alpha}\right)^4 \sin ^2(\alpha )}{2 \rho^3}+\frac{8 \left(\dot{\alpha}\right)^4 \sin ^2(\alpha ) \cos ^2(\alpha )}{\rho^3}\\
   &&+\frac{ \ddot{\alpha} \sin ^3(\alpha ) \cos ^3(\alpha )}{\rho
   ^5}-\frac{\dot{\alpha} \ddot{\alpha}}{\rho ^4}-\frac{3 \dot{\alpha}
    \ddot{\alpha} \sin ^4(\alpha )}{\rho ^4}+\frac{4 \dot{\alpha}
    \ddot{\alpha} \sin ^2(\alpha )}{\rho ^4}-\frac{4 \dot{\alpha}
   \ddot{\alpha} \cos ^4(\alpha )}{\rho ^4}-\frac{4 \dot{\alpha}
   \ddot{\alpha} \cos ^2(\alpha )}{\rho ^4}\\&&+\frac{8 \dot{\alpha}
   \ddot{\alpha} \sin ^2(\alpha ) \cos ^2(\alpha )}{\rho
   ^4}-\frac{8 \left(\dot{\alpha}\right)^2 \ddot{\alpha} \sin (\alpha
   ) \cos ^3(\alpha )}{\rho ^3}-\frac{4 \left(\dot{\alpha}\right)^2 \ddot{\alpha} \sin (\alpha ) \cos (\alpha )}{\rho
   ^3}\\
   &&+\frac{\left(\ddot{\alpha}\right)^2}{2 \rho ^3}+\frac{\left(\ddot{\alpha}\right)^2 \sin ^4(\alpha )}{2 \rho ^3}-\frac{\left(\ddot{\alpha}\right)^2 \sin ^2(\alpha )}{\rho ^3}+\frac{2 \left(\ddot{\alpha}\right)^2 \cos ^4(\alpha )}{\rho ^3}+\frac{2 \left(\ddot{\alpha}\right)^2 \cos ^2(\alpha )}{\rho ^3}\\&&-\frac{2 \left(\ddot{\alpha}\right)^2 \sin ^2(\alpha ) \cos ^2(\alpha )}{\rho
   ^3}+\frac{3 \left(\dot{\alpha }\right)^2 \ddot{\alpha} \sin
   ^3(\alpha ) \cos (\alpha )}{\rho ^3}+\frac{\left(\dot{\alpha}\right)^3 \ddot{\alpha} \sin ^2(\alpha )}{\rho
   ^2}+\frac{\left(\dot{\alpha}\right)^2 \left(\ddot{\alpha}\right)^2
   \sin ^2(\alpha )}{2 \rho }\\
   &&+\frac{2 \alpha ^{(3)} \sin (\alpha ) \cos ^3(\alpha )}{\rho
   ^4}-\frac{8 \alpha ^{(3)} \sin (\alpha ) \cos (\alpha
   )}{\rho ^4}+\frac{2 \left(\alpha ^{(3)}\right)^2}{\rho
   }-\frac{2 \alpha ^{(3)}  \ddot{\alpha}}{\rho ^2}+\frac{2 \alpha
   ^{(3)}  \ddot{\alpha} \sin ^2(\alpha )}{\rho ^2}\\&&-\frac{4 \alpha
   ^{(3)} \ddot{\alpha} \cos ^2(\alpha )}{\rho ^2}+\frac{2 \alpha
   ^{(3)} \dot{\alpha}}{\rho ^3}-\frac{6 \alpha ^{(3)} \dot{\alpha}
   \sin ^2(\alpha )}{\rho ^3}+\frac{4 \alpha ^{(3)}\dot{\alpha}
   \cos ^2(\alpha )}{\rho ^3}+\frac{8 \alpha ^{(3)}
   \left(\dot{\alpha}\right)^2 \sin (\alpha ) \cos (\alpha )}{\rho
   ^2}\\&&
   +\frac{\alpha ^{(4)} \sin (\alpha ) \cos ^3(\alpha )}{\rho
   ^3}-\frac{4 \alpha ^{(4)} \sin (\alpha ) \cos (\alpha
   )}{\rho ^3}+\frac{1}{2} \left(\alpha ^{(4)}\right)^2 \rho
   +2 \alpha ^{(3)} \alpha ^{(4)}-\frac{\alpha ^{(4)} \ddot{\alpha}}{\rho }+\frac{\alpha ^{(4)}  \ddot{\alpha} \sin ^2(\alpha
   )}{\rho }\\&&-\frac{2 \alpha ^{(4)}  \ddot{\alpha} \cos ^2(\alpha
   )}{\rho }+\frac{\alpha ^{(4)}  \dot{\alpha}}{\rho ^2}-\frac{3
   \alpha ^{(4)} \dot{\alpha} \sin ^2(\alpha )}{\rho ^2}+\frac{2
   \alpha ^{(4)} \dot{\alpha} \cos ^2(\alpha )}{\rho ^2}+\frac{4
   \alpha ^{(4)} \left(\dot{\alpha}\right)^2 \sin (\alpha ) \cos
   (\alpha )}{\rho }\,.
\end{eqnarray*}
\end{tiny}
\noindent Direct inspection of the various terms in $L$, using the H$\ddot{{\rm o}}$lder inequality together with $1 \leq \rho \leq 2$, leads us to conclude that, in order to prove \eqref{lim-a-infty}, it suffices to show that
\begin{enumerate}\label{integrals-cond-C}
\item[{\rm (a)}]
$$ \int_1^2 \sin^2 \alpha_a(\rho)\, d\rho \,\rightarrow 0 \,\,\quad \quad {\rm as} \,\, a \rightarrow + \infty ;
$$
\item[{\rm (b)}]
$$ \int_1^2 (\alpha_a^{(i)})^2(\rho)\, d\rho \,\rightarrow 0 \,\,\quad \quad {\rm as} \,\, a \rightarrow + \infty \,\, (i=1,2,3,4) \,.
$$
\end{enumerate}
To prove (a) we use the definition \eqref{def-alpha-a} of $\alpha_a(\rho)$, the fact that $\alpha_a(\rho)$ is strictly increasing on $[1,2]$ and $\alpha_a(1)> \pi/2$. Then we deduce the following uniform estimate on $[1,2]$:
\[
|\sin \alpha_a(\rho)|=\sin \alpha_a(\rho) \leq \sin \alpha_a(1) =\frac{2a}{1+a^2}
\]
from which (a) follows immediately. As for (b), we start with the case of the first derivative $(i=1)$. Let us denote
$$
M_i={\rm Max}\left \{ |\xi^{(i)}(\rho) |\,\,: \,\,1 \leq \rho \leq2 \right \}\,, \quad \quad i=1,2,3,4\,.
$$
Inspection of \eqref{Cond-C-first-der} leads us to the following uniform estimate on $[1,2]$:
\begin{equation*}
|\dot{\alpha_a}(\rho)|\leq \frac{2a}{1+a^2}+M_1 (\pi-2 \arctan a)
\end{equation*}
and so the case (b), $i=1$, is proved. The cases $i=2,3,4$ are similar and, just to give an idea of the required analysis, we supply the details for the most difficult case, i.e., $i=4$. Indeed, a computation shows that
\begin{eqnarray}
\nonumber
\alpha_a^{(4)}(\rho)&=&-\frac{8 a \xi ^{(3)}(\rho )}{a^2 \rho ^2+1}+\frac{96 a^7 \rho
   ^3 \xi (\rho )}{\left(a^2 \rho ^2+1\right)^4}-\frac{96 a^7
   \rho ^3}{\left(a^2 \rho ^2+1\right)^4}-\frac{64 a^5 \rho ^2
   \dot{\xi }}{\left(a^2 \rho ^2+1\right)^3}\\\nonumber&&-\frac{48 a^5
   \rho  \xi (\rho )}{\left(a^2 \rho ^2+1\right)^3}+\frac{48
   a^5 \rho }{\left(a^2 \rho ^2+1\right)^3}+\frac{24 a^3 \rho
    \ddot{\xi }(\rho )}{\left(a^2 \rho ^2+1\right)^2}+\frac{16 a^3
   \dot{\xi }(\rho )}{\left(a^2 \rho ^2+1\right)^2}\\&&+ \xi ^{(4)}(\rho )(\pi-2 \arctan (a \rho )) \,.
\nonumber
\end{eqnarray}
From this it is very easy to deduce for $|\alpha_a^{(4)}(\rho)|$ the necessary uniform upper estimate on $[1,2]$, which depends on $M_i$, $i=1,\ldots,4$, and tends to $0$ as $a$ increases to $+\infty$,   and so the proof of part (i) of Theorem\link\ref{theor-cond-C} is completed.

(ii) It is well-known that there exists no harmonic map $\varphi:\T^2 \to \s^2$ of degree one. Therefore, as a consequence of \eqref{Inf=0}, it suffices to show that $E_4^{ES}(\varphi)=0$ occurs only if $\varphi$ is harmonic. Indeed, if $E^{ES}_4(\varphi)=0$, then $E_4(\varphi)=0$ and, as $M=\T^2$ is compact, we get $\nabla^{\varphi}\tau(\varphi)=0$. But then it follows from formula \eqref{divZ} for $\operatorname{div}Z$ that $\varphi$ is harmonic.
\end{proof}
\begin{remark} The conclusion \eqref{Inf=0} in Proposition\link\ref{theor-cond-C} was obtained by Lemaire (see \cite{Lemaire}) in the case of the bienergy. Our proof is an extension of his method. We point out that the same conclusion holds for $E_3(\varphi)$ and $E_4(\varphi)$ as well, and the proof in these cases is the same. Actually, it is not difficult to extend this result to the cases that $r\geq5$, but we omit the details in this direction because no new idea is involved.
\end{remark}
\section{Second variation}\label{second-variation}
In this section we turn our attention to the study of the second variation. Very little is known in this context and, for the reasons explained in Subsection\link\ref{Sub-sec-curves}, we shall focus on the case that $\dim M=1$, so that the $E_r^{ES}(\varphi)=E_r(\varphi)$ and we can use the general theory developed by Maeta and Wang (\cite{Maeta3, Wang2}). Our goal is to compute index and nullity of some significant examples. Now, we prepare the ground to state our main results. To this purpose, first we explain some basic facts about the operator $I_r(V)$ and the definition of index and nullity. More specifically, let $\varphi:M\to N$ be an $r$-harmonic map between two Riemannian manifolds $(M,g)$ and $(N,h)$. We consider a two-parameters smooth variation $\left \{ \varphi_{t,s} \right \}$ $(-\varepsilon <t,s < \varepsilon,\,\varphi_{0,0}=\varphi)$ and denote by $V,W$ their associated vector fields:
\begin{align*}
& V(x)= \left . \frac{d}{dt} \varphi_{t,0}\right |_{t=0} \quad \in T_{\varphi(x)}N \\
& W(x)= \left . \frac{d}{ds} \varphi_{0,s}\right |_{s=0} \quad \in T_{\varphi(x)}N\,.
\end{align*}
Note that $V$ and $W$ are sections of $\varphi^{-1}TN$. The {\it Hessian} of the energy functional $E_r$ at its critical point $\varphi$ is defined by
\begin{equation}\label{Hessian-definition}
H(E_r)_\varphi (V,W)= \left . \frac{\partial^2}{\partial t \partial s}  E_r (\varphi_{t,s}) \right |_{(t,s)=(0,0)}\, .
\end{equation}
The following theorem was obtained by Jiang \cite{Jiang} for $r=2$ (see also \cite{Onic}), Wang \cite{Wang2} for $r=3$ and Maeta \cite{Maeta3} for $r \geq 4$.
\begin{theorem}\label{Hessian-Theorem} Let $\varphi:M\to N$ be an $r$-harmonic map between two Riemannian manifolds $(M,g)$ and $(N,h)$. Then the Hessian of the energy functional $E_r$ at a critical point $\varphi$ is given by
\begin{equation}\label{Operator-Ir}
H(E_r)_\varphi (V,W)= \int_M \langle I_r(V),W \rangle \, dV \,\,,
\end{equation}
where $I_r \,:C\left(\varphi^{-1} \, TN\right) \to C\left(\varphi^{-1} \, TN\right)$ is a semilinear elliptic operator of order $2r$.
\end{theorem}
The general expression for $I_r(V)$ involves iterated applications of the classical Jacobi operator and it is very long and complicated: it can be found in the work of Maeta \cite{Maeta3}. Our approach shall be based on a direct computation using general two-parameters variations $\left \{ \varphi_{t,s} \right \}$ and the definition \eqref{Hessian-definition}. Here it is important to recall from the general theory that, when $M$ is compact, the spectrum
\begin{equation*}
\lambda_1 < \lambda_2 < \ldots < \lambda_i < \ldots
\end{equation*}
of the operator $I_r(V)$ is \textit{discrete} and tends to $+\infty$ as $i$ tends to $+ \infty$. We denote by $\mathcal{V}_i$ the eigenspace associated to the eigenvalue $\lambda_i$. Then we define
\begin{equation*}
{\rm Index}(\varphi) = \sum_{\lambda_i <0} \dim(\mathcal{V}_i)\,.
\end{equation*}
The nullity of $\varphi$ is defined as
\begin{equation*}
{\rm Nullity}(\varphi) =  \dim \left \{ V \in C\left(\varphi^{-1} \, TN\right) \, : \, I_r(V)=0\right \} \,.
\end{equation*}
In the case that $r=2$, Index and Nullity of certain proper biharmonic maps have been computed (see references \cite {LO, TAMS, MOR-last}, for instance). By contrast, when $r \geq3$ very little is known about the index and the nullity of the $r$-harmonic maps which can be found in the literature. Now we are in the right position to describe the examples that we shall investigate in our context of second variation: each case  contains a short description of the $r$-harmonic maps under consideration and the corresponding result concerning their index and nullity.

\begin{example}\label{parallel-r-harmonic-circles}
Let $r \geq2$ and consider a map $\varphi_{r,k}\,: \s^1 \to \s^2\hookrightarrow  \R^3$ defined by
\begin{equation}\label{r-harmonic-examples}
 \gamma \mapsto \, \left ( \sin (\alpha^*) \cos(k\gamma),\,\sin (\alpha^* )\sin(k\gamma), \,\cos  (\alpha^* )\right ) \,, \quad 0 \leq \gamma \leq 2\pi\,,
\end{equation}
where $\alpha^*=\arcsin \left ( 1 \slash \sqrt{r} \right )$ and $k \in \n^*$ is a fixed positive integer. We know (see \cite{Maeta2, Mont-Ratto4}) that $\varphi_{r,k}$ is a proper $r$-harmonic map. Both the notions of $r$-harmonicity and that of index and nullity of an $r$-harmonic map are invariant under homothetic changes of the metric of either the domain or the codomain. Therefore, in this example, we have assumed for simplicity that the domain is the unit circle. In particular, the radius of the domain which would ensure the condition of isometric immersion for $k=1$ is $R=1/\sqrt r$, but any choice of $R$ would not affect the conclusions of our next result:
\begin{theorem}\label{Index-theorem-r-harmonic-circles} Assume that $2 \leq r \leq 4$ and let $\varphi_{r,k} : \s^1 \to \s^2$ be a proper $r$-harmonic map as in \eqref{r-harmonic-examples}. Then
\begin{equation}\label{*}
\begin{aligned}
& {\rm Nullity}(\varphi_{r,k})=3 \\
& {\rm Index}(\varphi_{r,k})=1+2(k-1) \,.
\end{aligned}
\end{equation}
\end{theorem}
The proof of this theorem is a case by case analysis for $r=2,3,4$ and it shall be carried out in Subsection\link\ref{otherproofs}.
\begin{remark}\label{remark-inex-depend-k} Theorem\link\ref{Index-theorem-r-harmonic-circles} was known in the case that $r=2$: it was proved for $k=1$ in \cite{LO}, where the index and the nullity of $i:\s^m(1/\sqrt 2) \hookrightarrow \s^m$ was computed.
The case $r=2$, $k \neq 1$ was proved in \cite{MOR-last}. In this work we shall give a different proof which is based on a direct method which is useful to prepare the ground for the study of the cases $r \geq 3$.
\end{remark}
\textbf{Conjecture:} we conjecture that the conclusion of Theorem\link\ref{Index-theorem-r-harmonic-circles} is true for all $r\geq2$. This belief shall be substantiated and discussed in more detail in Remark\link\ref{conjecture}.
\end{example}
\begin{example}\label{parallel-rotation-surface}
In the context of rotation surfaces, we know the following existence result (see \cite{Mont-Ratto5}).
Let $S_{{\rm par}} \subset \R^3$ be the paraboloid of revolution defined by
\begin{equation*}
z=\left (x^2+y^2 \right )\,.
\end{equation*}
Let $r \geq 3$ and consider the map $\varphi_{r}\,: \s^1 \to S_{{\rm par}}\hookrightarrow  \R^3$ defined by
\begin{equation}\label{r-harmonic-examples-par}
 \gamma \mapsto \, \left ( \alpha^* \, \sin\gamma,\,\alpha^* \, \cos\gamma, \, (\alpha^* )^2\right ) \,, \quad 0 \leq \gamma \leq 2\pi\,,
\end{equation}
where
\begin{equation*}\label{condizione-parallelo-paraboloide}
\alpha^{*}= \frac{1}{2\sqrt{r-2}} .
\end{equation*}
Then $\varphi_r$ is a proper $r$-harmonic map. These maps are interesting because we know that $S_{{\rm par}}$ does not admit neither closed geodesics nor proper biharmonic curves (see \cite{Mont-Ratto6}). Here we prove the following result:
\begin{theorem}\label{Theorem-index-par}
Assume that $r=3$ or $r=4$. Let $\varphi_{r}\,: \s^1 \to S_{{\rm par}}$ be the $r$-harmonic map defined in \eqref{r-harmonic-examples-par}. Then
\begin{align}
& {\rm Nullity}(\varphi_{r})=1\nonumber \\
& {\rm Index}(\varphi_{r})=1 \,\,.\nonumber
\end{align}
\end{theorem}
\end{example}

\subsection{Proof of the results on the second variation}\label{otherproofs}
In this subsection we shall prove Theorem\link\ref{Index-theorem-r-harmonic-circles} and Theorem\link\ref{Theorem-index-par}. We shall follow an approach which could prove useful in other related examples. It is based on the direct computation of \eqref{Hessian-definition} using a general two-parameters variation. In particular, our method does not require the use of the general expression for the operator $I_2$ and its complicated generalizations to the case $r\geq3$. As a preliminary step, since it shall be necessary to carry out covariant derivatives in local coordinates, we report here a calculation of Christoffel symbols which we shall use in several circumstances.
\begin{lemma}\label{Christoffel-lemma}
Let $(N,g)= \left (\s^1 \times (a,b), f^2(\alpha)dw^2+h^2(\alpha)d\alpha^2 \right )$. Then, if we consider $w$ and $\alpha$ as the coordinate number $1$ and $2$ respectively, the Christoffel symbols of $(N,g)$ are:
\begin{equation*}\label{Christoffel-general-expression}
\Gamma_{11}^1=\Gamma_{22}^1=\Gamma_{12}^2=\Gamma_{21}^2=0 \,,\quad
\Gamma_{12}^1=\Gamma_{21}^1=\frac{f'(\alpha)}{f(\alpha)}\,,\quad
\Gamma_{11}^2=-\,\frac{f(\alpha)f'(\alpha)}{h^2(\alpha)}\,,\quad
\Gamma_{22}^2=\frac{h'(\alpha)}{h(\alpha)}\,.
\end{equation*}
\end{lemma}
In particular, if $N=\s^2$, i.e., $f(\alpha)=\sin \alpha$ and $h(\alpha) \equiv 1$, the Christoffel symbols become:
\begin{equation}\label{Christoffel-sfera}
\Gamma_{11}^1=\Gamma_{22}^1=\Gamma_{12}^2=\Gamma_{21}^2=\Gamma_{22}^2=0 \,,\quad
\Gamma_{12}^1=\Gamma_{21}^1=\cot \alpha\,,\quad
\Gamma_{11}^2=-\,\frac{1}{2}\,\sin (2\alpha)\,.
\end{equation}In order to prove Theorem\link\ref{Index-theorem-r-harmonic-circles} we have to separate three cases: $r=2,3$ and $4$.

\subsection{Proof of Theorem\link\ref{Index-theorem-r-harmonic-circles}, Case ${r=2}$}
To simplify notation, in this case we shall write $\varphi_k$ instead of $\varphi_{2,k}$. We describe the 2-sphere $\s^2$ by means of polar coordinates:
\begin{equation}\label{2sfera}
\s^2 = \left ( \s^1 \times [0,\pi], \, \sin^2 \alpha \, \, dw^2+ d\alpha^2 \right ) \,\, , \qquad 0 \leq w \leq 2 \pi \,, \,\, 0 \leq \alpha\leq \pi \,\, .
\end{equation}
We consider a general map $\varphi : \s^1 \to \s^2$ and write it as
\begin{equation}\label{map-general-s1-to-s2}
     \gamma  \mapsto \left ( w(\gamma),  \alpha (\gamma) \right ) \,,
\end{equation}
where $w,\alpha$ are the coordinates introduced in \eqref{2sfera}.
We recall that the local coordinates expression for the second fundamental form of a map $\varphi:M \to N$ is
\begin{equation*}
\nabla d \varphi = \left ( \nabla d \varphi \right )^\gamma_{ij} \, dx^i dx^j \otimes\frac{\partial}{\partial y^\gamma} \,\, ,
\end{equation*}
where
\begin{equation}\label{seconda-forma-coordinate-locali-bis}
\left ( \nabla d \varphi \right )^\gamma_{ij} = \varphi^\gamma _{ij} - {}^M \Gamma^k_{ij} \varphi^\gamma _{k} + {}^N \Gamma^\gamma_{\beta \delta} \varphi^\beta _{i} \varphi^\delta _{j} \,\, .
\end{equation}
Now, since $\tau (\varphi)$ is the trace of the second fundamental form, its    description in local coordinates is
\begin{equation}\label{tau-coordinate-locali}
( \tau (\varphi) )^\gamma  = g_M^{ij} \, \left ( \nabla d \varphi \right )^\gamma_{ij}\,\, .
\end{equation}
Using \eqref{Christoffel-sfera} in \eqref{tau-coordinate-locali} we find that
\begin{equation*}
    \tau(\varphi)= \tau_w \,\frac{\partial}{\partial w}+
    \tau_\alpha \,\frac{\partial}{\partial \alpha}\,\, ,
\end{equation*}
where
\begin{equation}\label{componentitau-general-s1-to-s2}
\begin{aligned}
    \tau_w&=w''(\gamma)+2 \cot (\alpha(\gamma))\,
    w'(\gamma)\,\alpha'(\gamma) \\ 
    \tau_\alpha&= \alpha''(\gamma)-\,\frac{1}{2} \sin(2\alpha(\gamma))\,
    w'^2(\gamma)\,. 
\end{aligned}
\end{equation}
The $2$-energy functional becomes
\begin{equation}\label{2-energy-general-s1-to-s2}
    E_2(\varphi)= \frac{1}{2}\,\int_0^{2\pi} \left [\sin^2 \alpha \,\left (\tau_w \right)^2 +\left (\tau_\alpha \right)^2
    \right ] \,d\gamma\,. 
\end{equation}
A general two-parameters variation $\varphi_{t,s}$ of $\varphi_k$ can be written as follows:
\begin{eqnarray}\label{biharmonic-2var-general-s1-to-s2}
    \varphi_{t,s}:\s^1&\to& \s^2 \\ \nonumber
     \gamma  &\mapsto& \left ( k\gamma+tV_1(\gamma)+sW_1(\gamma),  \alpha^*+tV_2(\gamma)+sW_2(\gamma) \right )\,, \nonumber
\end{eqnarray}
where $\alpha^*=\arcsin (1 / \sqrt r)$ and $V_j,W_j \in {C}^\infty\left ( \s^1\right ),\,j=1,2$ (in this first case, $r=2$ and so
$\alpha^*= \pi/4$). We point out that $\varphi_{0,0}=\varphi_k$
and
\begin{equation*}
\begin{aligned}
  \left . \frac{d}{dt}  \varphi_{t,0} \right |_{t=0}&= V_1
  \,\frac{\partial}{\partial w}+
    V_2 \,\frac{\partial}{\partial \alpha}=V \in C\left (\varphi_k^{-1}T\s^2 \right ) \\
    \left . \frac{d}{ds}  \varphi_{0,s} \right |_{s=0}&= W_1
  \,\frac{\partial}{\partial w}+
    W_2 \,\frac{\partial}{\partial \alpha}=W \in C\left (\varphi_k^{-1}T\s^2 \right ) \,.
    \end{aligned}
\end{equation*}
We know from \eqref{Hessian-definition} and \eqref{Operator-Ir} that
\begin{equation}\label{Operator-I2-s1-to-s2}
\left . \frac{\partial^2}{\partial t \partial s}  E_2 (\varphi_{t,s})\right |_{(0,0)}= \int_0^{2\pi} \langle I_2(V),W \rangle \, d\gamma  \,.
\end{equation}
Now we have to insert the explicit expression \eqref{biharmonic-2var-general-s1-to-s2} into \eqref{2-energy-general-s1-to-s2} and compute the left side of \eqref{Operator-I2-s1-to-s2}. We obtain
\begin{eqnarray}\label{Hessiano-0-0-biharmonic}
\left . \frac{\partial^2}{\partial t \partial s}  E_2 (\varphi_{t,s})\right |_{(0,0)}&=&\frac{1}{2} \int_0^{2\pi}\left[-2 k^4 V_2(\gamma) W_2(\gamma)+k^2 V_1'(\gamma) W_1'(\gamma)+\left(2 k V_2'(\gamma)+V_1''(\gamma)\right) \right. \nonumber\\ 
&&\left .
\left(2 k W_2'(\gamma)+W_1''(\gamma)\right) +2
   \left(V_2''(\gamma)-k V_1'(\gamma)\right) \left(W_2''(\gamma)-k W_1'(\gamma)\right)\right]d\gamma \,.
   \end{eqnarray}
It is convenient to write
\begin{equation}\label{decomposizione-I2}
\begin{aligned}
I_2 \left (V_1 \,\frac{\partial}{\partial w} \right )& =I_w \,\frac{\partial}{\partial w}+ I_\alpha \,\frac{\partial}{\partial \alpha} \\ 
 I_2 \left (V_2 \,\frac{\partial}{\partial \alpha} \right )& =J_w \,\frac{\partial}{\partial w}+ J_\alpha \,\frac{\partial}{\partial \alpha} \,.
\end{aligned}
\end{equation}

In order to compute the expressions of $I_w,\,I_\alpha,\,
J_w,\,J_\alpha$ we study \eqref{Hessiano-0-0-biharmonic} by separating the following cases.

\textit{Case }$V_2=W_2=0$: comparing \eqref{Operator-I2-s1-to-s2} and \eqref{Hessiano-0-0-biharmonic} we find
\begin{equation}\label{Iw}
\int_0^{2\pi} \langle I_2 \left (V_1
  \,\frac{\partial}{\partial w} \right ),W_1 \,\frac{\partial}{\partial w} \rangle \, d\gamma  =\int_0^{2\pi}
\left [\frac{1}{2} \left(3 k^2 V_1'(\gamma) W_1'(\gamma)+V_1''(\gamma) W_1''(\gamma)\right) \right ] d\gamma\,.
\end{equation}
Now, integrating by parts on the right side of \eqref{Iw}, we eliminate all the derivatives of $W_1$ and we obtain:
\[
\int_0^{2\pi} I_w\,W_1\,\langle
  \,\frac{\partial}{\partial w} , \,\frac{\partial}{\partial w} \rangle \, d\gamma  =\int_0^{2\pi}
\left [\frac{1}{2} \left(V_1^{(4)}(\gamma)-3 k^2 V_1''(\gamma)\right) W_1\right ] d\gamma \,.
\]
Next, using
\[
\langle
  \,\frac{\partial}{\partial w} , \,\frac{\partial}{\partial w} \rangle= \sin^2 \alpha^* = \frac{1}{2} \,,
\]
it follows that
\begin{equation}\label{Iw-biharmonic}
I_w=V_1^{(4)}(\gamma)-3 k^2 V_1''(\gamma) \,.
\end{equation}
Similarly, in the cases $V_2=W_1=0$, $V_1=W_2=0$ and $V_1=W_1=0$ respectively, we obtain:
\begin{equation}\label{Ia,Jw,Ja biharmonic}
\begin{aligned}
I_\alpha&=-2 k V_1^{(3)}(\gamma)  \\ 
J_w&= 4 k V_2^{(3)}(\gamma) \\ 
J_\alpha&=V_2^{(4)}(\gamma)-2 k^2 V_2''(\gamma)-k^4 V_2(\gamma) \,.
\end{aligned}
\end{equation}
Now we decompose ${C}\left( \varphi_k^{-1}T\s^2\right )$  in a similar fashion to \cite{MOR-last}. We recall that the spectrum of $\Delta$ on $\s^1$ is $\{ m^2 \}_{m\in\n}$ and, for
$m \in \n$, we define
\begin{equation*}
\label{sottospazi-S-m}
 S^{m^2}=\left \{ V_1\,\frac{\partial}{\partial w}\,\,:\,\, \Delta V_1= m^2 V_1 \right \} \oplus \left \{ V_2\,\frac{\partial}{\partial \alpha}\,\,:\,\, \Delta V_2= m^2 V_2 \right \} \,.
\end{equation*}
Then we know that $S^{m^2} \perp S^{m'^2}$ if $m \neq m'$, and $\oplus_{m=0}^{+\infty}\, S^{m^2}$ is dense in
${C}\left( \varphi_k^{-1}T\s^2\right )$. Moreover, $I_2$ preserves all these subspaces. Next, we observe that $\dim \left (S^0 \right )=2$ and that an orthonormal basis of $S^0$ is $\left\{u_1,u_2 \right\}$, where ($r=2$ here)
\[
u_1= \sqrt{\frac{r}{2 \pi}}\,\frac{\partial}{\partial w},\,\quad
u_2= \frac{1}{\sqrt {2\pi}}\,\frac{\partial}{\partial \alpha}\,.
\]
Now, using \eqref{decomposizione-I2}, \eqref{Iw-biharmonic} and \eqref{Ia,Jw,Ja biharmonic}, it is immediate to construct the $(2 \times 2)$-matrix which describes the restriction of $I_2$ to $S^0$:
\begin{equation*}
\left (\begin{array}{rr}
0&0 \\
0&-k^4
\end{array}
\right )
\end{equation*}
from which we deduce immediately that the contribution of $S^0$ to
the index and the nullity of $\varphi_k$ is $+1$ for both. Next, we study the subspaces $S^{m^2}$, $m \geq1$. First, we observe that $\dim \left (S^{m^2} \right )=4$ and that an orthonormal basis of $S^{m^2}$ is $\left\{u_1,u_2,u_3,u_4 \right\}$, where ($r=2$ here)
{\small \begin{equation}\label{on-bases}
u_1= \sqrt {\frac{r}{ \pi}}\,\cos (m\gamma) \,\frac{\partial}{\partial w},\;
u_2= \sqrt {\frac{r}{ \pi}}\,\sin (m\gamma) \,\frac{\partial}{\partial w},\;
u_3= \frac{1}{\sqrt {\pi}}\,\cos (m\gamma) \,\frac{\partial}{\partial \alpha},\;
u_4= \frac{1}{\sqrt {\pi}}\,\sin (m\gamma) \,\frac{\partial}{\partial \alpha}\,.
\end{equation}
}
Now, using \eqref{decomposizione-I2}, \eqref{Iw-biharmonic} and \eqref{Ia,Jw,Ja biharmonic}, we construct the $(4 \times 4)$-matrices which describe the restriction of $I_2$ to $S^{m^2}$. The outcome is:
\begin{equation*}
\left(
\begin{array}{cccc}
 m^2 \left(3 k^2+m^2\right) & 0 & 0 & -2 \sqrt{2} k m^3 \\
 0 & m^2 \left(3 k^2+m^2\right) & 2 \sqrt{2} k m^3 & 0 \\
 0 & 2 \sqrt{2} k m^3 & m^4+2 m^2 k^2-k^4 & 0 \\
 -2 \sqrt{2} k m^3 & 0 & 0 & m^4+2 m^2 k^2-k^4
\end{array}
\right) \,,
\end{equation*}
whose eigenvalues are
\begin{equation*}
\lambda_m^{\pm}=\frac{1}{2} \left(-k^4+2
   m^4+5 k^2 m^2\pm \sqrt{k^8+2 k^6 m^2+k^4 m^4+32 k^2m^6}\right)
\end{equation*}
with multiplicity equal to $2$. Now, all the $\lambda_m^+$'s are clearly positive and so they do not contribute neither to the index nor to the nullity of $\varphi_k$. As for the $\lambda_m^-$'s, we can apply Lemma 2.15 of \cite{MOR-last}: it follows that the contribution to the nullity of $\varphi_k$ is $+2$ (coming from $\lambda_{k}^-$), while the contribution to the index is $+2(k-1)$, arising from $1 \leq m \leq  (k-1)$, so that the proof of the case $r=2$ is completed.

\subsection{Proof of Theorem\link\ref{Index-theorem-r-harmonic-circles}, Case $r=3$}
We shall again simplify the notation by writing $\varphi_k$ instead of $\varphi_{3,k}$. The proof proceeds by following precisely the same steps that we carried out in the case $r=2$, but we apply the method to the $3$-energy instead of the $2$-energy \eqref{2-energy-general-s1-to-s2}. So, first, let us compute the explicit expression for the $3$-energy functional in our context:
\begin{lemma}\label{3-energy-explicit}
Let $\varphi : \s^1 \to \s^2$ be a general map as in \eqref{map-general-s1-to-s2}. Then its $3$-energy is:
\begin{equation}\label{3-energy-general-s1-to-s2}
    E_3(\varphi)= \frac{1}{2}\,\int_0^{2\pi} \left [\sin^2 \alpha \,\left ((d\tau)_w\right)^2 +\left ((d\tau)_\alpha \right)^2
    \right ] \,d\gamma \,\,,
\end{equation}
where $\tau_w,\,\tau_\alpha$ are defined in \eqref{componentitau-general-s1-to-s2} and
\begin{equation}\label{comp-dtau}
\begin{aligned}
(d\tau)_w&= \tau_w'+(\tau_\alpha\, w'+
    \tau_w\, \alpha')\cot \alpha \\ 
(d\tau)_\alpha&= \tau_\alpha'-\frac{1}{2} \,\tau_w\,\sin(2 \alpha)\,w' \,.
\end{aligned}
\end{equation}
\end{lemma}
\begin{proof}[Proof of Lemma\link\ref{3-energy-explicit}]
We can write
\begin{equation*}
d\tau=(d\tau)_w \, d\gamma \otimes \frac{\partial}{\partial w}
+(d\tau)_\alpha \, d\gamma \otimes \frac{\partial}{\partial \alpha} \,.
\end{equation*}
The proof of the lemma amounts to showing that $(d\tau)_w$ and $(d\tau)_\alpha$ are given by the expressions in \eqref{comp-dtau}. To this purpose we compute using Lemma \ref{Christoffel-lemma}:
{\small\begin{eqnarray*}\label{cov-derivata-3-harmonic}
\nonumber
\nabla_{\partial / \partial \gamma}^{\varphi} \tau(\varphi)&=&
\nabla_{\partial / \partial \gamma}^{\varphi} \left [\tau_w \,
\frac{\partial}{\partial w}+\tau_\alpha \,
\frac{\partial}{\partial \alpha}\right ]\\ \nonumber
&=& \tau_w' \,
\frac{\partial}{\partial w}+\tau_\alpha' \,
\frac{\partial}{\partial \alpha}+\tau_w\,\nabla_{d\varphi (\partial / \partial \gamma)}^{\s^2}\frac{\partial}{\partial w}+
\tau_\alpha\,\nabla_{d\varphi (\partial / \partial \gamma)}^{\s^2}\frac{\partial}{\partial \alpha}\\ \nonumber
&=&\tau_w' \,
\frac{\partial}{\partial w}+\tau_\alpha' \,
\frac{\partial}{\partial \alpha}+\tau_w\,\nabla_{w' (\partial / \partial w)+\alpha' (\partial / \partial \alpha)}^{\s^2}\frac{\partial}{\partial w}+
\tau_\alpha\,\nabla_{w' (\partial / \partial w)+\alpha' (\partial / \partial \alpha)}^{\s^2}\frac{\partial}{\partial \alpha}\\
 \nonumber
&=&\tau_w' \,
\frac{\partial}{\partial w}+\tau_\alpha' \,
\frac{\partial}{\partial \alpha}+\tau_w\,w'\nabla_{\partial / \partial w}^{\s^2}\frac{\partial}{\partial w}+
\tau_w\,\alpha'\nabla_{\partial / \partial \alpha}^{\s^2}\frac{\partial}{\partial w}
\\ 
&&+\tau_\alpha\,w'\nabla_{\partial / \partial w}^{\s^2}\frac{\partial}{\partial \alpha}+
\tau_\alpha\,\alpha'\nabla_{\partial / \partial \alpha}^{\s^2}\frac{\partial}{\partial \alpha} \\\nonumber
&=&\tau_w' \,
\frac{\partial}{\partial w}+\tau_\alpha' \,
\frac{\partial}{\partial \alpha}+\tau_w\,w'\,\left (\Gamma_{11}^1\frac{\partial}{\partial w} +\Gamma_{11}^2\frac{\partial}{\partial \alpha}  \right )+
\tau_w\,\alpha'\,\left (\Gamma_{21}^1\frac{\partial}{\partial w} +\Gamma_{21}^2\frac{\partial}{\partial \alpha}  \right )
\\ \nonumber
&&+\tau_\alpha\,w'\,\left (\Gamma_{12}^1\frac{\partial}{\partial w} +\Gamma_{12}^2\frac{\partial}{\partial \alpha}  \right )+ 0 \\\nonumber
&=& \left [ \tau_w'+(\tau_\alpha\, w'+
    \tau_w\, \alpha')\cot \alpha \right ] \frac{\partial}{\partial w}
+\left [\tau_\alpha'-\frac{1}{2} \,\tau_w\,\sin(2 \alpha)\,w' \right ] \frac{\partial}{\partial \alpha}
\end{eqnarray*}
}
from which the conclusion of the proof follows immediately.
\end{proof}
From now on, the proof follows exactly the scheme which we have detailed in the case $r=2$: simply, we have to replace \eqref{2-energy-general-s1-to-s2} with \eqref{3-energy-general-s1-to-s2}. Therefore, since the calculations are long but straightforward, here we limit ourselves to report the relevant results. First, the version of \eqref{Hessiano-0-0-biharmonic} in this case is:
\begin{eqnarray*}\label{Hessiano-0-0-triharmonic}
\nonumber
\left . \frac{\partial^2}{\partial t \partial s}  E_3 (\varphi_{t,s})\right |_{(0,0)}&=&\frac{1}{2} \int_0^{2\pi}\frac{1}{9} \Big[-8  k^6 V_2 W_2+33 k^4 V_2' W_2'+9 \sqrt{2} k^3 V_1'' W_2'+9 \sqrt{2} k^3 V_2' W_1''  \\ \nonumber
&& +6 k^4 V_2 W_2''-2\sqrt{2}
   k^3 V_2 W_1{}^{(3)} -2 \sqrt{2} k^3 V_1{}^{(3)} W_2+20 k^4 V_1' W_1'\\ \nonumber
 &&-24 \sqrt{2} k^3 V_1' W_2''+6 k^2 V_1' W_1{}^{(3)}+6k^4 V_2'' W_2-24 \sqrt{2} k^3 V_2'' W_1'\\
 &&+54 k^2 V_2'' W_2''+9 \sqrt{2} k V_2'' W_1{}^{(3)}-6 k^2 V_1{}^{(3)} W_1'-15 k^2 V_2{}^{(3)} W_2'\\ \nonumber
  && -15 k^2
   W_2{}^{(3)} V_2'+18 k^2 V_1'' W_1''+9 \sqrt{2} k V_1{}^{(3)} W_2''-9 \sqrt{2} k V_2{}^{(3)} W_1''\\ \nonumber
   &&-9 \sqrt{2} k W_2{}^{(3)} V_1'' +3V_1{}^{(3)} W_1{}^{(3)}+9 V_2{}^{(3)} W_2{}^{(3)}\Big] \,d\gamma \nonumber\,.
\end{eqnarray*}


Next, separating cases as we did for $r=2$, we find that the operator $I_3$ is described by:
\begin{eqnarray*}\label{Ia,Jw,Ja triharmonic}
I_w&=&-\frac{1}{3}  \left(20 k^4 V_1''(\gamma)-30 k^2 V_1^{(4)}(\gamma)+3 V_1^{(6)}(\gamma)\right)  \\ \nonumber
I_\alpha&=&\frac{\sqrt 2}{9}  k  \left(18 V_1^{(5)}(\gamma)-35 k^2 V_1^{(3)}(\gamma)\right)  \\ \nonumber
J_w&=& -\frac{\sqrt 2}{3}  k \left(-35 k^2 V_2^{(3)}(\gamma)+18 V_2^{(5)}(\gamma) \right) \\ \nonumber
J_\alpha&=&\frac{1}{9} \left(-8 k^6 V_2(\gamma)-21 k^4 V_2''(\gamma) +84 k^2 V_2^{(4)}(\gamma)-9 V_2^{(6)}(\gamma) \right) \,.  \nonumber
\end{eqnarray*}
From these expressions it is easy to deduce that the contribution of $S^0$ to
the index and the nullity of $\varphi_k$ is $+1$ for both. Next, we study the $4$-dimensional subspaces $S^{m^2}$, $m \geq1$ and find that the relevant matrices (with respect to the orthonormal basis \eqref{on-bases} with $r=3$) are
\begin{equation}\label{matrici-I3-Sm}
\left(
\begin{array}{cccc}
 A_{m,k}& 0 & 0 & -C_{m,k} \\
 0 & A_{m,k} & C_{m,k} & 0 \\
 0 &C_{m,k} & B_{m,k}& 0 \\
 -C_{m,k} & 0 & 0 &B_{m,k}
\end{array}
\right)\,\,,
\end{equation}
where we have set:
\begin{eqnarray*}
A_{m,k}&=& \frac{1}{3} m^2 \left(20 k^4+30 m^2 k^2+3 m^4\right) \\ \nonumber
B_{m,k}&=& \frac{1}{9} \left(-8 k^6+21 m^2 k^4+84 m^4 k^2+9 m^6\right)  \\ \nonumber
C_{m,k}&=&  \frac{1}{3} \sqrt{\frac{2}{3}} k m^3 \left(35 k^2+18 m^2\right)  \,.\nonumber
\end{eqnarray*}
The eigenvalues of the matrices \eqref{matrici-I3-Sm} are
\begin{eqnarray*}
\nonumber
\lambda_m^{\pm}&=&\frac{1}{18} \Big(-8 k^6+81 k^4 m^2+174 k^2 m^4+18 m^6 \\ 
&&\pm \sqrt{64 k^{12}+624 k^{10} m^2+1617 k^8 m^4+29868 k^6 m^6+30276 k^4 m^8+7776 k^2 m^{10}} \Big)
\end{eqnarray*}
with multiplicity equal to $2$. Now, all the $\lambda_m^+$'s are clearly positive and so they do not contribute neither to the index nor to the nullity of $\varphi_k$. As for the $\lambda_m^-$'s, we carry out the relevant analysis in the following technical lemma:
\begin{lemma}\label{lemma-tecnico2}
If $1\leq m \leq (k-1)$, then $\lambda_{m}^- <0$. If $m=k$, then $\lambda_{m}^- =0$. If $m>k$, then $\lambda_{m}^- >0$.
\end{lemma}
\begin{proof} [Proof of Lemma~\ref{lemma-tecnico2}]
The eigenvalue $\lambda_{m}^- $ has the same sign of the expression
\begin{equation}\label{2}
\begin{aligned}
&-8 k^6+81 k^4 m^2+174 k^2 m^4+18 m^6+ \\ 
&- \sqrt{64 k^{12}+624 k^{10} m^2+1617 k^8 m^4+29868 k^6 m^6+30276 k^4 m^8+7776 k^2 m^{10}}\,.
\end{aligned}
\end{equation}
If we set $m=ck$ into \eqref{2} and divide by $k^6$ we obtain:
\[
18 c^6+174 c^4+81 c^2-8-\sqrt{7776 c^{10}+30276 c^8+29868 c^6+1617 c^4+624 c^2+64} \,.
\]
Now we set $y=c^2$ and rewrite the previous expression as
\begin{equation*}\label{p(y)}
18 y^3+174 y^2+81 y-8-\sqrt{7776 y^{5}+30276 y^4+29868 y^3+1617 y^2+624 y+64} =p(y)\,.
\end{equation*}
Since $p(1)=0$, it is sufficient to show that $p'(y)>0$ if $y>0$. We find
\begin{equation*}\label{p'(y)}
p'(y)=54 y^2+348 y+81-\frac{3 \left(6480 y^4+20184 y^3+14934 y^2+539 y+104\right)}{\sqrt{7776 y^5+30276 y^4+29868 y^3+1617 y^2+624 y+64}}\,.
\end{equation*}
Computing we find that $p'(y)>0$ if and only if the following polynomial has no positive root:
\begin{eqnarray*}\label{q(y)}
q(y)&=&22674816 y^9+2624400 y^8-119544336 y^7+88606548 y^6+210300192 y^5
\\ \nonumber
&&-99086004 y^4+187882848 y^3+23527152 y^2+6693120 y+322560 \,. \nonumber
\end{eqnarray*}
Now we rewrite $q(y)=q_A(y)+q_B(y)$ where:
\begin{eqnarray}\label{q(y)-bis}
\nonumber
q_A(y)&=&y^3 \left(20000000 y^2-99086004 y+187882848\right)
\\ \nonumber
q_B(y)&=&4 \big(5668704 y^9+656100 y^8-29886084 y^7+22151637 y^6\\ \nonumber
&&+47575048 y^5+5881788 y^2+1673280 y+80640\big)\\ \nonumber
&=&4 \big(y^5(5668704 y^4-29886084 y^2+47575048)\\\nonumber
&&+656100 y^8+22151637 y^6+5881788 y^2+1673280 y+80640\big)\,. \nonumber
\end{eqnarray}
Now it is easy to check that both $q_A(y)$ and $q_B(y)$ are positive for $y>0$ and so the proof of the lemma is completed.
\end{proof}
Now we apply Lemma\link\ref{lemma-tecnico2} and conclude that the contribution of the $\lambda_m^-$'s to the nullity of $\varphi_k$ is $+2$ (coming from $\lambda_{k}^-$), while the contribution to the index is $+2(k-1)$, arising again from $1 \leq m \leq  (k-1)$, so that the proof of the case $r=3$ is ended.
\subsection{Proof of Theorem\link\ref{Index-theorem-r-harmonic-circles}, Case $r=4$}
The scheme of the proof is as in the previous cases. We write $\varphi_k$ for $\varphi_{4,k}$ and, instead of Lemma\link\ref{3-energy-explicit}, here we use:
\begin{lemma}\label{4-energy-explicit}
Let $\varphi : \s^1 \to \s^2$ be a general map as in \eqref{map-general-s1-to-s2}. Then its $4$-energy is:
\begin{eqnarray*}
    E_4(\varphi)= \frac{1}{2}\,\int_0^{2\pi} &\Big[&\sin^2 \alpha \,\Big ((d\tau)_w'+\big((d\tau)_\alpha\, w'+
    (d\tau)_w\, \alpha'\big )\cot \alpha \Big )^2 \nonumber\\
    &&+\Big ((d\tau)_\alpha'-\frac{1}{2} \,(d\tau)_w\,\sin(2 \alpha)\,w' \Big )^2
    \Big] \,d\gamma \,\,,
\end{eqnarray*}
where $(d\tau)_w,\,(d\tau)_\alpha$ are defined in \eqref{comp-dtau}.
\end{lemma}
\begin{proof}[Proof of Lemma\link\ref{4-energy-explicit}] 
We write
\begin{equation*}
-d^*d\tau=\left( \nabla d\tau\right)_{11}^1 \,\frac{\partial}{\partial w}+ \left( \nabla d\tau\right)_{11}^2 \,\frac{\partial}{\partial \alpha}\,.
\end{equation*}
Then the proof reduces to showing that
\begin{equation}\label{comp-d*dtau}
\begin{aligned}
\left( \nabla d\tau\right)_{11}^1&= (d\tau)_w'+\Big((d\tau)_\alpha\, w'+
    (d\tau)_w\, \alpha' \Big )\cot \alpha \\ 
\left( \nabla d\tau\right)_{11}^2&= (d\tau)_\alpha'-\frac{1}{2} \,(d\tau)_w\,\sin(2 \alpha)\,w' \,.
\end{aligned}
\end{equation}
This assertion can be verified using \eqref{seconda-forma-coordinate-locali-bis}. Indeed, we have (for simplicity, we write here only the nonvanishing Christoffel symbols):
\begin{equation*}
\begin{aligned}
\left( \nabla d\tau\right)_{11}^1&= (d\tau)_w'+\Gamma_{12}^1
(d\tau)_\alpha\, w'+\Gamma_{21}^1
(d\tau)_w\, \alpha'\\ 
\left( \nabla d\tau\right)_{11}^2&= (d\tau)_\alpha'+\Gamma_{11}^2
(d\tau)_w\, w' \,\,.
\end{aligned}
\end{equation*}
Now, using again Lemma\link\ref{Christoffel-lemma}, it is easy to obtain \eqref{comp-d*dtau} and so the proof of the lemma is ended.
\end{proof}
The next relevant computation for the $4$-energy case yields:
\begin{eqnarray*}\label{Ia,Jw,Ja 4-harmonic}
I_w&=&-\frac{189}{16} k^6 V_1''(\gamma)+\frac{315}{8} k^4 V_1^{(4)}(\gamma)-21 k^2 V_1^{(6)}(\gamma)+ V_1^{(8)}(\gamma) \nonumber\\ \nonumber
I_\alpha&=&-\frac{1}{64} \sqrt{3} k \left(399 k^4 V_1^{(3)}(\gamma)-644 k^2 V_1^{(5)}(\gamma)+128 V_1^{(7)}(\gamma)\right)  \\ \nonumber
J_w&=& \frac{1}{16} \sqrt{3} k \left(399 k^4 V_2^{(3)}(\gamma)-644 k^2 V_2^{(5)}(\gamma)+128 V_2^{(7)}(\gamma)\right)\\ \nonumber
J_\alpha&=&-\frac{27}{32} k^8 V_2(\gamma)-\frac{171}{64} k^6 V_2''(\gamma)+\frac{505}{16} k^4 V_2^{(4)}(\gamma)-\frac{41}{2} k^2 V_2^{(6)}(\gamma)+V_2^{(8)}(\gamma) \,. \nonumber
\end{eqnarray*}
From these expressions it is easy to deduce that the contribution of $S^0$ to
the index and the nullity of $\varphi_k$ is $+1$ for both. Next, we study the $4$-dimensional subspaces $S^{m^2}$, $m \geq1$ and find that the relevant matrices (with respect to the orthonormal basis \eqref{on-bases} with $r=4$) are
\begin{equation}\label{matrici-I4-Sm}
\left(
\begin{array}{cccc}
 A_{m,k}& 0 & 0 & -C_{m,k} \\
 0 & A_{m,k} & C_{m,k} & 0 \\
 0 &C_{m,k} & B_{m,k}& 0 \\
 -C_{m,k} & 0 & 0 &B_{m,k}
\end{array}
\right)\,\,,
\end{equation}
where now we have set:
\begin{eqnarray*}
A_{m,k}&=& \frac{1}{16} m^2 \left(189 k^6+630 m^2 k^4+336 m^4 k^2+16 m^6\right)\\ \nonumber
B_{m,k}&=&\frac{1}{64} \left(-54 k^8+171 m^2 k^6+2020 m^4 k^4+1312 m^6 k^2+64 m^8\right)  \\ \nonumber
C_{m,k}&=&  \frac{1}{32} \sqrt{3} k m^3 \left(399 k^4+644 m^2 k^2+128 m^4\right)\,. \nonumber
\end{eqnarray*}
The eigenvalues of the matrices \eqref{matrici-I4-Sm} are:
\begin{eqnarray}\label{autovalori-I4-Sm}
\nonumber
\lambda_m^{\pm}&=&\frac{1}{128} \Big(-54 k^8 + 927 k^6 m^2 + 4540 k^4 m^4 + 2656 k^2 m^6 + 128 m^8 \\ \nonumber &&\pm
    \sqrt{[2916 k^{16} + 63180 k^{14} m^2 + 396225 k^{12} m^4 +
    8230104 k^{10} m^6 } \\ \nonumber
   && \overline{+ 24955216 k^8 m^8 + 24842240 k^6 m^{10}+
    7914496 k^4 m^{12} + 786432 k^2 m^{14}]}\,\Big ) 
\end{eqnarray}
with multiplicity equal to $2$. Now the conclusion of the proof can be obtained by an argument very similar to Lemma\link\ref{lemma-tecnico2} and so we omit further details.
\begin{remark}\label{conjecture} If we put together Lemmata\link\ref{3-energy-explicit} and \ref{4-energy-explicit} we are able to obtain a recursive expression for the $r$-energy of general maps $\varphi\,:\,\s^1 \to \s^2$. More precisely, for $r \geq 3$ we define
\begin{align*}
\mathcal{T}_{2,w}&= \tau_w \,\,\quad {\rm and} \qquad \mathcal{T}_{2,\alpha}= \tau_\alpha \,\,.\\ \nonumber
\mathcal{T}_{r,w}&= \left (\mathcal{T}_{(r-1),w}\right)'+
\Big(\mathcal{T}_{(r-1),\alpha}\, w'+
    \mathcal{T}_{(r-1),w}\, \alpha' \Big )\cot \alpha \\ \nonumber
    \mathcal{T}_{r,\alpha}&= \left (\mathcal{T}_{(r-1),\alpha}\right)'-\,\frac{1}{2}\,
\mathcal{T}_{(r-1),w}\,\sin(2 \alpha)\, w' \nonumber
\end{align*}
where $\tau_w$ and $\tau_\alpha$ are given in \eqref{componentitau-general-s1-to-s2}.
Then, for $r \geq3$, the expression for the $r$-energy of a general map $\varphi\,:\,\s^1 \to \s^2$ as in \eqref{map-general-s1-to-s2} is given by
\begin{equation}\label{r-energy-general-s1-to-s2}
    E_r(\varphi)= \frac{1}{2}\,\int_0^{2\pi} \left [\sin^2 \alpha \,\left ( \mathcal{T}_{r,w}\right)^2+\left( \mathcal{T}_{r,\alpha} \right )^2\, \right ] \,d\gamma \,\,.
\end{equation}
Using \eqref{r-energy-general-s1-to-s2} and performing suitable computer aided computations it is possible to apply the methods of Theorem\link\ref{Index-theorem-r-harmonic-circles} to the cases $r \geq 5$. In particular, we were able to verify that
 the conclusion \eqref{*} of Theorem\link\ref{Index-theorem-r-harmonic-circles} is true when $r=5,6$. Since the calculations involved are huge we do not include them here. One of the difficulties to tackle the general case is the fact that $\alpha^*$ depends on $r$ and so it is difficult to apply an induction argument.
\end{remark}
Next, we provide the proof of Theorem\link\ref{Theorem-index-par}. Also in this case it is convenient to separate the cases $r=3$ and $r=4$. The schemes of the proofs are precisely as those of Theorem\link\ref{Index-theorem-r-harmonic-circles}. Therefore, here we just report the main steps without inserting all the details.

\subsection{Proof of Theorem\link\ref{Theorem-index-par}, Case $r=3$}
We observe that $S_{par}$ is a manifold as in Lemma\link\ref{Christoffel-lemma}, with $f(\alpha)=\alpha$ and $h(\alpha)=\sqrt{1+4\alpha^2}$. Therefore, we can describe a general map $\varphi:\s^1 \to S_{par}$ with respect to local coordinates as in \eqref{map-general-s1-to-s2}. Now, the first step is to compute the $3$-energy: this can be done as in Lemma\link\ref{3-energy-explicit} using Lemma\link\ref{Christoffel-lemma}. The result is:
\begin{equation}\label{3-energy-par-general-s1-to-s2}
    E_3(\varphi)= \frac{1}{2}\,\int_0^{2\pi} \left [ \alpha^2 \,\left ((d\tau)_w\right)^2 +(1+4\alpha^2)\,\left ((d\tau)_\alpha \right)^2
    \right ] \,d\gamma \,,
\end{equation}
where
\begin{eqnarray*}\label{comp-taud-tau-par}
\tau_w&=&w''+\frac{2}{\alpha} \,w'\,\alpha'\\ \nonumber
\tau_\alpha&=&\alpha''-\,\frac{\alpha}{1+4\alpha^2}\,w'^2 +\frac{4\alpha}{1+4\alpha^2}\,\alpha'^2\\ \nonumber
(d\tau)_w&=& \tau_w'+\,\frac{1}{\alpha}\,(\tau_\alpha\, w'+
    \tau_w\, \alpha') \\ \nonumber
(d\tau)_\alpha&=& \tau_\alpha'-\frac{\alpha}{1+4\alpha^2} \,\tau_w\,w'
+\frac{4\alpha}{1+4\alpha^2}\,\tau_\alpha\,\alpha' \,.\nonumber
\end{eqnarray*}
Then the direct calculation of \eqref{Operator-I2-s1-to-s2} leads us to say that the operator $I_3$ is now described by:
\begin{equation}\label{Ia,Jw,Ja 3-harmonic-par}
\begin{aligned}
I_w&=\frac{1}{4} \left(-15  V_1''(\gamma )+30  V_1{}^{(4)}(\gamma )-4 V_1^{(6)}(\gamma )\right)\\ 
I_\alpha&=\frac{1}{4} \left(3  V_1^{(5)}(\gamma )-\frac{35}{8}  V_1^{(3)}(\gamma )\right) \\ 
J_w&=\frac{35}{2}  V_2{}^{(3)}(\gamma )-12  V_2{}^{(5)}(\gamma )\\ 
J_\alpha&=\frac{1}{4} \left(-\frac{1}{4}
   V_2(\gamma )-\frac{21}{8} V_2''(\gamma )+14  V_2^{(4)}(\gamma )-2 V_2^{(6)}(\gamma )\right) \,.
\end{aligned}
\end{equation}
Now, direct inspection shows that the contribution of $S^0$ to
the index and the nullity is $+1$ for both. Next, we study the $4$-dimensional subspaces $S^{m^2}$, $m \geq1$. An orthonormal basis is
\begin{equation}\label{base-par}
\begin{array}{lclcl}
u_1&=& \dfrac{2}{ \sqrt{\pi}}\,\cos (m\gamma) \,\dfrac{\partial}{\partial w},\quad
u_2&=& \dfrac{2}{\sqrt{ \pi}}\,\sin (m\gamma) \,\dfrac{\partial}{\partial w},\\
u_3&=& \dfrac{1}{2\sqrt {\pi}}\,\cos (m\gamma) \,\dfrac{\partial}{\partial \alpha}\,\quad
u_4&=& \dfrac{1}{2\sqrt {\pi}}\,\sin (m\gamma) \,\dfrac{\partial}{\partial \alpha}\,.
\end{array}
\end{equation}
We find that, with respect to this basis, the relevant matrices are
\begin{equation*}
\left(
\begin{array}{cccc}
 A_{m}& 0 & 0 & -C_{m} \\
 0 & A_{m} & C_{m} & 0 \\
 0 &C_{m} & B_{m}& 0 \\
 -C_{m} & 0 & 0 &B_{m}
\end{array}
\right)\,,
\end{equation*}
where now we have set:
\begin{equation}\label{coeff-matrici-par-I3}
\begin{aligned}
A_{m}&=\frac{1}{4} m^2 \left(15 +30 m^2 +4 m^4\right)\\ 
B_{m}&= \frac{1}{32} \left(-2 +21 m^2 +112 m^4 +16 m^6\right)  \\ 
C_{m}&= \frac{1}{8}  m^3 \left(35 +24 m^2\right) \,.
\end{aligned}
\end{equation}
The eigenvalues are:
\begin{eqnarray}\label{autovalori-I3-par}
\nonumber \lambda_m^{\pm}&=&\frac{1}{64} \Big(-2 +141 m^2+352  m^4+48 m^6 \\
&&\pm
  \sqrt{4 +396  m^2+10313 m^4+103808  m^6}\\
  &&\nonumber\overline{+127072  m^8+40960  m^{10}+256 m^{12}}\Big)
\end{eqnarray}
each of them with multiplicity equal to $2$. Now a routine analysis shows that all the $\lambda_m^{\pm}$'s are positive and so the proof is ended.
\subsection{Proof of Theorem\link\ref{Theorem-index-par}, Case $r=4$}
To carry out this proof, we perform computations as above and find that we have to replace \eqref{3-energy-par-general-s1-to-s2}, \eqref{Ia,Jw,Ja 3-harmonic-par}, \eqref{base-par}, \eqref{coeff-matrici-par-I3}, \eqref{autovalori-I3-par} with \eqref{4-energy-par-general-s1-to-s2}, \eqref{Ia,Jw,Ja 4-harmonic-par}, \eqref{base-par-4}, \eqref{coeff-matrici-par-I4} and \eqref{eigenvalues-I4-par} respectively:
\begin{equation}\label{4-energy-par-general-s1-to-s2}
    E_4(\varphi)= \frac{1}{2}\,\int_0^{2\pi} \left [ \alpha^2 \,\left ((\nabla d\tau)_{11}^1\right)^2 +(1+4\alpha^2)\,\left ((\nabla d\tau)_{11}^2 \right)^2
    \right ] \,d\gamma \,\,,
\end{equation}
where
\begin{eqnarray*}
\left( \nabla d\tau\right)_{11}^1&=& (d\tau)_w'+\Big((d\tau)_\alpha\, w'+
    (d\tau)_w\, \alpha' \Big )\,\frac{1}{\alpha} \\ \nonumber
\left( \nabla d\tau\right)_{11}^2&=&(d\tau_\alpha)'-\frac{\alpha}{1+4\alpha^2} \,(d\tau_w)\,w'
+\frac{4\alpha}{1+4\alpha^2}\,(d\tau_\alpha)\,\alpha' \,.\nonumber
\end{eqnarray*}

\begin{equation}\label{Ia,Jw,Ja 4-harmonic-par}
\begin{array}{lcl}
I_w&=&\frac{1}{27} \left(-224 V_1''(\gamma )+840 V_1^{(4)}(\gamma )-504 V_1^{(6)}(\gamma )+27
   V_1^{(8)}(\gamma )\right) \\
I_\alpha&=&-\frac{1}{81} \sqrt{2} \left(266 V_1^{(3)}(\gamma )-483 V_1^{(5)}(\gamma )+108 V_1^{(7)}(\gamma )\right)  \\
J_w&=& 8 \left(\frac{133}{27} \sqrt{2}  V_2^{(3)}(\gamma )-\frac{161 V_2^{(5)}(\gamma )}{9 \sqrt{2}}+2 \sqrt{2}
   V_2^{(7)}(\gamma )\right)\\
J_\alpha&=&\frac{2}{3} \left(-\frac{32}{81} V_2(\gamma )-\frac{76}{27} V_2''(\gamma )+\frac{1010}{27} V_2^{(4)}(\gamma )-\frac{82}{3} V_2^{(6)}(\gamma )+\frac{3}{2} V_2^{(8)}(\gamma )\right)
\end{array}
\end{equation}

\begin{equation}\label{base-par-4}
\begin{array}{lclcl}
u_1&=& \dfrac{2\sqrt 2}{ \sqrt{\pi}}\,\cos (m\gamma) \,\dfrac{\partial}{\partial w},\quad
u_2&=& \dfrac{2 \sqrt 2}{\sqrt{ \pi}}\,\sin (m\gamma) \,\dfrac{\partial}{\partial w},\\
u_3&=& \dfrac{\sqrt 2}{\sqrt {3\pi}}\,\cos(m\gamma) \,\dfrac{\partial}{\partial \alpha},\quad
u_4&=& \dfrac{\sqrt 2}{\sqrt {3\pi}}\,\sin (m\gamma) \,\dfrac{\partial}{\partial \alpha}\,.
\end{array}
\end{equation}

\begin{eqnarray}\label{coeff-matrici-par-I4}\nonumber
A_{m}&=& \frac{1}{27} m^2 \left(224 +840 m^2 +504 m^4 +27 m^6\right)\\
B_{m}&=&  \frac{1}{243} \left(-64 +456 m^2 +6060 m^4 +4428 m^6 +243 m^8\right) \\ \nonumber
C_{m}&=&  \frac{2}{27} \sqrt{\frac{2}{3}}  m^3 \left(266 +483 m^2 +108 m^4\right)\nonumber
\end{eqnarray}
and
\begin{eqnarray}\label{eigenvalues-I4-par}
\nonumber\lambda_m^{\pm}&=&(1/243)\Big [-32 +1236  m^2+6810 m^4+4482 m^6+243 m^8 \\
&&\pm 2 \sqrt{256 +12480  m^2+164100  m^4+4114188  m^6+14037309  m^8}\\ \nonumber
&& \overline{+15720480  m^{10}+5634441
   m^{12}+629856  m^{14}} \Big] \,. \nonumber
\end{eqnarray}
\begin{remark}\label{remark-paraboloide} If we consider the composition of the map $\varphi_r$ in \eqref{r-harmonic-examples-par} with the $k$-fold rotation $e^{{\rm i}\gamma} \mapsto e^{{\rm i}k\gamma}$, then we still have an $r$-harmonic map whose index and nullity can be studied with the methods of this paper. However, since that would not add anything new in terms of methods, we have decided to omit further details in this direction.
\end{remark}

\end{document}